\documentclass[11pt,a4]{article}

\usepackage{amsmath,amsthm,amssymb}
\usepackage{geometry,subfigure,hyperref,epsfig,epstopdf,bbm}

\newtheorem{theorem}{Theorem}[section]
\newtheorem{lemma}[theorem]{Lemma}

\newtheorem{remark}[theorem]{Remark}
\newtheorem*{defn}{Definition}
\numberwithin{equation}{section}
\newtheorem{exmp}[theorem]{Example}

\newcommand{\R}{\mathbb{R}}
\newcommand{\bxi}{\boldsymbol{\xi}}
\newcommand{\indep}{\rotatebox[origin=c]{90}{$\models$}}
\newcommand{\E}{\mathbb{E}}
\newcommand{\Prm}{\mathbb{P}}
\usepackage{enumerate}
\usepackage[shortlabels]{enumitem}

\title{Dynamical Polynomial Chaos Expansions and Long Time Evolution of Differential Equations with Random Forcing}

\author{H. Cagan Ozen\thanks{Department of Applied Physics \& Applied Mathematics, Columbia University, New York, NY 10027, USA (\href{mailto:hco2104@columbia.edu}{hco2104@columbia.edu}, \href{gb2030@columbia.edu}{gb2030@columbia.edu}). }  \and Guillaume Bal\footnotemark[1] }

\date{}

\begin{document}
\maketitle

\begin{abstract}
Polynomial chaos expansions (PCE) allow us to propagate uncertainties in the coefficients of differential equations to the statistics of their solutions. Their main advantage is that they replace stochastic equations by systems of deterministic equations. Their main challenge is that the computational cost becomes prohibitive when the dimension of the parameters modeling the stochasticity is even moderately large. We propose a generalization of the PCE framework that allows us to keep this dimension as small as possible in favorable situations. For instance, in the setting of stochastic differential equations (SDEs) with Markov random forcing, we expect the future evolution to depend on the present solution and the future stochastic variables. We present a restart procedure that precisely allows PCE to depend only on that information. The computational difficulty then becomes the construction of orthogonal polynomials for dynamically evolving measures. We present theoretical results of convergence for our Dynamical generalized Polynomial Chaos (DgPC) method. Numerical simulations for linear and nonlinear SDEs show that it adequately captures the long-time behavior of their solutions as well as their invariant measures when the latter exist. 
\newline
\noindent
Keywords: Polynomial chaos, stochastic differential equations, uncertainty quantification,
Markov processes
\end{abstract}

\pagestyle{myheadings}
\thispagestyle{plain}
\markboth{ H.C. Ozen and G. Bal}{DgPC}

\section{Introduction}
\label{sec:Introduction}

\textit{Polynomial chaos} (PC) is a computational spectral method that is used to propagate and quantify uncertainties arising in the modeling of physical systems. The method originated from the works of Wiener~\cite{W38} and Cameron and Martin~\cite{CM47} on the decomposition of functionals in a basis of Hermite polynomials of Gaussian random variables. Ghanem and Spanos~\cite{GS91} used such an expansion to solve stochastic equations with random data. Xiu and Karniadakis~\cite{XK02,Xiu_thesis} introduced \textit{generalized polynomial chaos} (gPC) involving non-Gaussian random parameters.  Extensions to arbitrary probability measures followed in the works~\cite{SG04,WK05}; see also \cite{EMSU12,SG04} for recent theoretical developments and a convergence result for gPC expansions.

\textit{Polynomial chaos expansions} (PCE) provide an explicit expression of quantities of interests as functionals of the underlying uncertain parameters and in some situations allow us to perform uncertainty quantifications at a considerably lower computational cost than Monte Carlo methods. However, they suffer from the curse of dimensionality and thus typically work efficiently for systems involving low dimensional uncertainties. A related major drawback appears in the long-term integration of evolution equations.  In the presence of random forcing in time, the number of stochastic variables in the system increases linearly with time and thus quickly becomes overwhelming. Moreover, standard PCE utilize  orthogonal polynomials of the initial distribution and as time evolves, the dynamics deviate from the initial data substantially (e.g., due to nonlinearities)
and the solution may become poorly represented in the initial basis. This led the authors in~\cite{BM13} to legitimately question the usefulness of (standard) PCE to address long-term evolution properties of solutions such as intermittent instabilities. This paper aims to address the aforementioned drawbacks.

We propose a PC-based method that constructs evolving chaos expansions based on polynomials of projections of the time dependent solution and the random forcing. More precisely, chaos expansions at each iteration are constructed based on the knowledge of the moments of underlying distributions at a given time. In the setting of dynamics satisfying a Markov property, we crucially exploit this feature to introduce projections of the solution at prescribed time steps that allow us to "forget" about the past and as a consequence keep the dimension of the random variables fixed and independent of time. Our chaos basis is adapted to the evolving dynamics with the following consequences: (i) our expansion retains its optimality for long times; and (ii) the curse of dimensionality is mitigated. Notably, we  establish, with appropriate modifications, theoretical convergence analysis which sheds light on our numerical findings. Further, asymptotic analysis for computational complexity will also be discussed. Inspired by examples of~\cite{BM13,GHM10}, we will apply our method to a nonlinear coupled system of stochastic differential equations (SDEs) and its variants.  

Other iterative methodologies have already been proposed in~\cite{GSVK10, HS14, AGPRH12,AGPRH14} in various contexts of random differential equations (RDEs). The main novelty of our approach is that it allows us to compute long-term solutions of evolution equations with complex stochastic forcing, which here will be a Brownian motion but could be generalized to more complex models. White noise--driven evolution equations have important roles in modeling the small scale effects and certain uncertainties in many applications such as turbulence, filtering, mathematical finance,  and stochastic control~\cite{OB03,BM13,HLRZ06,GHM10}.

The plan of the paper is as follows. Section \ref{sec:PC} introduces background material and establishes the necessary notation for PCE. Our methodology, called Dynamical generalized Polynomial Chaos (DgPC), is described in detail in section \ref{sec:method}. Numerical experiments comparing DgPC to Hermite PC and Monte Carlo simulations are presented in section \ref{sec:numeric}. Some conclusions are offered in section \ref{sec:conclu}.


\section{Polynomial Chaos}  \label{sec:PC}

We briefly introduce the original PC framework.  The notation we follow can be found in~\cite{HLRZ06,Luo,BM13}. 

Given a probability space $(\Omega, \mathcal{F}, \Prm)$ where $\Omega$ is a general sample space, $\mathcal{F}$ is a $\sigma$-algebra of subsets and $\Prm$ is a probability measure, we consider $L^2(\Omega, \mathcal{F}, \Prm)$, the space of real-valued random variables with finite second moments.  Let  $\boldsymbol{\xi}= (\xi_1,\xi_2,\ldots)$ be a countable collection of independent and identically distributed (i.i.d) standard Gaussian random variables belonging to the probability space, and $ \mathcal{F}= \sigma(\boldsymbol{\xi})$. Then, we define the Wick polynomials
$
T_{\alpha}(\bxi) := \prod_{i=1}^{\infty} H_{\alpha_i}(\xi_i),
$
where $\alpha$ belongs to set of multi-indices with a finite number of nonzero components 
$
\mathcal{J} = \{ \alpha=(\alpha_1,\alpha_2,\ldots) \, | \, \alpha_j \in \mathbb{N}_0, \, |\alpha|=\sum_{i=1}^{\infty} \alpha_i < \infty \},
$
$H_{n}$ is the $n$th order normalized one-dimensional Hermite polynomial and $\mathbb{N}_0:=\mathbb{N} \cup \{0\}$. Note that the Wick polynomials are orthogonal to each other with respect to the measure induced by $\bxi$. 

The Cameron and Martin theorem~\cite{CM47} establishes that the Wick polynomials form a complete orthonormal basis in $L^2(\Omega, \mathcal{F},\Prm)$. This means that any functional $u(\cdot,\boldsymbol{\xi}) \in L^2$, can be expanded as
\begin{align} \label{eq:PCE}
u(\cdot,\boldsymbol{\xi}) = \sum_{\alpha \in \mathcal{J}} u_{\alpha}(\cdot) T_{\alpha}(\boldsymbol{\xi}), \quad  u_{\alpha}(\cdot) = \E[u(\cdot,\boldsymbol{\xi}) T_{\alpha}(\boldsymbol{\xi})],
\end{align}
and the sum converges in $L^2$. We define $\E$ as expectation with respect to $\Prm$. The expansion \eqref{eq:PCE} is called polynomial chaos expansion (PCE). We will also use the term Hermite PCE to emphasize that the expansion utilizes Gaussian random variables.    

One of the notable features of \eqref{eq:PCE} is that it separates the randomness in $u$  such that the coefficients $u_{\alpha}$ are deterministic and all statistical information is contained in the coefficients. In particular, the first two moments are given by
$
\E[u] = u_0$ and  $\E[u^2] = \sum_{\alpha \in \mathcal{J}} |u_{\alpha}|^2.
$
Higher order moments may then be computed using either triple products of Wick polynomials~\cite{Xiu_book,Xiu_thesis,LeMK10} or the Hermite product formula~\cite{MR04,Luo}; see section \ref{sec:moments}.

In numerical computations, the doubly infinite expansion \eqref{eq:PCE} is truncated so that it becomes a finite expansion
\begin{align} \label{eq:simp_trunc}
u \approx u_{K,N}(\cdot,\xi_1,\ldots,\xi_K) := \sum_{|\alpha| \leq N} u_{\alpha}(\cdot) \prod_{i=1}^K H_{\alpha_i}({\xi_i}),
\end{align}
where we simply used polynomials up to degree $N$ in the variables $(\xi_1,\xi_2\ldots,\xi_K)$. Throughout the paper, we use the graded lexicographic ordering for multi-indices; see \cite[Table 5.2]{Xiu_book}. 

 In the context of white noise--driven SDEs, the random variables $\xi_i$ can be obtained by the projection $
 \xi_i = \int_0^t m_i(s)\, dW(s),
 $
 where $W(s)$ is a Brownian motion on $[0,t]$ and $m_i$ is a complete orthonormal system in $L^2[0,t]$. Moreover, $\bxi=(\xi_i)_i$ is comprised of i.i.d. standard Gaussian random variables, and the expansion
 \begin{align} \label{eq:expansion_W} \sum_{i=1}^{\infty} \, \xi_{i} \, \int_0^s m_i(\tau) d\tau
 \end{align}
  converges in $L^2$ for all $s \leq t$ and 
 has the error bound 
 \begin{align} \label{eq:exp_W_err}
 \E \left[ W(s) - \sum_{i=1}^K  \xi_i \, \int_0^s m_i(\tau) d\tau \right]^2  = O(t K^{-1}),
 \end{align}
 provided the $m_i$'s are trigonometric polynomials; see \eqref{eq:cosines} and~\cite{Luo,HLRZ06}.  Here, Brownian motion $\{W(s), \, 0 \leq s \leq t \}$ is projected onto $L^2[0,t]$ for a fixed time $t>0$ so that the corresponding Wick polynomials $T_{\alpha}(\bxi)$ depend implicitly on $t$ . 

The simple truncation \eqref{eq:simp_trunc} leads to 
$
K+N \choose{K}$ terms in the approximation. Thus, the computational cost increases rapidly with high dimensionality, which in turn decreases the efficiency of PCE. This is the  ``curse of dimensionality''. Another related major problem of PC is that expansions may converge slowly and even fail to converge for long time evolutions~\cite{WK05, WK06,HLRZ06,BM13,MNGK04}. These problems led to numerous extensions of the Hermite PCE, which we now briefly discuss. 

The paper ~\cite{XK02} proposed a method called generalized polynomial chaos (gPC), in which the above random variables $\xi$ have specific, non-Gaussian, distributions in the Askey family~\cite{XK02,Xiu_thesis}. Further generalizations to arbitrary probability measures beyond the Askey family were proposed in~\cite{WK05}, where the probability space is decomposed into multiple elements and chaos expansions are employed in each subelement. The approach taken in~\cite{GSVK10} is to use a restart procedure, where the chaos expansion is restarted at different time-steps in order to mitigate the long-term integration issues of chaos expansions. Another generalization toward arbitrary distributions is presented in~\cite{ON12}, using only moment information of the involved distribution with a data-driven approach. 

\section{Description of the proposed method}
\label{sec:method}
Our key idea is to adapt the PCE to the dynamics of the system. Consider for concreteness the  following SDE:
\begin{align}  \label{eq:sde}
d u(s) = L(u) \, ds + \sigma \, dW(s), \quad u(0)=u_0, \quad s \in [0,t],
\end{align}
where $L(u)$ is a general function of $u$ (and possibly other deterministic or stochastic parameters), $W(s)$ is a Brownian motion, $\sigma>0$ is a constant, and $u_0$ is an initial condition with a prescribed probability density. We assume that a solution exists and is unique on $[0,t]$.

As the system evolves, a fixed polynomial chaos basis adapted to $u_0$ and $dW$ may not be optimal to represent the solution $u(t)$ for long times. Moreover, the dimension of the representation of $dW$ increases linearly with time $t$, which renders the PCE method computationally intractable even for moderate values of $t$. We will therefore introduce an increasing sequence of restart times $0<t_j<t_{j+1}<t$ and construct a new PCE basis at each $t_j$ based on the solution $u(t_j)$ and all additional random variables that need to be accounted for.  A very similar methodology with $\sigma=0$ was considered earlier in~\cite{GSVK10,HS14}.  When random forcing is present and satisfies the Markov property as in the above example, the restarting strategy allows us to "forget" past random variables that are no longer necessary and focus on a significantly smaller subset of random variables that influence the future evolution. As an example of application, we will show that the restarting strategy allows us to capture the invariant measure of the above  SDE, when such a measure exists. 

To simplify notation, we present our algorithm on the above scalar SDE with $L(u)$ a function of $u$, knowing that all results also apply to systems of SDEs with minor modifications; see sections \ref{sec:higherdim} and \ref{sec:numeric}.

\subsection{Formulation}

First, we notice that the solution $u(t)$ of \eqref{eq:sde} is a random process  depending on the initial condition and  the paths of Brownian motion up to time $t$: 
\[
u  = u(t; \, \{W_s, \, 0 \leq s \leq t \}, u_0 ).
\]
Therefore, recalling the expansion for $W_s$ \eqref{eq:expansion_W}, the solution at time $t$ can be seen as a nonlinear functional of $u_0$ and the countably infinite variables $\bxi$. As previously noticed in~\cite{GSVK10,HS14}, the solution $u(t)$ can be represented as a linear chaos expansion in terms of the polynomials in itself and therefore, for sufficiently small later times $t+\varepsilon$, $\varepsilon>0$, the solution $u(t+\varepsilon)$ can be efficiently captured by low order chaos expansions in $u(t)$ everything else being constant. Moreover, the solution $u(t+\varepsilon)$ on the interval $0<\varepsilon<\varepsilon_0$ depends on $W_{[t,t+\varepsilon_0]}$ and not on values of $W$ outside of this interval. This significantly reduces the number of random variables in $\bxi$ that need to be accounted for. This crucial observation clearly hinges upon the Markovian property of the dynamics.  Hence, dividing the time horizon into small pieces and iteratively employing PCE offer a possible way to alleviate both curse of dimensionality and long-term integration problems.

We decompose the time horizon $[0,t]$ into $n$ subintervals; namely $[0,t_1],[t_1,t_2],\ldots,[t_{n-1},t]$ where $0=t_0< t_1<\ldots<t_n=t$. The idea is then to employ polynomial chaos expansion in each subinterval and restart the approximation depending  on the distributions of both $\bxi_j$ and $u(t_j)$ at each $t_j$, $1 \leq j <n$; see Figure \ref{fig:prop_chaos}. Here, $\bxi_j$ denotes the Gaussian random variables required for Brownian forcing on the interval $[t_{j},t_{j+1}]$. 
Throughout the paper, we utilize the term $T_{\alpha}$ to represent orthonormal chaos basis involving its arguments. In order to establish evolution equations for chaos basis $T_{\alpha}(\bxi_j)$, triple products   $\E[T_{\alpha}(\bxi_j) T_{\beta}(\bxi_j) T_{\gamma}(\bxi_j)]$ are needed. This procedure basically corresponds to computing the coefficients and indices in the Hermite product formula \eqref{eq:her_prod}.  

The probability distribution of $u(t_j)$ does not belong to any classical family such as the Askey family in general. 
The construction of the orthogonal polynomials in $u(t_j)$ is therefore fairly involved computationally and is based on the general PCE results mentioned earlier in the section. At each time step, we compute the moments of $u(t_j)$ using its previously obtained chaos representation and incorporate them in a modified Gram--Schmidt method. This is a computationally expensive step. Armed with the orthogonal basis, we then compute the triple products in $u(t_j)$ to perform the necessary Galerkin projections onto spaces spanned by this orthogonal basis and thus obtain evolution equations for the deterministic expansion coefficients. Letting $u_{j}:=u(t_j)$, equations for the coefficients $(u_{j+1})_{\alpha}$ of $u_{j+1}$ are given in the general form 
 \begin{align*} 
d (u_{j+1})_{\alpha}  = \E \left[ T_{\alpha }(\bxi_j,u_j) \,  L\left(\sum_{\beta} (u_{j+1})_{\beta} \, T_{\beta}(\bxi_j,u_j) \right)  \right] ds + \sigma \, \E \left[T_{\alpha}(\bxi_j,u_j) \, d{W_s}\right], 
 \end{align*}
 where  $\alpha,\beta \in \mathcal{J}$. 
Before integration of these expansion coefficients in time, $u_j$ is represented in terms of its own orthogonal polynomials, which provides a description of the initial conditions on that interval. We then perform a high order time integration. Cumulants of the resulting solution are then computed to obtain relevant statistical information about the underlying distribution.

\begin{remark} \rm
The proposed methodology does not require the computation of any probability density function (pdf) and rather depends only on its moments at each restart time. The statistical information contained in such moments provides useful data for decision making in modeling and, for determinate distributions, moments characterize the distribution uniquely.  This aspect will be essential in the proof of convergence in section \ref{sec:Conv}.

Comparing our method with probability distribution function methods, e.g., the Fokker--Planck equation, we note that it essentially evolves the coefficients of orthogonal polynomials of the projected solution in time instead of evolving the pdf. 
\end{remark}

Let $Z_j$ represents the nonlinear chaos expansion mapping $(u_{j-1},\bxi_{j-1})$ to $u_{j}$. Our scheme may be demonstrated by Figure \ref{fig:prop_chaos}.

\begin{figure}[h]
\centering
\includegraphics[scale=1]{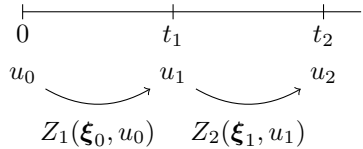}
  \caption{ Propagation  of chaos expansions.}
  \label{fig:prop_chaos}
\end{figure}
Mathematically, we have
\begin{align} \label{eq:sumgpc}
u_j = Z_j(\bxi_{j-1},u_{j-1}) = \sum_{\alpha \in \mathcal{J}} (u_j)_{\alpha} \, T_{\alpha} (\bxi_{j-1},u_{j-1}), \quad j \in \mathbb{N}.
\end{align}

In simulations, $\bxi_j$ is truncated with finite dimension $K \in \mathbb{N}$. Let also $N \in \mathbb{N}$ denote the maximum degree of the polynomials in the variables $(\bxi_j,u(t_j))$. Here, to simplify, $K$ and $N$ are chosen independently of $j$. The multi-indices having $K$ dimensions and degree $N$ belong to the set  
$
\mathcal{J}_{K,N} := \{ \alpha=(\alpha_1,\alpha_2,..,\alpha_K) \, | \, \alpha_j \in \mathbb{N}_0, \, |\alpha| \leq N \}.
$

We now present our algorithm, called \textit{Dynamical generalized Polynomial Chaos} (DgPC), in one dimension for concreteness; see also section \ref{sec:higherdim}.
\\

\noindent 
\textbf{Algorithm 1d}:  DgPC \\
\noindent
- decompose time domain $[0,t] =[t_0,t_1] \cup \ldots.\cup [t_{n-1},t]$  \\
- initialize degrees of freedom $K,N$ \\
- compute coefficients/indices in Hermite product formula for $\bxi_0$ \\
- integration in time:  \\
- \hspace{0.3cm}	time step $t_j \geq 0$ : obtain random variable $u_j$  \\
- \hspace{0.8cm} calculate moments $\E[(u_j)^{m}]$  \\
- \hspace{0.8cm} construct orthogonal polynomials $T_k(u_j)$ \\ 
- \hspace{0.8cm} compute triple products $\E[T_{k}(u_j)\, T_{l}(u_j) \,  T_{m}(u_j)]$\\
- \hspace{0.8cm} perform Galerkin projection onto  $ \mbox{span} \{ T_{\alpha}(\bxi_{j},u_j) \}$ \\
- \hspace{0.8cm} set up initial conditions for the coefficients $(u_j)_{\alpha}$ \\
- \hspace{0.8cm} evolve the expansion coefficients $(u_j)_{\alpha}$  \\
- \hspace{0.3cm} calculate cumulants \\
- postprocessing
\\

Several remarks are in order. 
First, since our stochastic forcing has identically distributed independent increments, we observe that at each subinterval the distribution of $\bxi_j$ is the same. Computing and storing (in sparse format) the coefficients for the product formula  \eqref{eq:her_prod} only once  drastically reduces the computational time. 

At each iteration, the random variable $u(t_j)=u_j$ is projected onto a finite dimensional chaos space and the next step is initialized with the projected random variable. For a $d$-dimensional SDE system, this projection leads to  ${K+N \choose N} \times {d+N \choose N}$ terms in the basis for the subinterval $[t_j,t_{j+1}]$.  Therefore, the total number of degrees of freedom used in $[0,t]$ becomes $n \times{K+N \choose N} \times {d+N \choose N}$; see also section \ref{sec:cost}.  We emphasize that small values of $K,N$ are utilized in each subinterval such that computations can be carried out quickly. 

The idea of iteratively constructing chaos expansions for arbitrary measures was considered earlier in~\cite{GSVK10,HS14,ON12, AGPRH12,AGPRH14}. In~\cite{AGPRH12,AGPRH14}, an iterative method is proposed to solve coupled multiphysics problems, where a dimension reduction technique is used to exchange information between iterations while allowing the construction of PC in terms of arbitrary (compactly supported) distributions at each iteration. A similar iterative procedure for Hermite PC was suggested for stochastic partial differential equations (SPDEs) with Brownian forcing in the conclusion section of~\cite{HLRZ06}  without any mathematical construction or numerical example. We also stress an important difference between our approach and~\cite{GSVK10}: our scheme solely depends on the statistical information, i.e., moments, which are available through chaos expansion,  whereas~\cite{GSVK10} either requires the pdf of $u_j$ at $t_j$ or uses  mappings to transform $u_j$ back to the original input random variables, which in our problem is high dimensional, including the $\xi$ variables.

\subsection{Implementation} \label{sec:Imp}

We now describe the implementation of our algorithm.

\subsubsection{Moments} \label{sec:moments}

Provided that the distribution of the initial condition $u_0$ is known, several methods allow us to compute moments: analytic formulas, quadrature methods, or Monte Carlo sampling. Also, in the case of limited data, moments can be generated from raw data sets. Throughout the paper, we assume that the moments of the initial condition can be computed  to some large finite order so that the algorithm can be initialized. 

We first provide the Hermite product formula, which establishes the multiplication of two PCE \eqref{eq:PCE} for random variables $u$ and $v$. The chaos expansion for the product becomes
\begin{align} \label{eq:her_prod}
uv  = \sum_{\alpha \in \mathcal{J}}  \sum_{\gamma \in \mathcal{J}} \!  \sum_{0 \leq \beta \leq \alpha} \! C(\alpha,\beta,\gamma)  u_{\alpha-\beta+\gamma}  v_{\beta+\gamma}    \, T_{\alpha}(\bxi), 
\end{align}
with
$ C(\alpha,\beta,\gamma)  = [{\alpha \choose \beta} {\beta+\gamma \choose \gamma } {\alpha-\beta+\gamma \choose \gamma} ]^{1/2},
$~\cite{Luo,HLRZ06}. As our applications include nonlinear terms, this formula will be used repeatedly to multiply Hermite chaos expansions. 

To compute moments of the random variable $u_j$ (recall \eqref{eq:sumgpc}) at time step $t_j$, $j=1,\ldots,n$, we use its chaos representation and previously computed triple products recursively as follows.  Due to the Markov property of the solution, the normalized chaos basis in $\bxi_{j-1} $ and $u_{j-1}$ becomes a tensor product, and thus we can write
\begin{align*}
u_j = \sum_{\alpha'} (u_j)_{\alpha'} \, T_{\alpha'} (\bxi_{j-1},u_{j-1}) = \sum_{\alpha,k} (u_j)_{\alpha,k} \, T_{\alpha}(\bxi_{j-1}) \, T_k(u_{j-1}),
\end{align*}
where $\alpha,\alpha' \in \mathcal{J}$ and $k \in \mathbb{N}_0$. Then for an integer $m>1$,
\begin{align*}
u_j^m & = \sum_{\alpha,k } (u_j)_{\alpha,k} \, T_{\alpha}(\bxi_{j-1}) \, T_k(u_{j-1})  \, \sum_{\beta,l}(u^{m-1}_j)_{\beta,l} \, T_{\beta} (\bxi_{j-1}) \, T_l(u_{j-1}). 
\end{align*}
Here, $\alpha,\beta \in  \mathcal{J}$ and $k,l \in \mathbb{N}_0$. Hence, using \eqref{eq:her_prod}, the coefficients $(u_j^m)_{\gamma,r}$, where $\gamma \in \mathcal{J}$  and $r \in \mathbb{N}_0$, can be calculated using the expression
\begin{align}\label{eq:gendec}
(u_j^m)_{\gamma,r}  = \sum_{\alpha \in \mathcal{J}} \sum_{ 0 \leq \theta \leq \gamma } \sum_{k,l} \, C(\gamma,\theta,\alpha)  \, \E[T_{r}(u_{j-1})T_{k}(u_{j-1})T_{l} (u_{j-1})] \, (u_j^{m-1})_{\gamma-\theta+\alpha,l}\,(u_j)_{\theta+\alpha,k},
\end{align}
which allows us to compute moments by simply taking the first coefficient, i.e. $\E[u_j^m] =  (u_j^m)_{\boldsymbol{0},0}$. Even though the coefficients $C(\gamma,\theta,\alpha)$  can be computed offline,  the triple products  $\E[T_{r}(u_{j})T_{k}(u_{j})T_{l} (u_{j})] $ must be computed at every step as the distribution of $u_{j}$ evolves. 

\subsubsection{Orthogonal Polynomials}

Given the real random variable $u_j$ obtained by the orthogonal projection of the solution $u$ onto homogeneous chaos at $t_j$, the Gram--Schmidt orthogonalization procedure can be used to construct the associated orthogonal polynomials of its continuous distribution. For theoretical reasons, we assume that the moment problem for $u_j$ is uniquely solvable; i.e., the distribution is nondegenerate and uniquely determined by its moments~\cite{akhiezer,gautschi}. 

To construct orthogonal polynomials, the classical Gram--Schmidt procedure uses the following recursive relation
\begin{align*} 
T_0= 1 , \quad  
T_m(u_j) = u_j^m - \sum_{l=0}^{m-1} a_l^m  \,T_l(u_j), \quad m \geq 1,
\end{align*}
where the coefficients are given by $a_l^m =\mathbb{E}[ u_j^m \, T_l(u_j)] \left(\E[T^2_l(u_j)]\right)^{-1}$. Note that $\E$ denotes the Lebesgue--Stieltjes integral with respect to distribution of $u_j$. We observe that the coefficient $a_l^m$  requires moments up to degree $2m-1$ as $T_{l}$ is at most of degree $m-1$. Thus, normalization would need first $2m$ moments for orthonormal polynomials up to degree $m$. 

After the construction of orthonormal polynomials of $u_j$, the basis $T_{\alpha}(\bxi_{j},u_j)$ can be generated by tensor products since $\bxi_{j}$ and $u_j$ are independent. Moreover, triple products
$
\E[T_{k}(u_j)\, T_{l}(u_j) \,  T_{m}(u_j)]$, where $k,l,m=0,\ldots,N,
$
can be found by simply noting that this expectation is a triple sum involving moments up to order $3N$.  Hence, the knowledge of moments yields not only the orthonormal set of polynomials but also the triple products required at each iterative step.

In our numerical computations, we prefer to utilize a modified Gram--Schmidt algorithm as the classical approach is overly sensitive to roundoff errors in cases of  high dimensionality~\cite{GL96}. We refer the reader to~\cite{gautschi} for other algorithms, e.g., Stieltjes and modified Chebyshev methods.

\subsubsection{Initial conditions}

When the algorithm is reinitialized at the restart point $t_j$, it needs to construct the initial condition in terms of chaos variables for the next iteration on $[t_{j},t_{j+1}]$. In other words, we need to find the coefficients $ (u_j)_{\alpha}$ in terms of $u_j$ and $\bxi_{j}$. To this end, we observe that the first two polynomials in $u_j$ are given by
$
T_0(u_j) =1$ and  $ T_1(u_j) = \sigma_{u_j}^{-1}(u_j - \E[u_j]),
$
where $\sigma_{u_j}$ is the standard deviation of the random variable $u_j$. Notice that $\sigma_{u_j}>0$, because we assumed in the preceding section that the distribution was nondegenerate.  Hence, the initial condition can be easily found by the relation
\begin{align*}
u_j = \E[u_j] T_0(u_j) + \sigma_{u_j} T_1(u_j) = \sum_{\alpha \in \mathcal{J}, \, k \in \mathbb{N}_0} (u_j)_{\alpha,k} \, T_{\alpha}(\bxi_{j}) \,T_k(u_{j}).
\end{align*}

\subsubsection{Extension to higher dimensions} \label{sec:higherdim}

We now extend our algorithm to higher dimensions, and for concreteness, we consider two dimensions. Let $u_{j+1}=(v_1,v_2) \in \R^2$ be a two-dimensional random variable obtained by projection of the solution $u$ onto homogeneous chaos space at time $t_{j+1}$. 

The two-dimensional case requires the calculation of the mixed moments $\E[v_1^l v_2^m], l,m \in \mathbb{N} \cup \{0 \}$. In the case of independent components $v_1,v_2$, the moments become products of marginal moments and the basis $T_{\alpha}(v_1,v_2)$ is obtained by tensor products of the one-dimensional bases. For a nonlinear system of equations, the components of the solution typically do not remain independent as time evolves. Therefore, we need to  extend the procedure to correlated components. 

Denoting by $Z_{j+1,1}$ and $Z_{j+1,2}$ the corresponding chaos expansions on  $[t_{j},t_{j+1}]$ so that $v_1= Z_{j+1,1}(\bxi_{j},u_{j})$ and $v_2= Z_{j+1,2}(\bxi_{j},u_{j})$, we compute the mixed moments by the change of variables formula
\begin{align} \label{eq:mixedmom}
 \E[v_1^l \, v_2^m]=  \E_{(v_1,v_2)}[v^l_1 \, v^m_2] =  \E_{(\bxi_{j},u_{j})} \left[Z_{j+1,1}^l \, Z_{j+1,2}^m \right],
\end{align}
where $ \E_{(\cdot)}$ represents the expectation with respect to the joint distribution induced by the subscript. Note that the latter integral in \eqref{eq:mixedmom} can be computed with the help of previously computed triple products of $\bxi_j$ and  $u_{j}$. Incidentally, similar transformation methods were used previously  in~\cite{HS14,AGPRH12,AGPRH14}.

Since the components may not be independent, the tensor product structure is lost but the construction of orthogonal polynomials is still possible. Based on the knowledge of marginal moments, we first compute the tensor product $T_{\alpha}(v_1,v_2)$. However, this set is not orthogonal with respect to the joint probability measure of $(v_1,v_2)$. Therefore, we further orthogonalize it via a Gram--Schmidt method. Note that this procedure requires only mixed moments $E_{(v_1,v_2)}[v^l_1 v^m_2]$, which have already been computed in the previous step of the algorithm. It is worth noting that the resulting set of orthogonal polynomials is not unique as it depends on the ordering of monomials~\cite{Xu97}. In applications, we consider the same ordering used for the set of multi-indices. Finally, calculation of triple products and initial conditions can be extended in an obvious way. 

\subsection{Computational Complexity} \label{sec:cost}

In this section, we discuss the computational costs of our algorithm using a vector-valued version of the SDE \eqref{eq:sde} in $\R^d$ for a fixed $d \in \mathbb{N}$. For each scalar SDE, we assume that the nonlinearity is proportional to $m${th} order monomial with $m \in \mathbb{N}$. Furthermore, we compare costs of DgPC and Hermite PC in the case where both methods attain a similar order of error.

Throughout this section, $K,N$ will denote the dimension and the degree in  $\bxi= (\xi_1, \ldots, \xi_K)$ for Hermite PC. Here $\bxi$ represents the truncation of $d$-dimensional Brownian motion; therefore, $K\gg d$.  The time interval is fixed and given by $[0,1]$.  We also let  $M(K,N):=$ $ K+N \choose K$ be the degrees of freedom for the simple truncation \eqref{eq:simp_trunc} so that the dimension of  truncated multi-indices $\mathcal{J}_{K,N}$ becomes $M$.
For the DgPC method,  we divide the interval into $n>1$ identical subintervals so that $\Delta t = n^{-1}$. We now slightly change the notation of previous sections and denote by $K_*, N_*$ the dimension and degree of polynomials of $\bxi_j=(\xi_{jK_*+1},\ldots,\xi_{jK_*+K_*})$  used in DgPC approximation for each $1 \leq j <n$. With additional $d$ variables at each restart, the dimension of the multi-index set for DgPC becomes $M(K_*,N_*) M(d,N_*)$ due to the tensor product structure. Further, let $h <1$ and $\zeta \geq 1$ denote the time step and global order of convergence for the time integration method employed in Hermite PC, respectively.

\subsubsection{Computational costs}

 With the above notation, we summarize the computational costs in Table \ref{table:compcost}.  All terms should be understood in big $O$ notation.

 \begin{table}[!htb]
\centering
\renewcommand{\tabcolsep}{0.15cm}
\renewcommand{\arraystretch}{1.1}
 \begin{tabular} { r | p{6.5cm}|p{3.9cm}}
         Flops & {DgPC} & Hermite PC \\ \hline
    Offline &  $K_{*} \times M(K_*,N_*)^3$  & $K \times M(K,N)^3$\\
    Time evolution &  $d \times h^{-1} \times n^{1/ \zeta} \times (M(K_*,N_*)M(d,N_*)) \times   [1 +(m-1) \times (M(K_*,N_*)M(d,N_*))^2]$ & $ d \times h^{-1} \times M(K,N) \times  [ 1+  (m-1) \times M(K,N)^2]$\\
    Moments &  $ n  \times (M(K_*,N_*)M(d,N_*))^3 \times [ (6d-3d^2)\times N_* + (d-1) \times M(d,3N_*)]$ & \\
	Gram--Schmidt & $ n \times M(d,N_*)^3$ & \\
	Triple products & $ n \times M(d,N_*)^6$ & \\
	Initials & $ n \times d \times (M(K_*,N_*) M(d,N_*))$&  \\
  \end{tabular} 
  \caption{Total computational costs.}
   \label{table:compcost}
\end{table}

The estimates in Table \ref{table:compcost} are obtained as follows.
Triple products for $\bxi$ with $\alpha,\beta,\gamma \in \mathcal{J}_{K,N}$ can be calculated by the equation
\begin{align*} \E[T_{\alpha}(\bxi)T_{\beta}(\bxi)T_{\gamma}(\bxi)]& = \E\left[ \prod_{i=1}^K H_{\alpha_i} (\xi_i) \prod_{j=1}^K H_{\beta_j} (\xi_j) \prod_{k=1}^K H_{\gamma_k} (\xi_k) \right]= \prod_{i=1}^K
 \E[H_{\alpha_i}(\xi_i)H_{\beta_i}(\xi_i)H_{\gamma_i}(\xi_i)].
\end{align*}
Thus, the offline cost is of order $K \times M(K,N)^3 $ assuming that one-dimensional triple products can be read from a precomputed table. 

Since  a time discretization with error $O(h^{\zeta})$ is employed for Hermite PC, then for DgPC in each subinterval,  time steps of order $h^{-1} n^{(1-\zeta)/ \zeta}$ should be used to attain the same global error.  Without nonlinearity, the total cost in Hermite PC for evolution of a $d$-dimensional system becomes $d \times h^{-1} \times M(K,N)$. Due to the functional $L(u) \propto u^m, m>1$, the cost should also include the computation of $u^m$ at each time integration step. This requires additional $d \times h^{-1} \times (m-1) \times M(K,N)^3$  work (see the computation of moments below). Hence, the corresponding total cost of evolution becomes the sum of these costs.

The coefficients of the $k${th} power of a single projected variable $u_j$ are given by
\[
(u_j^k)_{\gamma} = \sum_{\alpha,\beta} (u_j^{k-1})_{\alpha} \, (u_j)_{\beta} \, \E[T_{\alpha}T_{\beta}T_{\gamma}] , \quad 2 \leq k \leq 3N_*, 
\]
assuming that, after each multiplication, the variable is projected onto the PC basis. Here, $\alpha,\beta,\gamma \in \mathcal{J}_{K_*,N_*} \otimes \mathcal{J}_{d,N_*}$. Thus, the computation of marginal moments of all variables requires $d \times 3N_{*} \times (M(K_{*},N_*) M(d,N_*))^3 $ calculations at each restart time in DgPC. Since we use the same ordering for multidimensional moments as $\mathcal{J}_{d,3N_*}$, computing joint moments further needs $(d-1) \times (M(d,3N_*)-3dN_*) \times (M(K_{*},N_*)M(d,N_*) ) ^3$ amount of work. 

The Gram--Schmidt procedure costs are cubic in the size of the Gram matrix, which has dimension $M(d,N_*)$ in DgPC.  Each triple product $\E [T_{\alpha}(u_j) T_{\beta}(u_j) T_{\gamma}(u_j)], \alpha,\beta,\gamma \in \mathcal{J}_{d,N_*}$, is a triple sum; therefore, the total cost is proportional to $M(d,N_*)^6$. Also, initial conditions can be computed by orthogonal projection onto the DgPC basis, which costs $M(K_*,N_*) M(d,N_*)$ per dimension. 

Note that although we employed simple truncation \eqref{eq:simp_trunc} for multi-indices, sparse truncation techniques can be utilized to further reduce the costs of both computing moments and evolution; see \cite{Luo,HLRZ06, LeMK10}.

\subsubsection{Error bounds and cost comparison} 
 
 We now discuss the error terms in both algorithms.  We posit that the  Hermite PCE converges algebraically in $N$ and $K$ and that the error at time $t$ satisfies 
\begin{align}\label{eq:pce_err}
|| u - u_{pce} ||_{L^2} \lesssim  t^{\delta} ( N^{-\eta}  + K^{-\lambda}),
\end{align}
for some constants $\eta,\lambda >0$ and  $\delta >1$ depending on the SDE. Here $A \lesssim B$ denotes that $A \leq c B $ for $c>0$ with $A,B  \geq 0 $.  This assumption enforces an increase in the degrees of freedom $N,K$ if one wants to maintain the accuracy in the long term. Algebraic convergence in $K$ and the term $t^{\delta}$ stems from the error bound \eqref{eq:exp_W_err}. We do not show errors resulting from time discretization in \eqref{eq:pce_err} since we already fixed those errors to the same order in both methods. Incidentally, even though we assumed algebraic convergence in $N$, exponential numerical convergence, depending on the regularity of the solution in $\bxi$, is usually observed in the literature \cite{Xiu_thesis,Xiu_book,GS91,Luo,XK02,WK06}.  

Although it is usually hard to quantify the constants $\eta,\lambda,\delta$, this may be done for simple cases. For instance, for the Ornstein--Uhlenbeck (OU) process \eqref{eq:OU}, using the exact solution and Fourier basis of cosines,  we can obtain the parameters  $\delta=2$ and $\lambda=3/2$ for short times. The parameter $\eta$ does not play a role in this case since the SDE stays Gaussian and hence we take $N=1$.  We  also note that for complex dynamics, $\eta$ may depend on $N$; see \cite{Luo,WK06} in the case of the stochastic Burgers equation and  advection equations.

Now, we fix the error terms in \eqref{eq:pce_err} to the same order $O(\varepsilon)$, $\varepsilon>0$, by choosing $N = O(K^{\lambda/\eta})$. Since DgPC uses truncated $\bxi$ of dimension $K_*$ and polynomials of degree $N_*$ in each subinterval of size $\Delta t =n^{-1}$, the $L^2$ error at $t=1$ has the form
\begin{align} \label{eq:dopc_err}
|| u - u_{dgpc} ||_{L^2} \lesssim  \frac{n}{n^{\delta}}  \left( N_*^{-{\eta}}  + K_*^{-\lambda} \right).  
\end{align}
Thus, to maintain the same level $\varepsilon$ of accuracy, we can choose 
\[n^{1-\delta} K_*^{-\lambda}  \cong K^{-\lambda} \quad \mbox{and} \quad  n^{1-\delta} N_*^{-\eta}  \cong K^{-\lambda}. \]
From Table \ref{table:compcost}, we observe that to minimize costs for DgPC, we can take $K_*=d$ and $N_*=1$ and maximize $n$ according to the previous equation as 
\begin{align} \label{eq:dogpc_n}
 n  \cong K^{\frac{\lambda}{\delta-1}}.
\end{align}
 With these choices of parameters, the total computational cost for DgPC is of algebraic order in $K$ and in general, is dominated by the computation of moments.
 
For Hermite PC with large $K$ and $N=O(K^{\lambda/\eta})$ (or $N=O(\lambda/\eta \log K)$ in the case of exponential convergence in $N$), both offline and evolution stages include the cubic term $M(K,N)^3$. Both costs increase exponentially in $K$. Therefore, the asymptotic analysis suggests that DgPC can be performed with substantially lower computational costs using frequent restarts (equation \eqref{eq:dogpc_n}) in nonlinear systems driven by high dimensional random forcing.

\subsection{Convergence results} \label{sec:Conv}

We now consider the convergence properties of our scheme as degrees of freedom tend to infinity. 

\subsubsection{Moment problem and density of polynomials}

There is an extensive literature on the moment problem for probability distributions; see~\cite{akhiezer,BC81,P82,freud} and their references. We provide the relevant background in the analysis of DgPC.

\begin{defn}\emph{(Hamburger moment problem)} For a probability measure $\mu$ on $(\R,\mathcal{B}(\R))$, the moment problem is uniquely solvable provided moments of all orders 
$
 \int_{\R} x^k \mu(dx)$, $ k \in \mathbb{N}_0,
$  
exist and they uniquely determine the measure. In this case, the distribution $\mu$ is called \textit{determinate}. 
\end{defn}

There are several sufficient conditions for a one-dimensional distribution to be determinate in Hamburger sense: these are compact support, exponential integrability and Carleman's moment condition~\cite{akhiezer,BC81}. For instance, Gaussian and uniform distributions are determinate, whereas the lognormal distribution is not. The moment problem is intrinsically related to the density of associated orthogonal polynomials. Indeed, if the cumulative distribution $\mathbb{F}_u$ of a random variable $u$ is determinate, then the corresponding orthogonal polynomials constitute a dense set in $L^2(\R,\mathbb{B}({\R}),\mathbb{F}_u)$ and therefore also in $L^2(\Omega,\sigma(u),\Prm)$~\cite{akhiezer,BC81,freud}. Additionally, finite dimensional distributions on $\R^d$ with compact support are determinate; see~\cite{P82}.

Now, let $\boldsymbol{\zeta}=(\zeta_i)_{i \in \mathbb{N}}$ denote an independent collection of general random variables, where each $\zeta_i$ (not necessarily identically distributed) has finite moments of all orders and  its cumulative distribution function is continuous. Under these assumptions,~\cite{EMSU12} proves the following theorem. 

\begin{theorem} \label{thm:conv_ernst} For any random variable $\eta \in L^2(\Omega, \sigma(\boldsymbol{\zeta}),\Prm)$, gPC  converges to $\eta$ in $L^2$ if and only if the moment problem for each $\zeta_i$,
$i \in \mathbb{
N}$, is uniquely solvable. 
\end{theorem} 

This result generalizes the convergence of Hermite PCE to general random variables whose laws are determinate. Notably, to prove $L^2$ convergence of chaos expansions, it is enough to check one of the determinacy conditions mentioned above for each one-dimensional $\zeta_i$.  

We now consider the relation between determinacy and distributions of SDEs. Consider the following diffusion
\begin{align} \label{eq:diff}
du_s  = b(u_s) \, ds + \sigma(u_s) \, dW_s, \quad u(0)=u_0,  \quad s \in [0,t],
\end{align}
where $W_s$ is $d$-dimensional Brownian motion, $u_0 \in \R^d$, and $b: \R^d \rightarrow \R^d, \sigma: \R^d \rightarrow \R^{d \times d}$ are globally Lipschitz and satisfy the usual linear growth bound. These conditions imply that the SDE has a unique strong solution with continuous paths~\cite{OB03}. Determinancy of the distribution of $u_s$ is established by the following theorem; see~\cite{P82,SJ02,FL14} for details. 

\begin{theorem} \label{thm:det_sde}
The law of the  solution of  \eqref{eq:diff} is determinate if $\sup_x || \sigma \sigma^T (x) ||$ is finite.
\end{theorem} 

\subsubsection{$L^2$ convergence with a finite number of restarts}

Consider the solution of the SDE \eqref{eq:diff} written as
\begin{align} \label{eq:markov_proc}
u_s = F_s(u_0, \bxi_0^s),
\end{align}
where $F_s: \R^d \times \R^{\infty} \rightarrow \R^d$ is the exact evolution operator mapping the initial condition $u_0$ and $\bxi_0^s=(\xi^s_1,\xi^s_2,\ldots)$ to the solution $u_s$ at time $s>0$. Here, with a slight change of notation, $\bxi_0^s$ represents Brownian motion on the interval $[0,s]$. Note that future solutions $u_{s+\tau}$, $\tau >0$, are obtained by $F_s(u_{\tau},\bxi_{\tau}^{s+\tau})$, where $\bxi_{\tau}^{s+\tau}$ denotes Brownian motion on the interval $[\tau, \tau+s]$. 

Even though the distribution of the exact solution $u_s$ is determinate under the hypotheses of Thm. \ref{thm:det_sde}, it is not clear that this feature still holds for the projected random variables. To address this issue, we introduce, for $R \in \R_{+}$, a truncation function $\chi_R \in C_c^{\infty}(\R^d \to \R^d)$
\[
\chi_R(u) := \begin{cases} u, & \mbox{ when } u \in A_R ,  \\
0, & \mbox{ when }  u \in \R^d \setminus A_{3R} ,
\end{cases}
\]
where $A_R:= \prod_{i=1}^d [-R,R] \subset \R^d$, and such that $\chi_R(u)$ decays smoothly and sufficiently slowly on $A_R^c$ such that the Lipschitz constant ${\rm Lip}(\chi_R)$ of $\chi_R$  equals $1$; i.e., for $|\cdot|$, the Euclidean distance in $\R^d$, $|\chi_R(u)-\chi_R(v)|\leq|u-v|$. Such a function is seen to exist by appropriate mollification for each coordinate using the continuous, piecewise smooth, and odd function $\chi(u_1)$ defined by $R-|u_1-R|$ on $[0,2R]$, and extended by $0$ outside $[-2R,2R]$. This truncation is a theoretical tool that allows us to ensure that all distributions with support properly restricted on compact intervals remain determinate; see~\cite{P82}.

We can now introduce the following approximate solution operators for $R \in \R_{+}$ and $M \in \mathbb{N}_0$: 
\begin{align}
\label{eq:chainR}
F_s^R (u,\bxi_0^s) &: = \chi_R \circ F_s (u,\bxi_0^s) : \R^d \times \R^{\infty} \to A_{3R}, \\
\label{eq:chainNR}
F_s^{M,R}(u,\bxi_0^s) &: = P^M \circ F_s^R(u,\bxi_0^s) = \sum_{ i=0}^{M-1} F^R_{i,s} \, T_{i}(u,\bxi_0^s), 
\end{align}
where $P^M$ denotes the orthogonal projection onto polynomials in $u$ and $\bxi_0^s$, and $M$ is the total number of degrees of freedom used in the expansion. Throughout this section, we  use a linear indexing to represent polynomials in both $u$ (restricted on $A_{3R}$) and the random variables $\bxi$.

The main assumptions we imposed on the solution operator $F_s$ are summarized  as follows: 
\begin{enumerate}[i)]
\item  $ \E_{\bxi_0^s}|F_{s}(0,\bxi_0^s)|^p  \leq C_s$, where  $0 <C_s< \infty$ and $p = 2+\epsilon$ with  $\epsilon>0$. \label{item:zero}
\item  $ \E_{\bxi_0^s}|F_{s}(u,\bxi_0^s) - F_{s}(v,\bxi_0^s)|^p  \leq \rho^p_{s} \, |u-v|^p$, where $u,v \in \R^d$,  $0 <\rho_s< \infty$, and $p = 2+\epsilon$ with  $\epsilon>0$. \label{item:one}
\item For  $\varepsilon>0$ and $R>0$, there is $M(\varepsilon, R,s)$ so that $\E_{\bxi_0^s} |F^R_s(u,\bxi_0^s) - F_s^{M,R}(u,\bxi_0^s)|^2  \leq \varepsilon$, where $u \in A_{3R}$. \label{item:two}
\end{enumerate}
Assumption \ref{item:zero} is a stability estimate ensuring that the chain remains bounded in $p$th norm starting from a point  in $\R^d$ (here $0$ without loss of generality owing to \ref{item:one}). Assumption \ref{item:one} is a Lipschitz growth condition controlled by a constant $\rho_s$. These assumptions involve the $L^p$ norm, and hence are slightly stronger than control in the $L^2$ sense. Note that $F_s^R$ also satisfies assumption \ref{item:one} with the same constant $\rho_s$ since ${\rm Lip}(\chi_R)=1$. Finally, assumption \ref{item:two} is justified by the  Weierstrass approximation theorem for the $u$ variable in $A_{3R}$ and by the Cameron--Martin theorem to handle the $\bxi$ variable. 

Consider now the following versions of the Markov chains:  
\begin{align}  
 \label{eq:sampch1} 
u_{j+1} &:= F_{s}(u_{j},\bxi_j),  \quad j=0,\ldots,n-1, \\
 \label{eq:sampch2} 
u_{j+1}^R &:= F_{s}^R (u_j^R, \bxi_j), \quad u_0^R = \chi_R (u_0),\quad j=0,\ldots,n-1,  \\ 
 \label{eq:sampch3}
u_{j+1}^{M,R} &:= F_{s}^{M,R} (u_j^{M,R}, \bxi_j), \quad u_0^{M,R} = \chi_R (u_0),\quad j=0,\ldots,n-1,
\end{align}
where $u_0\in\R^d$ (possibly random and independent of the variables $\bxi$) and $\bxi_j = \bxi_{js}^{(j+1)s}$ is a sequence of i.i.d. Gaussian random variables representing Brownian motion on the interval $[js, (j+1)s]$. We also use the notation $\bxi_0^{js}$ to denote Brownian motion on the interval $[0,js]$. Then we have the first result.

\begin{lemma}  \label{lemma:Rapprox}
Assume \ref{item:zero}, \ref{item:one}, and \ref{item:two} hold, and let $n \in \mathbb{N}$ be finite. Suppose also that $\E |u_0|^p =C_0 <\infty$.  Then, for each $\varepsilon>0$, there exists $R \in \R_{+}$ depending on $n$, $s$, and $C_{0}$ such that $\E |u_j - u_j^R|^2 \leq \varepsilon$ for each $0 \leq j \leq n$.
\end{lemma}
\begin{proof}
Let $\varepsilon>0$ be fixed. Using properties  \ref{item:zero} and \ref{item:one}, we observe that
\begin{align*}
\E_{\bxi_0^{(j+1)s}} |F_s(u_j,\bxi_j)|^p  &\leq 2^{p-1} \left(  \E_{\bxi_{0}^{js}} \E_{\bxi_j} |F_s(u_j,\bxi_j) -F_s(0,\bxi_j) |^p + \E_{\bxi_{0}^{js}} \E_{\bxi_j} |F_s(0,\bxi_j)|^p \right), \\
& \leq 2^{p-1} \left(  \rho_s^p \, \E_{\bxi_{0}^{js}}  |u_j |^p+C_s\right),
\end{align*}
where for all $j$, $\bxi_{j}$ and $\bxi_0^s$ are identically distributed. Then, by induction, we obtain
\begin{align} \label{eq:ujpbound}
\E_{u_0,\bxi_0^{(j+1)s}} |u_{j+1}|^p \leq (2^{p-1} \rho_s^p)^{j+1} \, \E|u_0|^p +      C_{s,n}^p \leq C_{s,n,C_{0}} , \quad  0 \leq j \leq n-1,
\end{align}
where the constant $C_{s,n,C_0}$ is bounded. The last inequality indicates that $p$th norm of the solution grows with $j$, possibly exponentially, but remains bounded. 

For $u_j =F_{js}(u_0,\bxi_{0}^{js})$, we note that $\E |u_j|^2 = \E_{u_0,\bxi_{0}^{js}}|u_j|^2$ since $u_j$ is independent of future forcing. By Markov's inequality and  \eqref{eq:ujpbound}, we get
\begin{align}\label{eq:markovbound}
\Prm ( u_j \in A_R^c )  \leq \frac{\E |u_j|^2}{R^2}  \leq C_{s,n,C_0} \, R^{-2}, \quad  1 \leq j \leq n.
\end{align}
Let $\mathbbm{1}_{A}$ be the indicator function of the set $A$.
Using \eqref{eq:ujpbound} and \eqref{eq:markovbound}, we compute the following error bound
\begin{align}
\nonumber
 \E | u_j - \chi_R u_j  |^2  
  &  =\E [ (\mathbbm{1}_{\{u_j \in A_R \}}+\mathbbm{1}_{ \{ u_j \in A_R^c \}} ) |u_{j} -\chi_R u_{j} |^2 ], \quad  1 \leq j \leq n,\\
  \nonumber
  & \leq 4 \, \E [ \mathbbm{1}_{ \{ u_{j} \in A_R^c \}}  |u_{j} |^2 ]  \leq 4 \, (\E |u_j|^p )^{2/p} \, \Prm (u_j \in  A_R^c)^{1/q}, \quad p=2+\epsilon  , q= 1+2\epsilon^{-1}, \\ 
  \label{eq:ujRbound}
  &  \leq C_{s,n,C_0} \,  R^{-2/q}.
\end{align}
Then, using property \ref{item:one} gives
\begin{align*}
\E_{\bxi_{j-1}}| u_j - u_j^R |^2& = \E_{\bxi_{j-1}}| F_{s}(u_{j-1},\bxi_{j-1}) - F^R_{s}(u^R_{j-1},\bxi_{j-1})|^2, \\
\nonumber
 & \leq  (1+\delta^{-1}) \E_{\bxi_{j-1}}|  u_j - \chi_R  u_j  |^2   \\
 & \qquad \quad
  + (1+\delta) \E_{\bxi_{j-1}}| \chi_R  F_{s}(u_{j-1},\bxi_{j-1}) - F^R_{s}(u^R_{j-1},\bxi_{j-1}) |^2, \quad \delta >0 ,\\
\nonumber
& \leq  (1+\delta^{-1}) \E_{\bxi_{j-1}}|  u_j - \chi_R u_j  |^2   
 +(1+\delta)   \rho_s^2  | u_{j-1} - u_{j-1}^R |^2. 
\end{align*}
 Taking expectation with respect to remaining measures and using \eqref{eq:ujRbound}  yields
\begin{align}
\nonumber
\E| u_j - u_j^R |^2 & \leq  (1+\delta)  \rho_s^2  \, \E| u_{j-1} - u_{j-1}^R |^2  + (1+\delta^{-1}) C_{s,n,C_0} \, R^{-2/q},\\
\label{eq:lasteqn}
&  \leq ((1+\delta)  \rho_s^2)^{j}  \E | u_{0} - u_{0}^R |^2 + (1+\delta^{-1}) C_{s,n,C_0} \,  R^{-2/q} \sum_{i=0}^{j-1} ((1+\delta)  \rho_s^2)^{i}.
 \end{align}
Setting $\delta =1$ in \eqref{eq:lasteqn} entails
 \begin{align*}
 \E | u_j - u_j^R |^2 & 
  \leq 2^j  \rho_s^{2j} \, \E | u_{0} - u_{0}^R |^2  +  \,C_{s,n,C_0} \, R^{-2/q}. 
 \end{align*}
 Since $\E |u_0|^p <\infty$, we also have $\E |u_0 -u_0^R|^2 =O(R^{-2/q})$. Thus, we can choose $R( n, s ,C_{0}) \in \R_{+}$ large enough such that $\E | u_j - u_j^R |^2 \leq \varepsilon$ for each $0 \leq j \leq n$. 
\end{proof}

Based on the preceding lemma, we prove that the chain \eqref{eq:sampch3} converges to the solution of the chain \eqref{eq:sampch1} in the setting of a finite number of restarts. 

\begin{theorem} \label{thm:convfinite}
 Under the assumptions of Lemma \ref{lemma:Rapprox}, for each $\varepsilon>0$, there exists $R(n, s, C_{0}) \in \R_{+}$ and then $M(R,n, s)\in \mathbb{N}_0$ such that $\E |u_j - u_j^{M,R}|^2 \leq \varepsilon$ for each $0 \leq j \leq n$.
\end{theorem}

\begin{proof}
Let $\varepsilon >0$ be fixed. By the previous lemma, choose $R>0$ such that $\E | u_j - u_j^R |^2 \leq \varepsilon/4 $ for $0 \leq j \leq n$. Then, using calculations similar  to those above, by \ref{item:one} we get 
\begin{align*}
\E_{\bxi_{j-1}} | u_{j}^R - u_{j}^{M,R} |^2 & \leq (1+\delta) \E_{\bxi_{j-1}} | F^R_{s}(u_{j-1}^{R},\bxi_{j-1}) - F^R_{s}(u_{j-1}^{M,R},\bxi_{j-1}) |^2  \\
& \qquad  \quad  +(1+\delta^{-1}) \E_{\bxi_{j-1}}| F^R_{s}(u_{j-1}^{M,R},\bxi_{j-1}) - F^{M,R}_{s}(u_{j-1}^{M,R},\bxi_{j-1})|^2,  \delta >0,\\
& \leq  2  \rho^2_{s} \,| u_{j-1}^R - u_{j-1}^{M,R} |^2  \\
& \qquad  \quad  +2 \, \E_{\bxi_{j-1}}| F^R_{s}(u_{j-1}^{M,R},\bxi_{j-1}) -  F^{M,R}_{s}(u_{j-1}^{M,R},\bxi_{j-1})  |^2, \, \delta= 1,\\
& =   \rho_* \,| u_{j-1}^R - u_{j-1}^{M,R} |^2   + 2 \, \E_{\bxi_{j-1}}| F^R_{s}(u_{j-1}^{M,R},\bxi_{j-1}) -  F^{M,R}_{s}(u_{j-1}^{M,R},\bxi_{j-1})  |^2,
\end{align*}
where $j \geq 1$ and $\rho_*:= 2\rho_s^2$. Then by \ref{item:two}, we choose $M(R,n,s)$ sufficiently large so that
\begin{align} \label{eq:ujmr}
\E_{\bxi_{j-1}} | u_{j}^R - u_{j}^{M,R} |^2 
& \leq \rho_{*} \,| u_{j-1}^R - u_{j-1}^{M,R} |^2   + \frac{\varepsilon}{4} \frac{\rho_*-1}{\rho_*^n-1}.
\end{align}
The last inequality can be rewritten as 
\begin{align*}
\E | u_{j}^R - u_{j}^{M,R} |^2 
& \leq \rho_{*}^j \,\E| u_{0}^R - u_{0}^{M,R} |^2   + \varepsilon/4 = \varepsilon/4 .
\end{align*}
Therefore, $\E| u_{j} - u_{j}^{M,R} |^2 \leq  \varepsilon$ for each $ 1 \leq j \leq n$.
\end{proof}

\subsubsection{Convergence to invariant measures and long time evolution}

Consider the setting of the solution operator to the SDE given in \eqref{eq:markov_proc}. As $s$ increases to $\infty$, the random variable $u(s)$ may converge in distribution to a limiting random variable $u_\infty$, whose distribution is the invariant measure of the evolution equation \eqref{eq:markov_proc}. Although there are many efficient ways to analyze and compute such invariant measures (see for instance \cite{SS2000}), we wish to show that our iterative algorithm also converges to that  invariant measure as degrees of freedom tend to infinity. In other words, our PCE-based method allows us to remain accurate for long-time evolutions--something that cannot be achieved without restarts.

Now we iterate each discrete chain in \eqref{eq:sampch1}, \eqref{eq:sampch2} and \eqref{eq:sampch3} for an arbitrary  $j \in \mathbb{N}_0$. In order to ensure long-time convergence, we need a stricter condition than \ref{item:one} and impose instead the following contraction condition 
\begin{enumerate}[label=ii')]
\item  $ \E_{\bxi_0^s}|F_{s}(u,\bxi_0^s) - F_{s}(v,\bxi_0^s)|^p  \leq \rho^p_{s} \, |u-v|^p$, where $0<\rho_s< 1$,  $p = 2+\epsilon,$ $\epsilon>0$, and $u,v \in \R^d$. \label{item:onemod}
\end{enumerate}

We first prove the following result about the existence and uniqueness of an invariant measure for the original chain \eqref{eq:sampch1}. 
\begin{lemma} \label{lemma:invariant}
Under conditions \ref{item:zero} and \ref{item:onemod}, there exists a unique invariant measure $\nu$ of the chain \eqref{eq:sampch1} with bounded $p$th moment. Moreover, if $\E|u_0|^p <\infty$, then  $\E|u_j|^p$ is bounded uniformly in $j$.  
\end{lemma}
\begin{proof}
For $u \in \R^d$ (and using that $|a+b|^p\leq (1+\delta)^p|a|^p+(1+\delta^{-1}))^p|b|^p$), we compute 
\begin{align*}
\E_{\bxi_0^s}|F_{s}(u,\bxi_0^s)|^p &\leq (1+\delta)^p \, \E_{\bxi_0^s}|F_{s}(u,\bxi_0^s) - F_s(0,\bxi_0^s)|^p + (1+\delta^{-1})^p \, \E_{\bxi_0^s}|F_{s}(0,\bxi_0^s)|^p, \quad \delta >0,\\
&  \leq  ((1+\delta) \rho_s)^p \, |u|^p + (1+\delta^{-1})^p \, C_s, \\
&  = \rho_*\, |u|^p + C,
\end{align*}
where $\delta$ is chosen so that $\rho_{*}:= ((1+\delta)\rho_s)^p <1$ and $C <\infty$. Thus, there exists a Lyapunov function $V(u)=|u|^p$ with a constant $\rho_* <1$. This in turn implies the existence of an invariant measure $\nu$ with finite $p$th moment for the process \eqref{eq:sampch1}; see \cite[Corollary 4.23]{H06}. Uniqueness also follows from assumption \ref{item:onemod}; see \cite[Theorem 4.25]{H06}.

Let $v_0$ be a random variable with law $\nu$ and independent of $u_0$. Consider the chain $v_{j+1} = F_s(v_j, \bxi_j)$ started from $v_0$.  Note  $v_j \sim \nu$ for all $j \in \mathbb{N}_0$. Then, from the bound
\begin{align*}
\E |u_j-v_j|^p  &\leq (\rho_s^p)^j \E|u_0-v_0|^p,
\end{align*}
 we conclude that   $\sup_j \E |u_j|^p < \infty$.
\end{proof}

The following theorem establishes the exponential convergence of the PC chain \eqref{eq:sampch3} to the chain \eqref{eq:sampch1} as time $j$ increases. 

\begin{theorem} Assume \ref{item:zero}, \ref{item:onemod} and \ref{item:two}  hold, and  $\E |u_0|^p <\infty$. Then, for each $\varepsilon>0$, there exists $R>0$ independent of $j$, and then $M(R,s) \in \mathbb{N}_0$, such that $\E|u_j - u_j^{M,R}|^2 \leq \varepsilon$ for all $j \in \mathbb{N}_{0}$. 
\end{theorem}

\begin{proof} Let $\varepsilon >0$ be given. By Lemma \ref{lemma:invariant}, $\E|u_j|^p$ is bounded uniformly in $j$. The inequality \eqref{eq:ujRbound} holds with a constant independent of $j$. Then, choosing $\delta>0$ so that $(1+\delta) \rho_s^2<1$ in \eqref{eq:lasteqn} implies that
there exists a constant $R(s,\nu)>0$ independent of $j$ such that $\E |u_j- u_j^R|^2 \leq \varepsilon/4$.   

As in the proof of Theorem \ref{thm:convfinite}, we choose $\delta>0$ so that $\rho_*:= (1+\delta) \rho_s^2<1$. Then, with $M(R,s) \in \mathbb{N}_0$ large enough, \eqref{eq:ujmr} is replaced by
\begin{align*}
\E_{\bxi_{j-1}} | u_{j}^R - u_{j}^{M,R} |^2 
& \leq \rho_{*} \,| u_{j-1}^R - u_{j-1}^{M,R} |^2   +   \frac{\varepsilon}{4} {(1-\rho_*)}, \quad  j \geq 1.
\end{align*}
 Therefore,
\begin{align*}
\E | u_{j}^R - u_{j}^{M,R} |^2 
& \leq \rho_{*}^j \,\E | u_{0}^R - u_{0}^{M,R} |^2   +     \frac{\varepsilon}{4} (1-\rho_*)  (1-\rho_*)^{-1} = \varepsilon/4 ,\quad  j \geq 1.  
\end{align*}
This concludes our proof of convergence of the DgPC method to the invariant measure $\nu$. 
\end{proof}

Let us conclude this section with the following remark.
\begin{remark} \rm
It would be desirable to obtain that the chain (corresponding to $R=\infty$, or equivalently no support truncation) $u^M_{j+1}=F^M_s(u^M_j,\bxi_j)$ remains determinate. We were not able to do so and instead based our theoretical results on the assumption that the true distributions of interest were well approximated by compactly supported distributions. In practice, the range of $M$ has to remain relatively limited for at least two reasons. First, large $M$ rapidly involves very large computational costs; and second, the determination of measures from moments becomes exponentially ill-posed as the degree of polynomials $N$ increases. For these reasons, the  support truncation has been neglected in the following numerical section since, heuristically, for large $R$ and limited $N$, the computation of (low order) moments of distributions with rapidly decaying density at infinity is hardly affected by such a  support truncation.
\end{remark}


\section{Numerical Experiments} \label{sec:numeric}

In this section, we present several numerical simulations of our algorithm and discuss the implementation details. We consider several of the equations described in~\cite{BM13,GHM10} to show that PCE-based simulations may perform well in such settings. We mostly consider the following two-dimensional nonlinear system of coupled SDEs: 
\begin{equation}
\begin{aligned} \label{eq:system}
du(s) & = -(b_u+ a_u v(s)) u(s) \, ds + f(s) \, ds+ \sigma_u \, dW_u(s) , \\
dv(s) &= -(b_v+a_v u(s)) \, v(s) \, ds + \sigma_v \, dW_v(s),
\end{aligned}
\end{equation}
where $a_u,a_v \geq 0$, $b_u,b_v>0$ are damping parameters, $\sigma_u,\sigma_v>0$  are constants, and $W_u$ and $W_v$  are two real independent Wiener processes.

The system \eqref{eq:system} was proposed in~\cite{GHM10} to study filtering of turbulent signals which exhibit intermittent instabilities.  The performance of Hermite PCE is analyzed in various dynamical regimes by the authors in~\cite{BM13}, who conclude that truncated PCE struggle to accurately capture the statistics of the solution in the long term due to both the truncated expansion of the white noise and neglecting higher order terms, which become crucial because of the nonlinearities. For a review of the different dynamical regimes that \eqref{eq:system} exhibits, we refer the reader to~\cite{GHM10,BM13}.

For the rest of the section, $t \in \R$ stands for the endpoint of the time interval while $\Delta t=t/n$ denotes the time-step after which restarts occur at $t_j= j \Delta t $. Moreover, $K$ denotes the number of basic random variables $\xi_i$ used in the truncation of the expansion \eqref{eq:expansion_W}. In the presence of two Wiener processes it will denote the total number of variables. We slightly change the previous notation and let $N$ and $L$ denote the maximum degree of the polynomials in $\bxi$ and $u_j$, respectively. Recall that $u_j$ is the projected PC solution at $t_j$.   Furthermore, following~\cite{Luo,HLRZ06} we choose the orthonormal bases for $L^2[t_{j-1},t_{j}]$ as
\begin{align} \label{eq:cosines}
m_1(s) =\frac{1}{\sqrt{t_{j}-t_{j-1}}}, \quad m_i(s) = \sqrt{\frac{2}{t_{j}-t_{j-1}}} \cos \left( \frac{(i-1) \pi s }{t_{j}-t_{j-1}} \right), \quad i \geq 2, \quad s \in [t_{j-1},t_{j}].
\end{align}
Other possible options for $m_i(s)$ include sine functions, a combination of sines and cosines, and wavelets. We refer the reader to ~\cite{MNGK04} for details on the use of wavelets to deal with discontinuities for random inputs. 

Assuming that the solution of \eqref{eq:system} is square integrable,  we utilize intrusive Galerkin projections in order to establish the following equations for the coefficients of the PCE:
\begin{align*}  
\dot{u}_{\alpha} & = -b_u \, u_{\alpha}+ a_u \, (uv)_{\alpha} + f \delta_{\alpha \boldsymbol{0}}+ \sigma_u \, E[\dot{W}_u T_{\alpha}] , \\
\dot{v}_{\alpha} &= -b_v\, v_{\alpha}+ a_v \, (uv)_{\alpha}  + \sigma_v \,E[\dot{W}_v T_{\alpha}].
\end{align*}
This system of ODEs is then solved by either  a second- or a fourth-order time integration method. Finally, we note that other methods are also available to compute the PCE coefficients
such as nonintrusive projection and collocation methods~\cite{XH05,LeMK10,GS91,Xiu_book,SG04}.

In most of our numerical examples, the dynamics converge to an invariant measure. To demonstrate the convergence behavior of our algorithm, we compare our results to exact second order statistics or Monte Carlo simulations with sufficiently high sampling rate $N_{samp}$ (e.g., Euler--Maruyama or weak Runge--Kutta methods~\cite{Kloeden}) where the exact solution is not available. Comparisons involve  the following relative pointwise errors:
\begin{align*}
\epsilon_{\mbox{mean}}(s) :=\bigg  \vert \frac{\mu_{\mbox{pce}}(s)- \mu_{\mbox{exact}}(s)}{\mu_{\mbox{exact}}(s)} \bigg \vert, \quad  \epsilon_{\mbox{var}}(s) :=\bigg  \vert \frac{\sigma^2_{\mbox{pce}}(s)- \sigma^2_{\mbox{exact}}(s)}{\sigma^2_{\mbox{exact}}(s)} \bigg \vert,
\end{align*}
where $\mu,\sigma^2$ represents the mean and variance. In the following figures, we plot the evolution of mean, variance, $\epsilon_{\mbox{mean}}$ and $\epsilon_{\mbox{var}}$ using the same legend. In addition, we exhibit the evolution of higher order cumulants, which will be denoted by $\kappa$ to demonstrate the convergence to steady state as time grows.  Finally, in our numerical computations, we make use of the C++ library UQ Toolkit~\cite{DNPKGL04}.

\subsection{Numerical examples}

\begin{exmp} \rm \,
As a first example, we consider the one-dimensional Ornstein-Uhlenbeck (OU) process
\begin{align} \label{eq:OU}
 d v(s) = -b_v \, v(s) \, ds + \sigma_v \, dW_v(s), \quad s \in [0,t], \quad v(0) = v_0,
\end{align}
on $[0,3]$ with the parameters $b_v=4,\sigma_v=2, v_0=1$. We first present the results of Hermite PCE.

\begin{figure}[!htb]
\centering
\subfigure[mean]{
\includegraphics[scale=0.24]{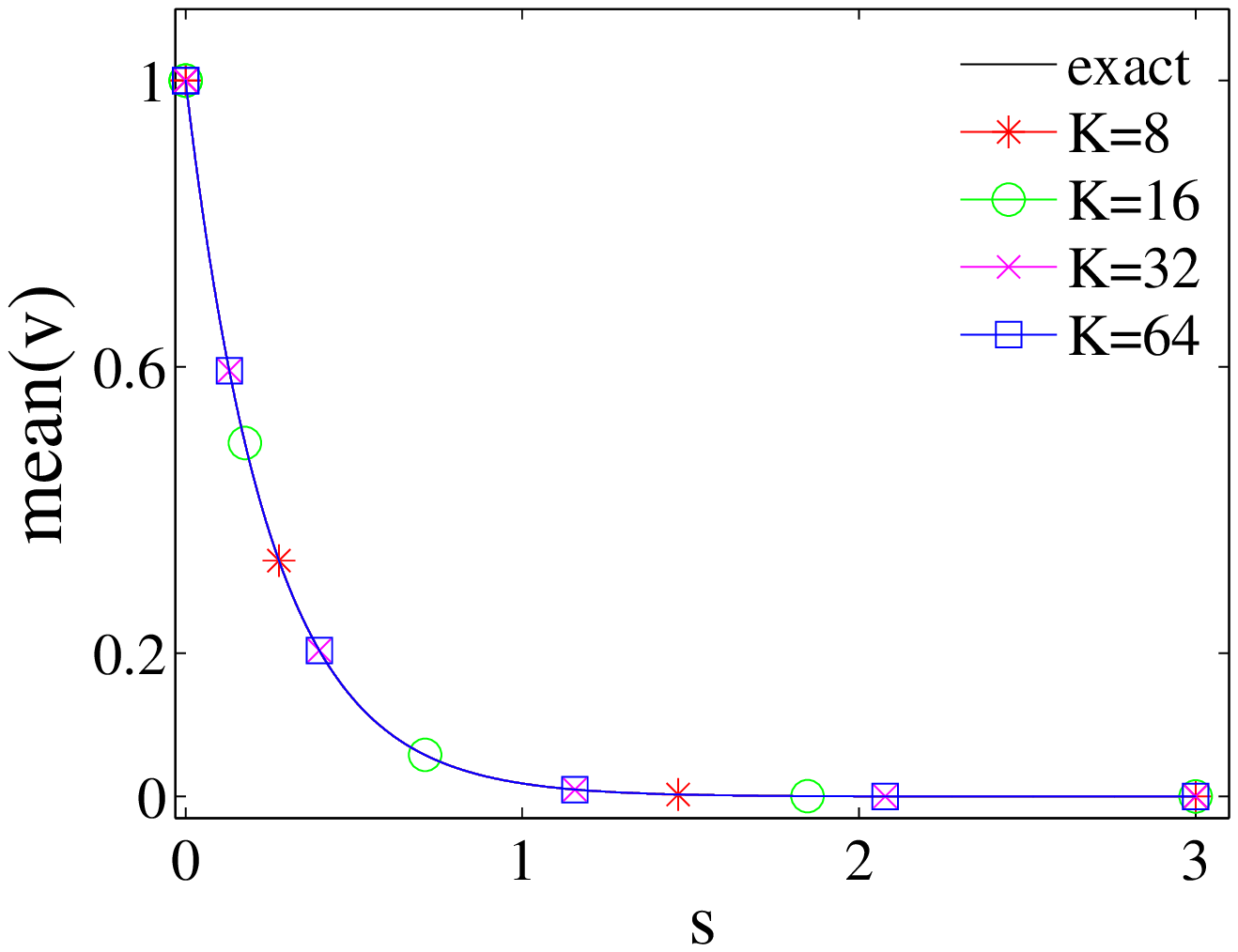}}
\subfigure[variance]{
\includegraphics[scale=0.24]{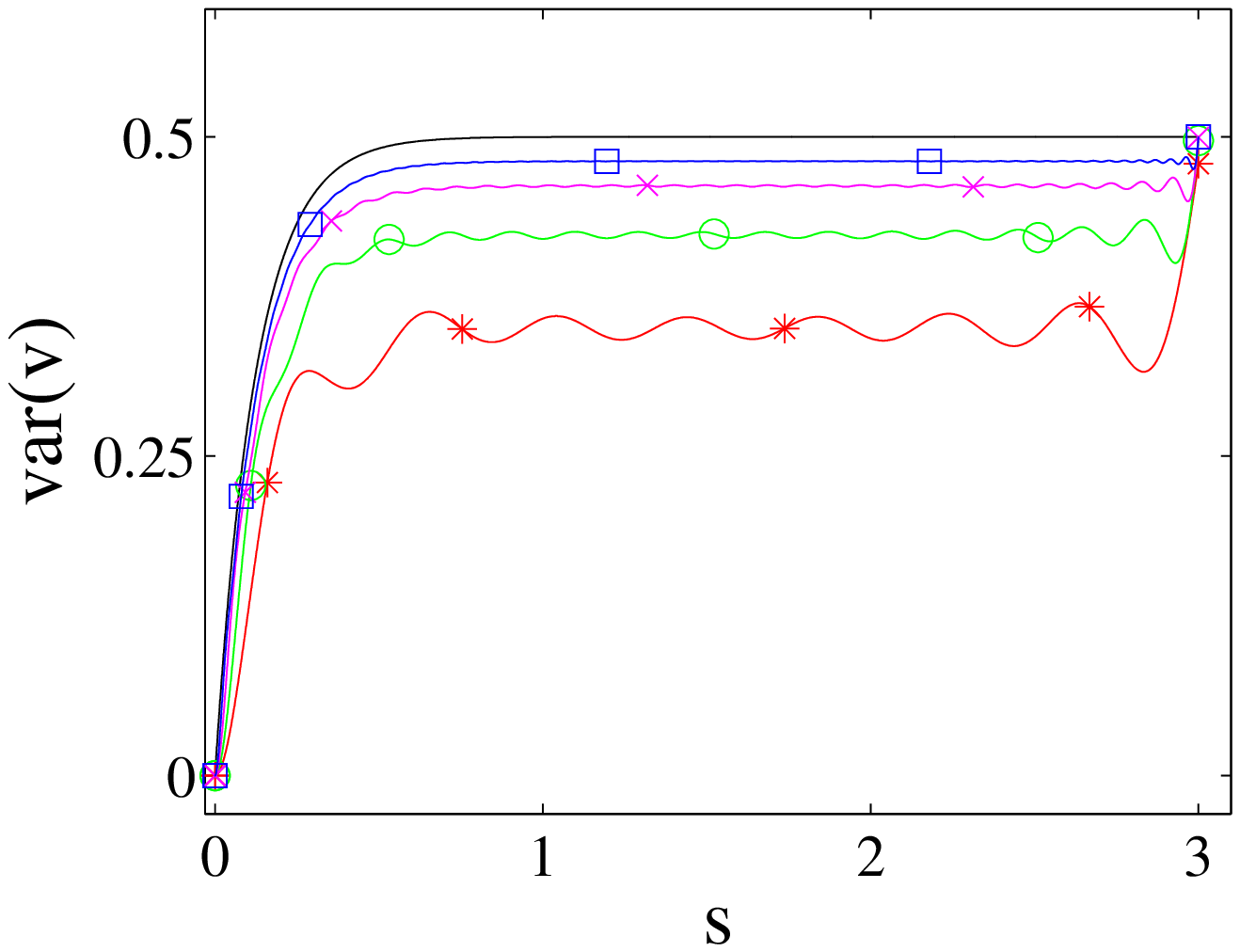}} 
\subfigure[$\epsilon_{\mbox{mean}}$]{
\includegraphics[scale=0.24]{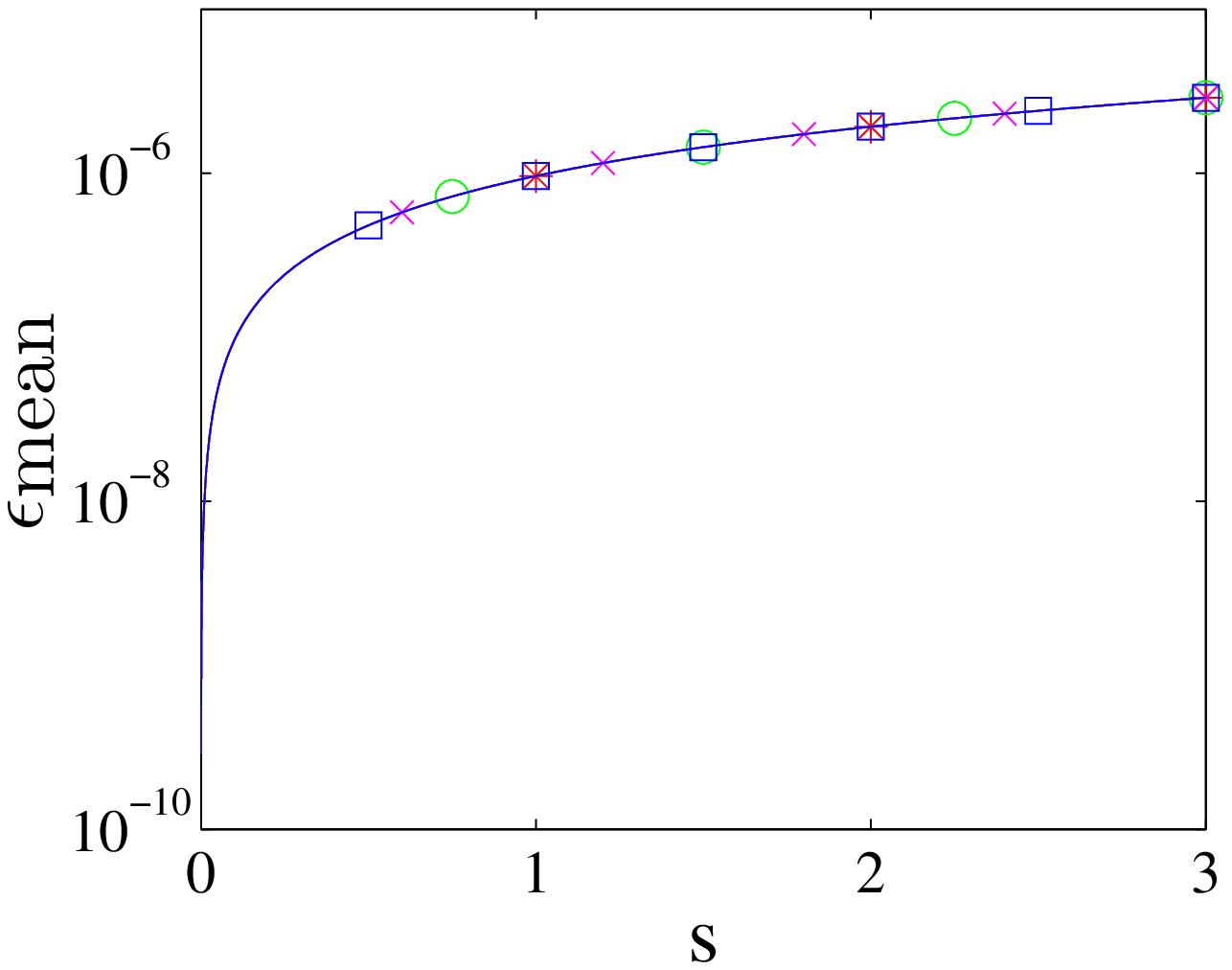}}
\subfigure[$\epsilon_{\mbox{var}}$]{
\includegraphics[scale=0.24]{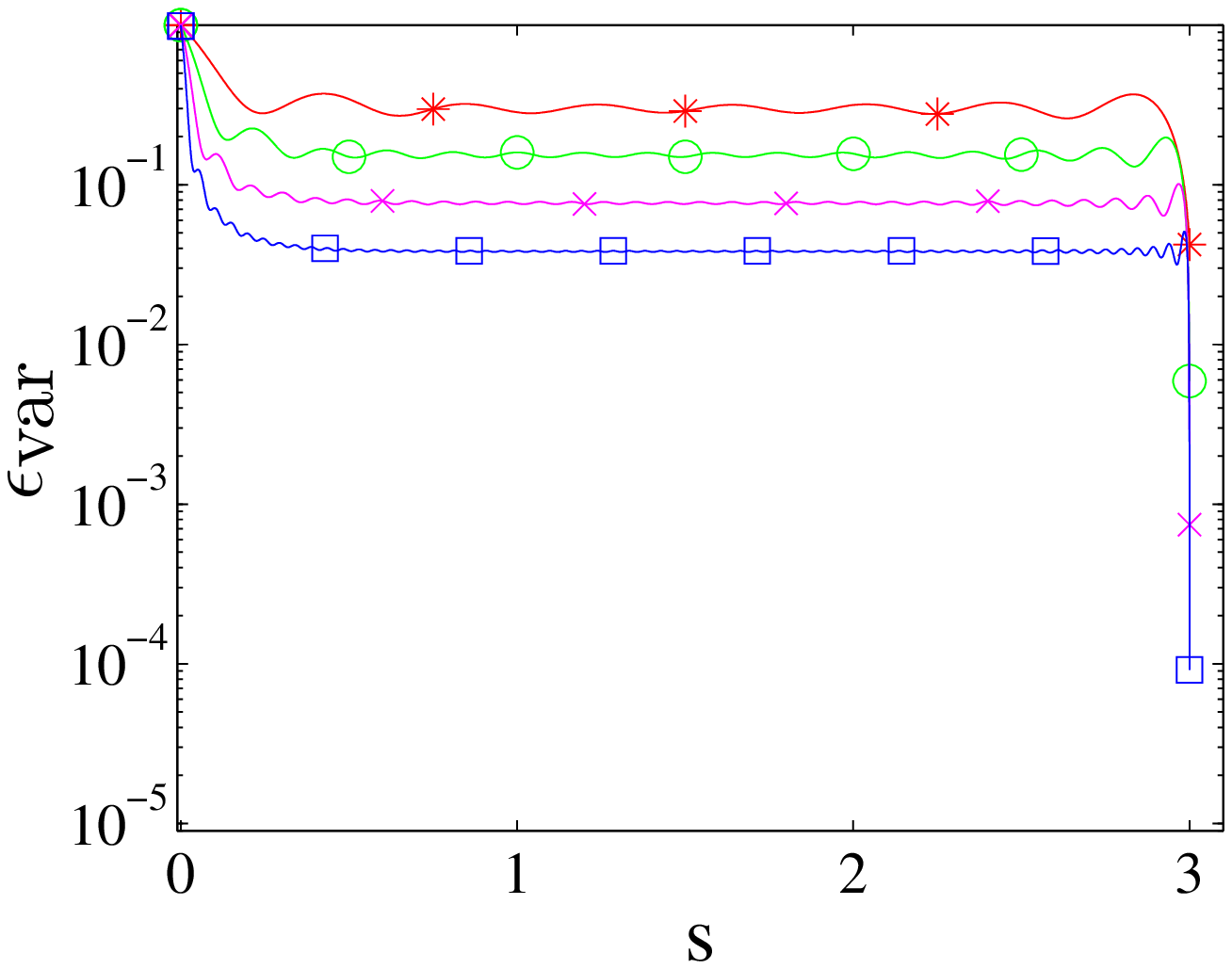}} 
\vspace{-0.2cm}
\caption{Hermite PCE.}
\label{fig:Ex1_1}
\end{figure} 

Figure \ref{fig:Ex1_1} shows that Hermite PC captures the mean accurately but  approximation for the variance is accurate only at the endpoint.  This is a consequence of the expansion in \eqref{eq:exp_W_err}, which violates the Markov property of Brownian motion and is inaccurate for all $0<s<t$, while it becomes (spectrally) accurate at the endpoint $t$. Using frequent restarts, these oscillations are significantly attenuated as expected; see Figure \ref{fig:Ex1_2b}. This oscillatory behavior could also be alleviated by expanding the Brownian motion as in \eqref{eq:exp_W_err} on subintervals of $(0,1)$. Note that the global basis functions $T_{\alpha}^t(\bxi)$ may also be filtered by taking the conditional expectation $\E[T_{\alpha}^t(\bxi) | \mathcal{F}_s^W]$,  where $\mathcal{F}_s^W$ is the $\sigma$-algebra corresponding to Brownian motion up to time $s$~\cite{Luo, MR04}. We do not pursue this issue here but note that the accuracy of the different methods should be compared at the end of the intervals of discretization of Brownian motion.

Next, we demonstrate numerical results for DgPC taking  $N=1$, $L=1$ and a varying $K=4,6,8$. Figure \ref{fig:Ex1_2} shows that as the exact solution converges to a steady state, our algorithm captures the second order statistics accurately even with a small degrees of freedom utilized in each subinterval. Moreover, although we do not show it here, it is possible, by increasing $L$, to approximate the higher order zero cumulants $\kappa_i, i=3,4,5,6$, with an error on the  order of machine precision,  which implies that the algorithm converges numerically to a Gaussian invariant measure.

\begin{figure}[!htb]
\centering
\subfigure[mean]{
\includegraphics[scale=0.24]{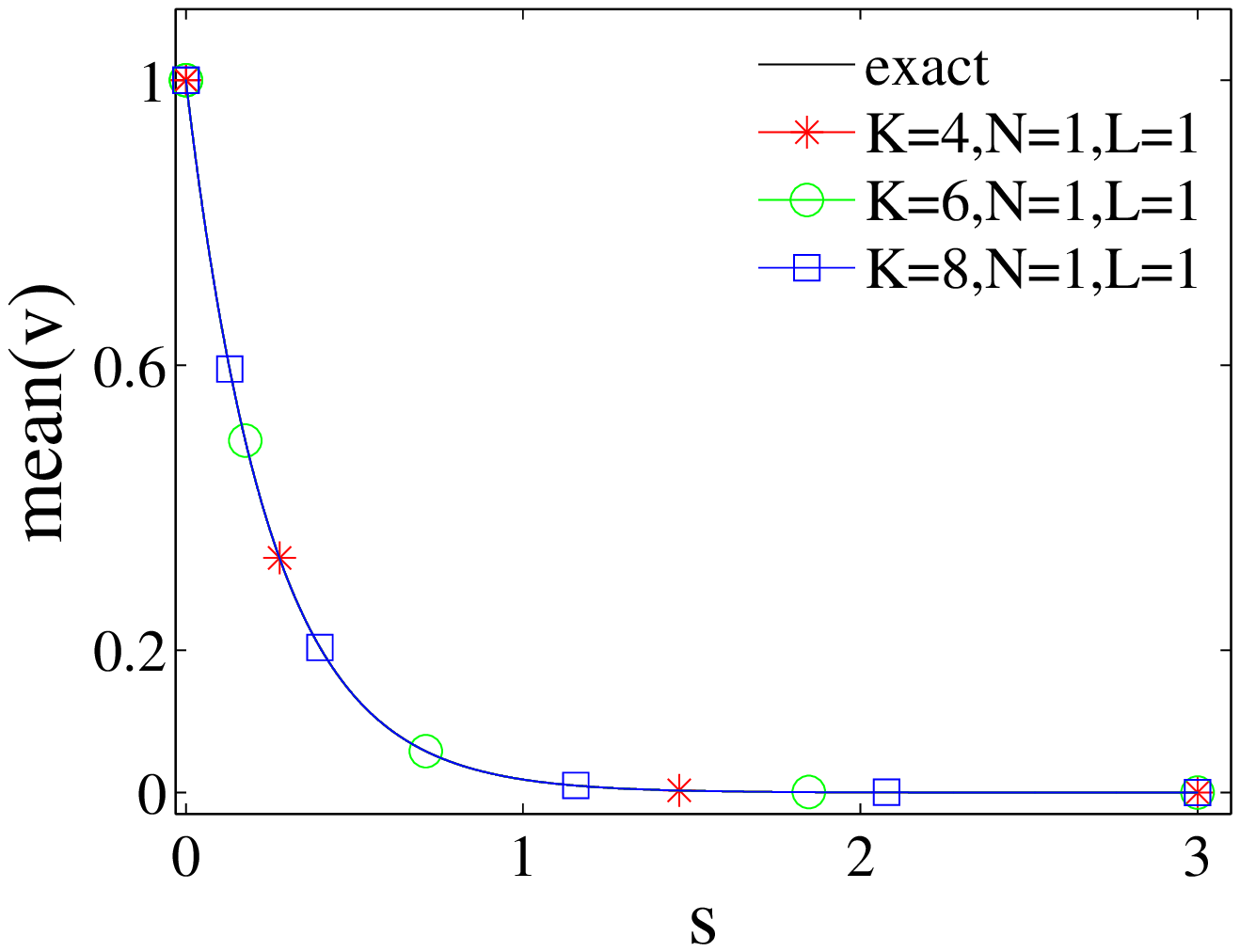}}
\subfigure[variance]{ \label{fig:Ex1_2b}
\includegraphics[scale=0.24]{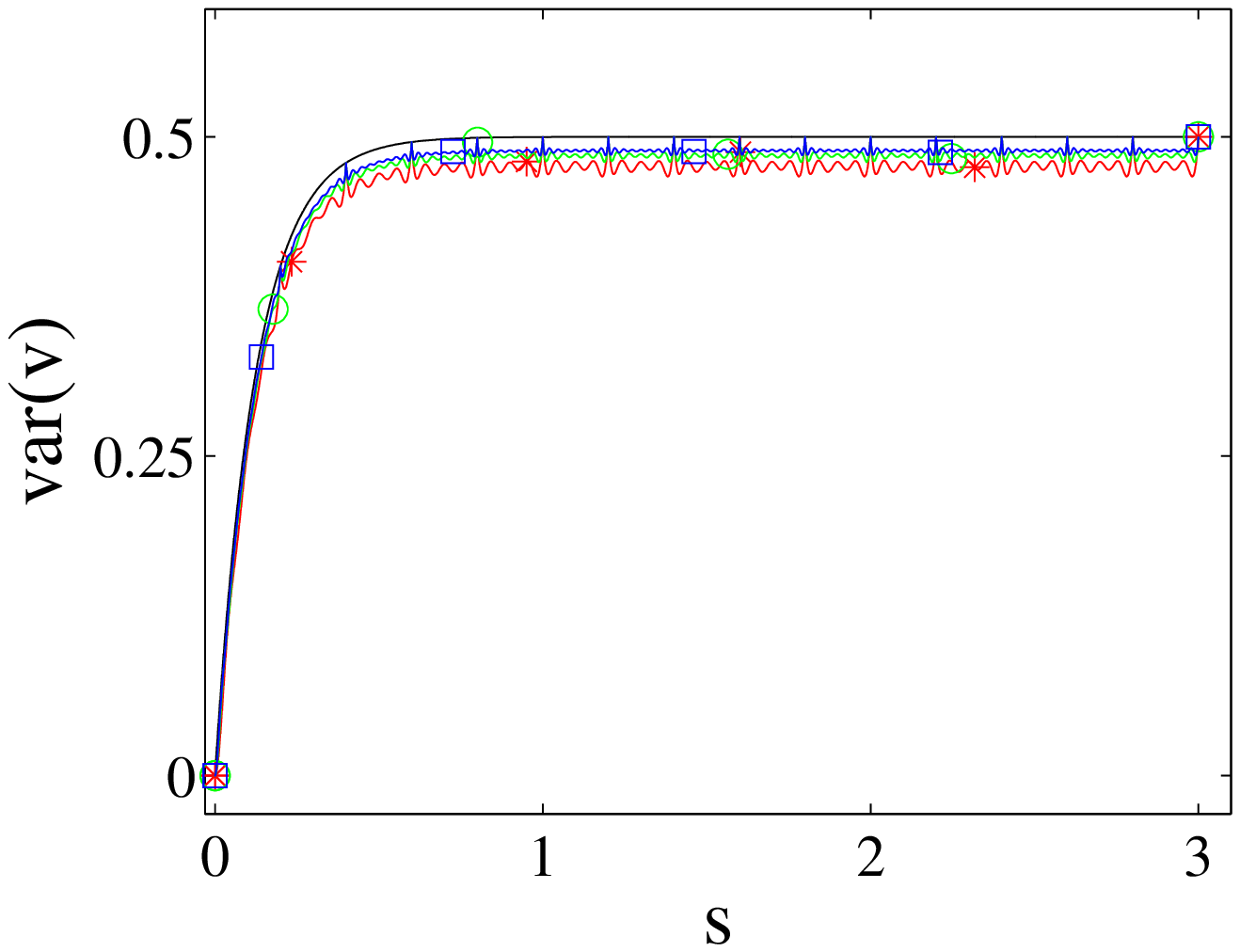}} 
\subfigure[ $\epsilon_{\mbox{mean}}$ ]{
\includegraphics[scale=0.24]{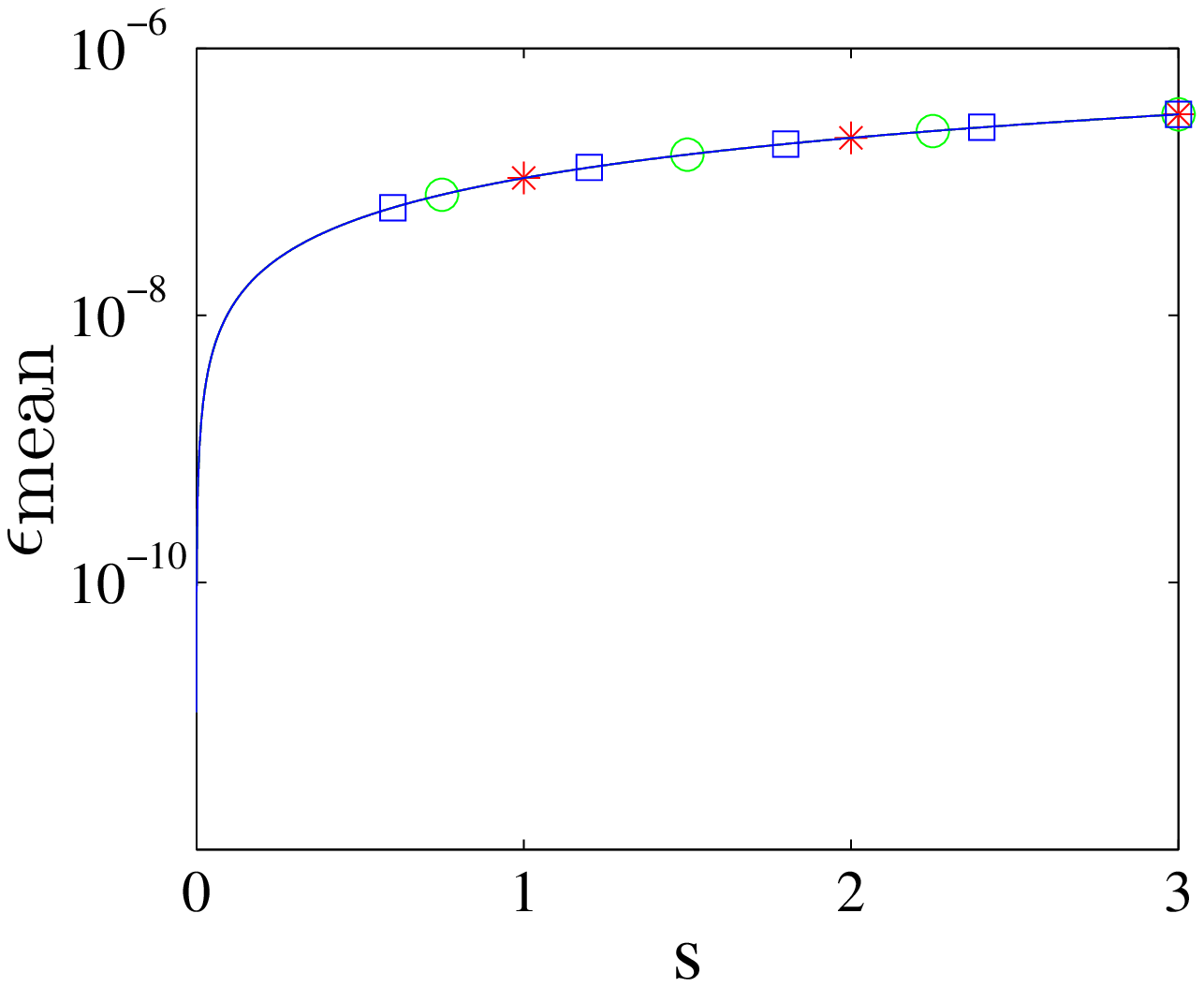}} 
\subfigure[$\epsilon_{\mbox{var}}$]{ 
\includegraphics[scale=0.24]{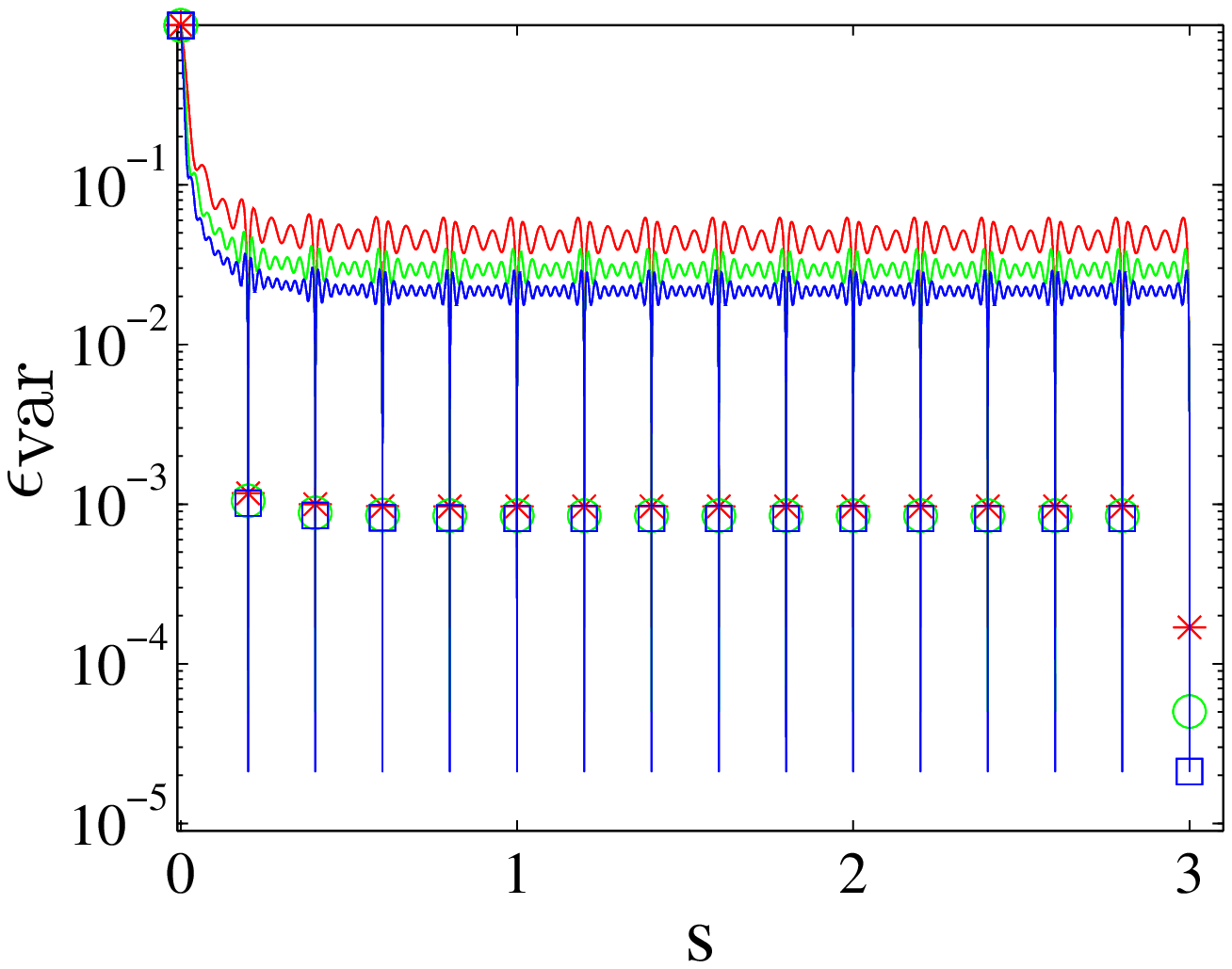}}
\vspace{-0.2cm}
\caption{DgPC $\Delta t = 0.2$.}
\label{fig:Ex1_2}
\end{figure}

In Figure \ref{fig:Ex1_3}, we illustrate an algebraic convergence of the variance  in terms of the dimension $K$ for $t=3$ and $t=15$.  DgPC uses the same size of interval $\Delta t =0.2$ for both cases.  For comparison, we also include the convergence behavior of Hermite PCE. Values of $K$ are taken as $(2,4,6,8)$  for Figures \ref{fig:Ex1_3a} and \ref{fig:Ex1_3c} and $K=(8,16,32,64)$ for Figures \ref{fig:Ex1_3b} and \ref{fig:Ex1_3d}. We deduce that our algorithm maintains an algebraic convergence of order $O(K^{-3})$ for both $t=3$ and $t=15$ whereas the convergence behavior of Hermite PCE drops from $O(K^{-3})$ to $O(K^{-2})$ as time increases. This confirms the fact that the degrees of freedom required for standard PCE to maintain a desired accuracy should increase with time.

\begin{figure}[!htb]
\centering
\subfigure[DgPC ($t=3$)]
{ \label{fig:Ex1_3a}
\includegraphics[scale=0.24]{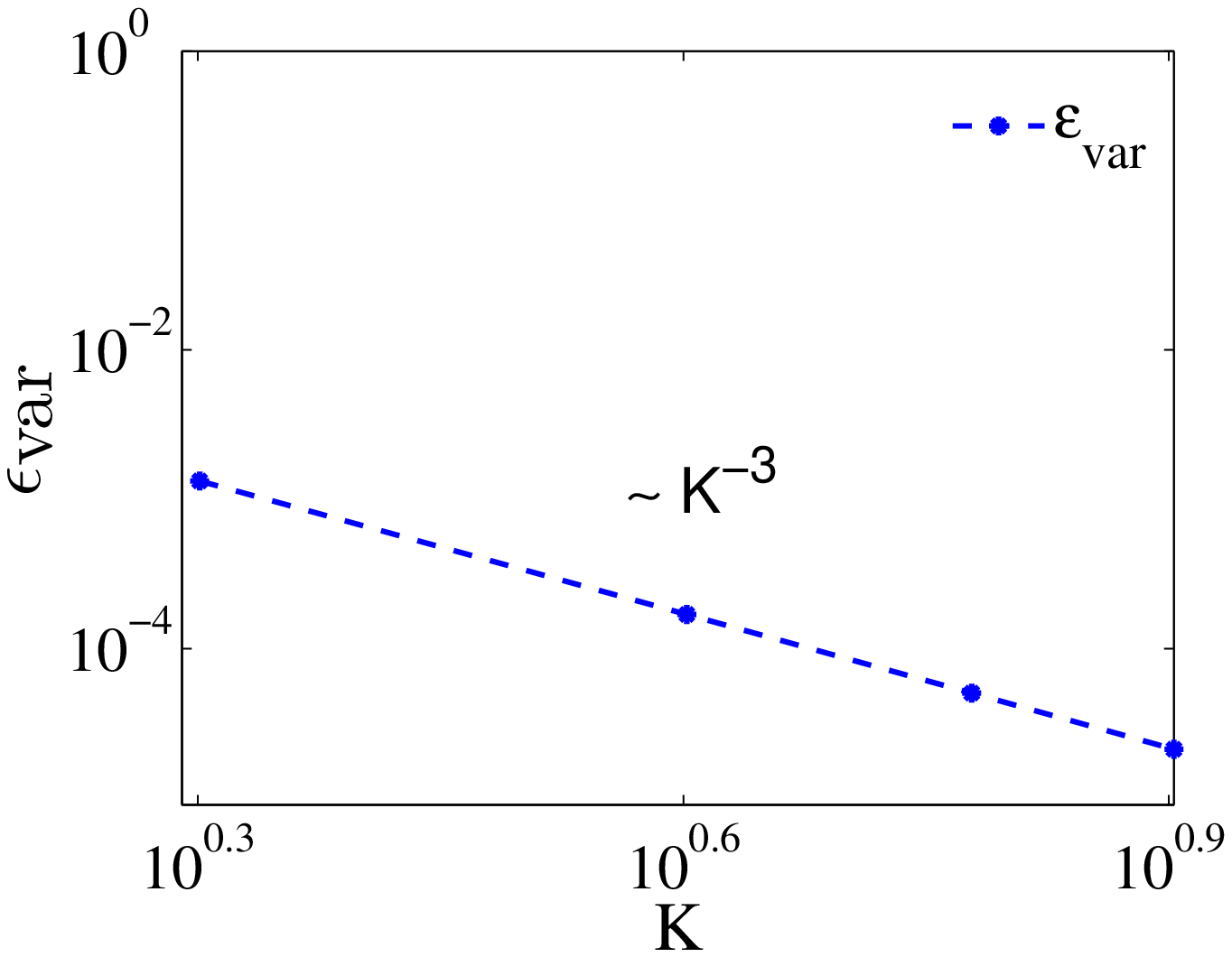}}
\subfigure[Hermite PC ($t=3$)]{\label{fig:Ex1_3b}
\includegraphics[scale=0.24]{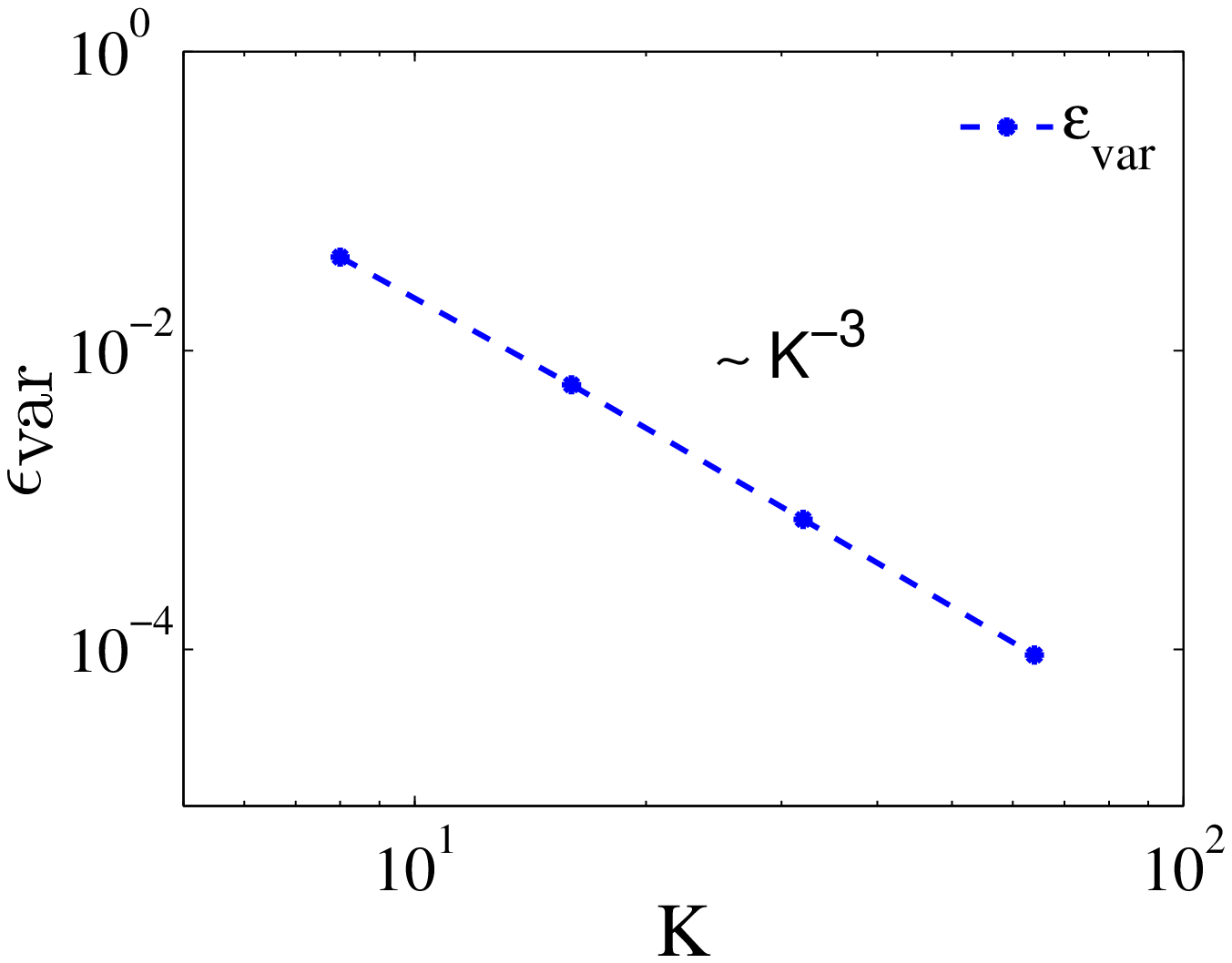}}
\subfigure[DgPC ($t=15$)]{\label{fig:Ex1_3c}
\includegraphics[scale=0.24]{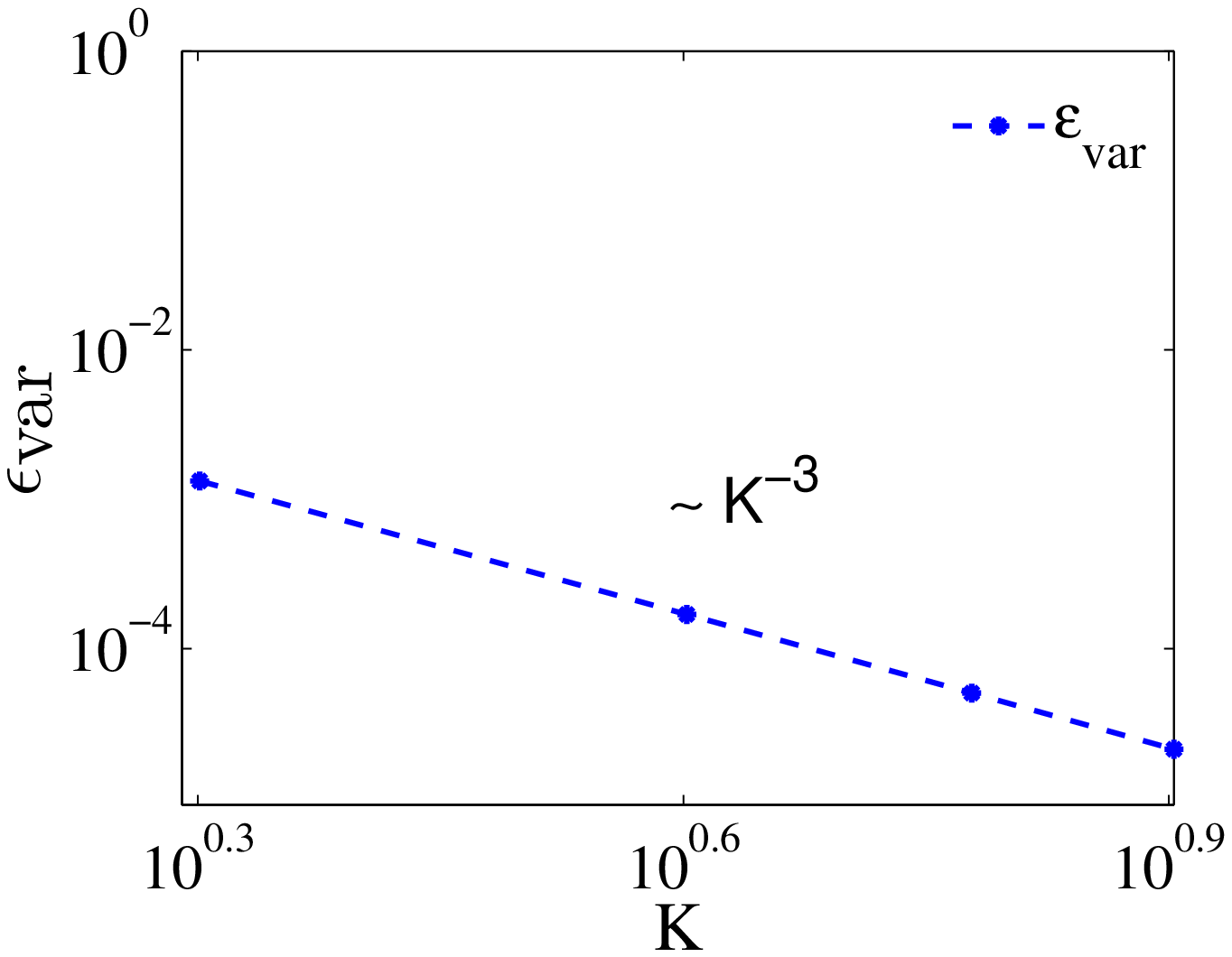}}
\subfigure[Hermite PC ($t=15$)]{\label{fig:Ex1_3d}
\includegraphics[scale=0.24]{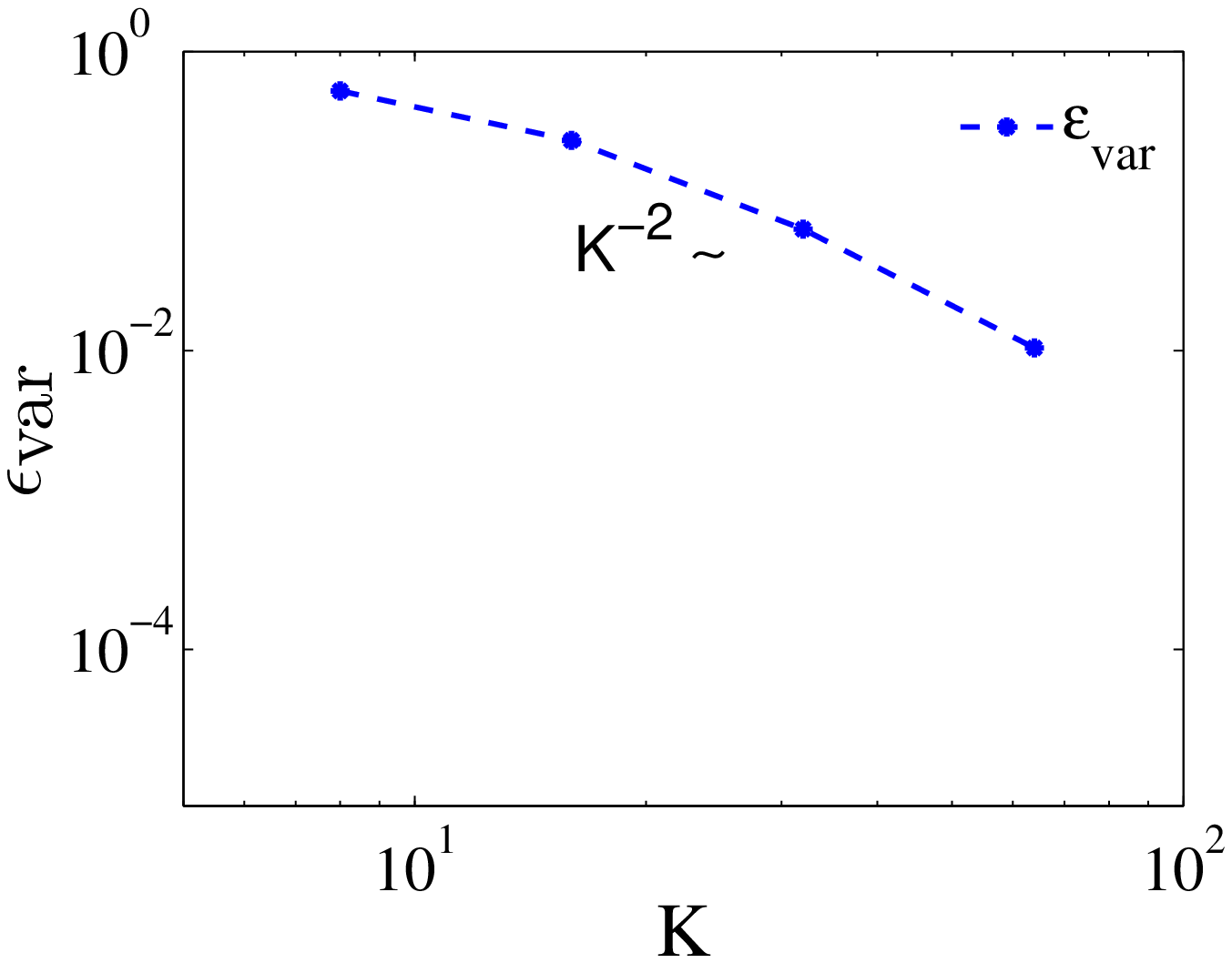}}
\vspace{-0.2cm}
\caption{K convergence.}
\label{fig:Ex1_3}
\end{figure}

\end{exmp}

 \begin{exmp} \rm \,
  We now introduce a nonlinearity in the equation so that the damping term includes a cubic component:
\begin{align*}
 d v = - (v^2 + 1)v \,ds + \sigma_v \, dW_v, \quad v(0)=1.
\end{align*}

Figure \ref{fig:Ex2_1} displays several numerical simulations for the degree of polynomials in $\bxi$ given by $N=1,2,3$ and $\sigma_v=2$. This is compared with the weak Runge--Kutta method using $N_{samp}=200000$ and $dt=0.001$. 
\begin{figure}[!htb]
\centering
\subfigure[mean]{
\includegraphics[scale=0.24]{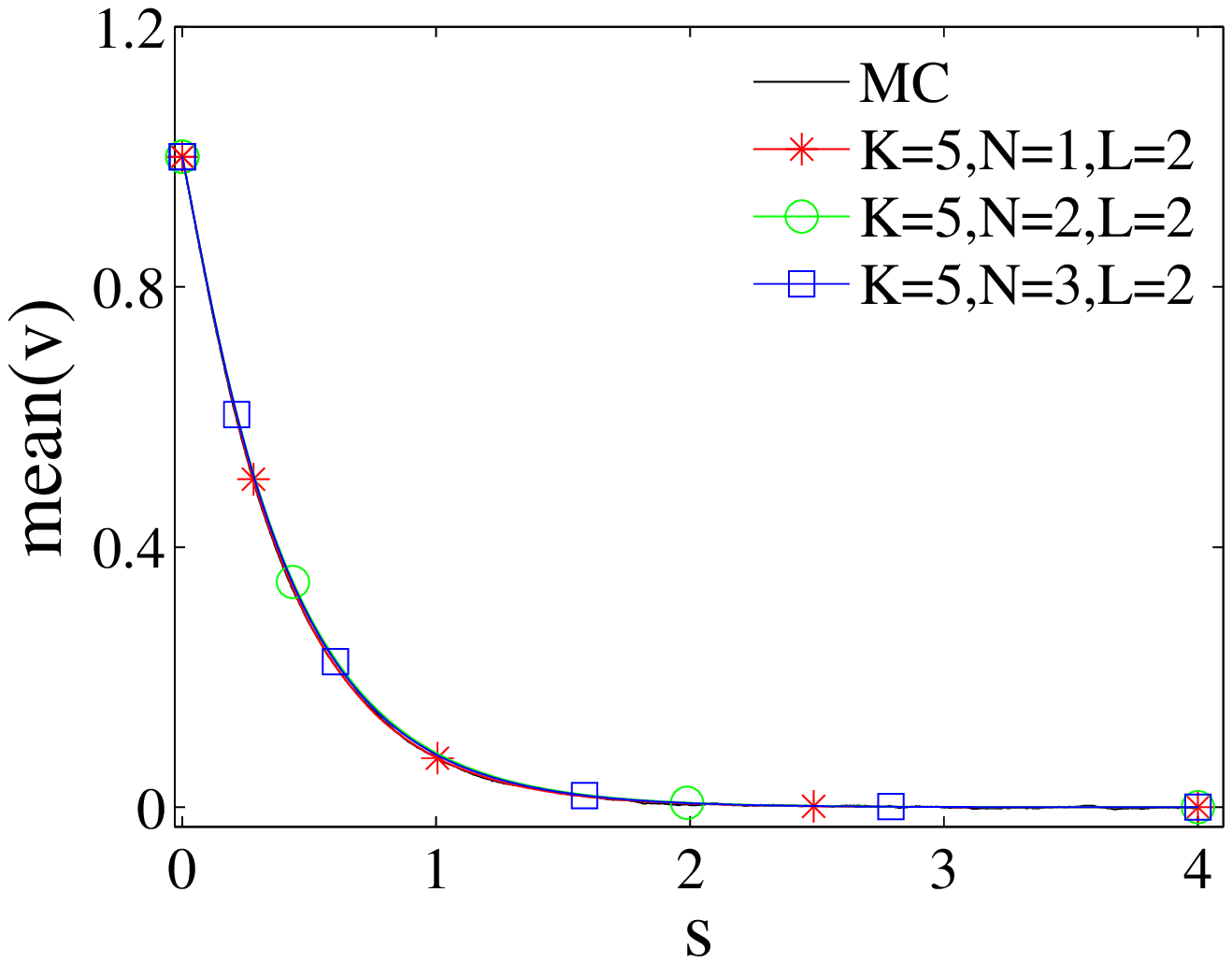}}
\subfigure[variance]{
\includegraphics[scale=0.24]{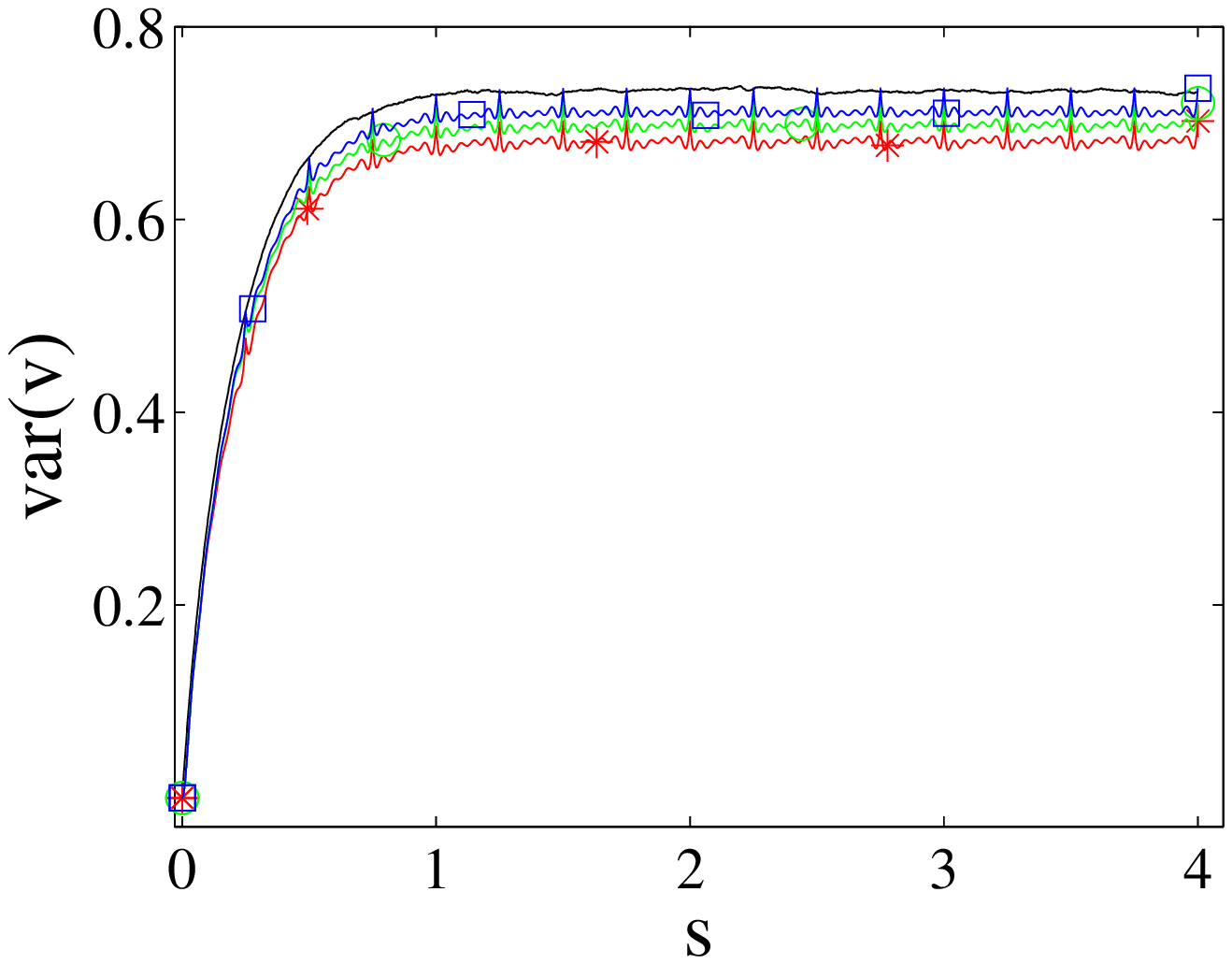}} 
\subfigure[ cumulants]{
\includegraphics[scale=0.24]{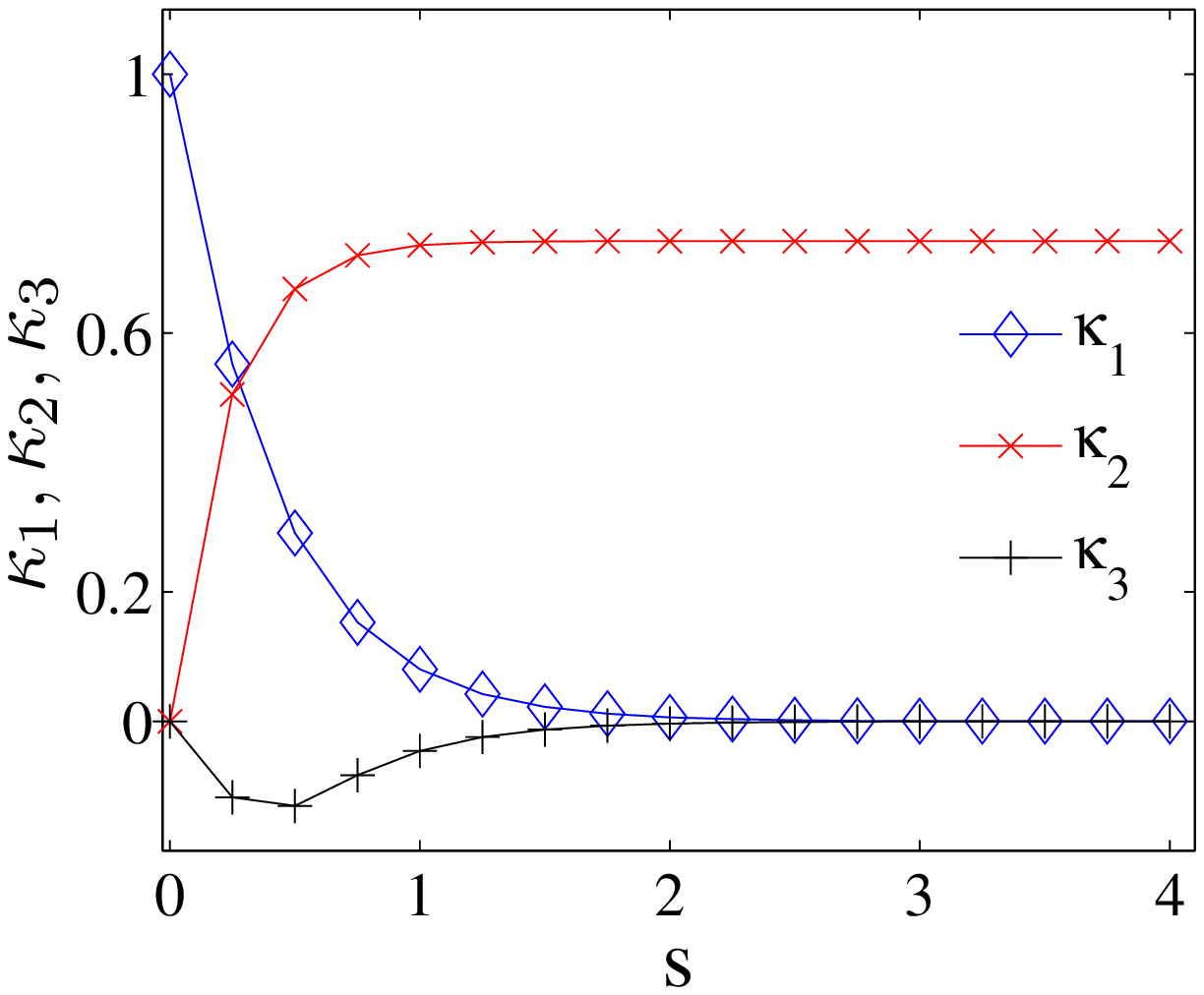}}
\subfigure[cumulants]{
\includegraphics[scale=0.24]{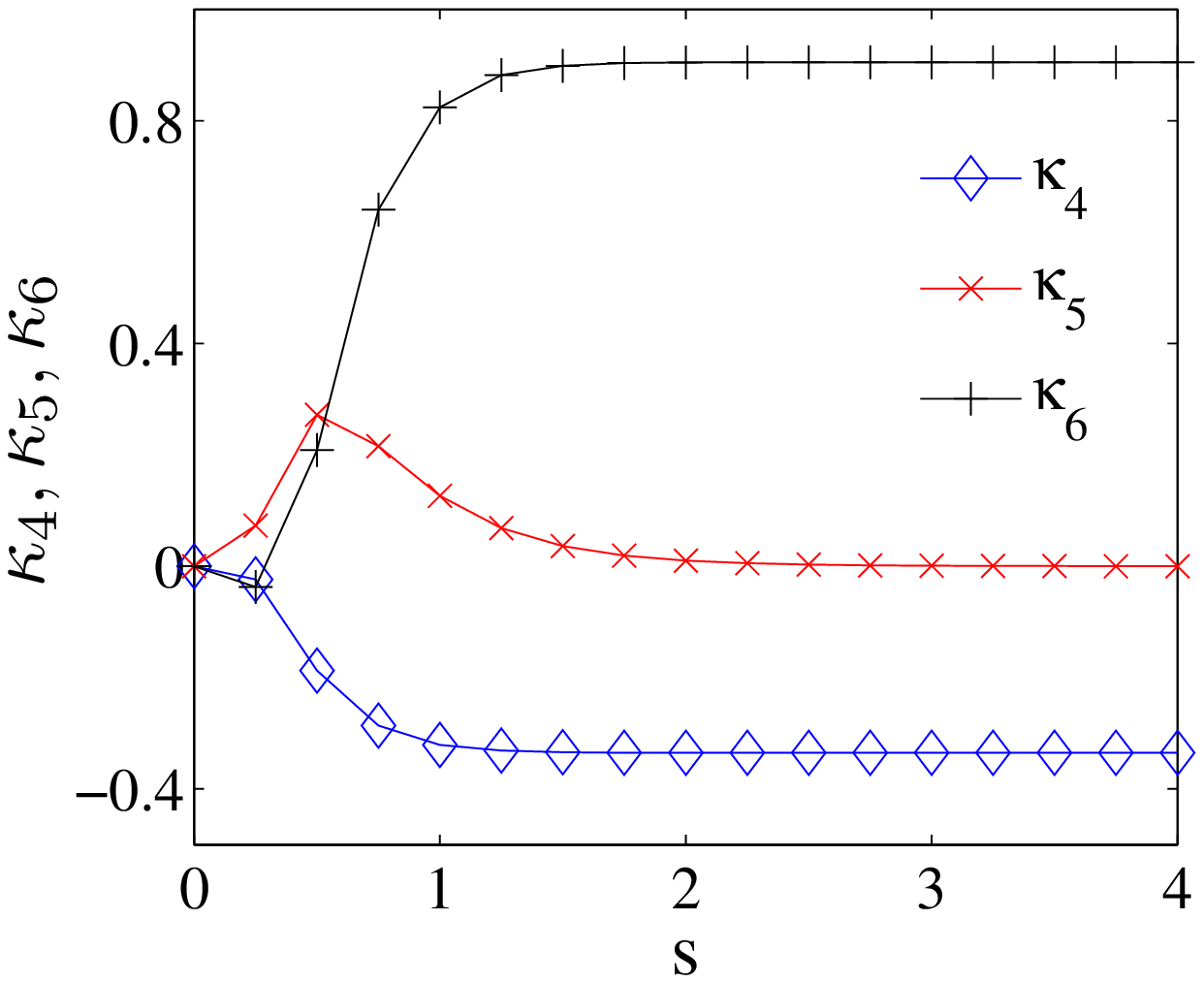}} 
\vspace{-0.2cm}
\caption{DgPC $ \Delta t =0.25$.}
\label{fig:Ex2_1}
\end{figure} 
We observe that increasing $N$ improves the accuracy of the solution to this nonlinear equation as expected. Moreover, Table \ref{table:cubic} presents a comparison for statistical cumulants between our method with $K=5$, $N=2$, $L=4$ and Monte Carlo (MC) at time $t=4$. The (stationary) cumulants of the invariant measure are also estimated by solving a standard Fokker--Planck (FP) equation \cite{OB03,F75} for the invariant measure. We conclude this example by noting that the level of accuracy of approximations in DgPC for cumulants is similar to that of the MC method. 

\begin{table}[!htb]
\begin{minipage}[b]{1\linewidth}
\centering
\renewcommand{\tabcolsep}{0.1cm}
\renewcommand{\arraystretch}{1.00}
 \begin{tabular} { |r | r | r |r|r|r|r|}
  \hline
   & $\kappa_1$&$\kappa_2$ &$\kappa_3$  & $\kappa_4$&$\kappa_5$ & $\kappa_6$\\ \hline
   DgPC & 3.58E-5  &   7.33E-1 & -2.08E-5  & -3.37E-1&  6.02E-5  &9.35E-1 \\ 
   MC & -2.53E-3&   7.33E-1    &   7.85E-4& -3.38E-1 &-3.34E-3 &  9.58E-1\\
    FP &  0&    7.33E-1 &   0 &-3.39E-1 &0  &9.64E-1\\ 
    \hline  
  \end{tabular}
  \caption{Cumulants.}
  \label{table:cubic}
\end{minipage}
\end{table}
\end{exmp}

 \begin{exmp} \rm \,
  As a third example we consider an OU process \eqref{eq:OU} in which the damping parameter is random and uniformly distributed in $[1,3]$, i.e. $b_v \sim U(1,3)$.  This is an example of non-Gaussian dynamics that may be seen as a coupled system for $(v,b_v)$ with $db_v=0$.

 We consider a time domain $[0,8]$ and divide it into $n=40$ subintervals. The initial condition is normally distributed $v_0 \sim N(1,0.04) \indep W_v$ and $\sigma_v=2$. In the next figure, we compare second order statistics obtained by our method to Monte Carlo method for which we use the Euler--Maruyama scheme with the time step $dt = 0.002$ and sampling rate $N_{samp}= 1000 \times 1000$ implying $10^6$ samples in total. We stress again that this problem is essentially two-dimensional since the damping is random. 

\begin{figure}[!htb]
\centering
\subfigure[mean]{
\includegraphics[scale=0.24]{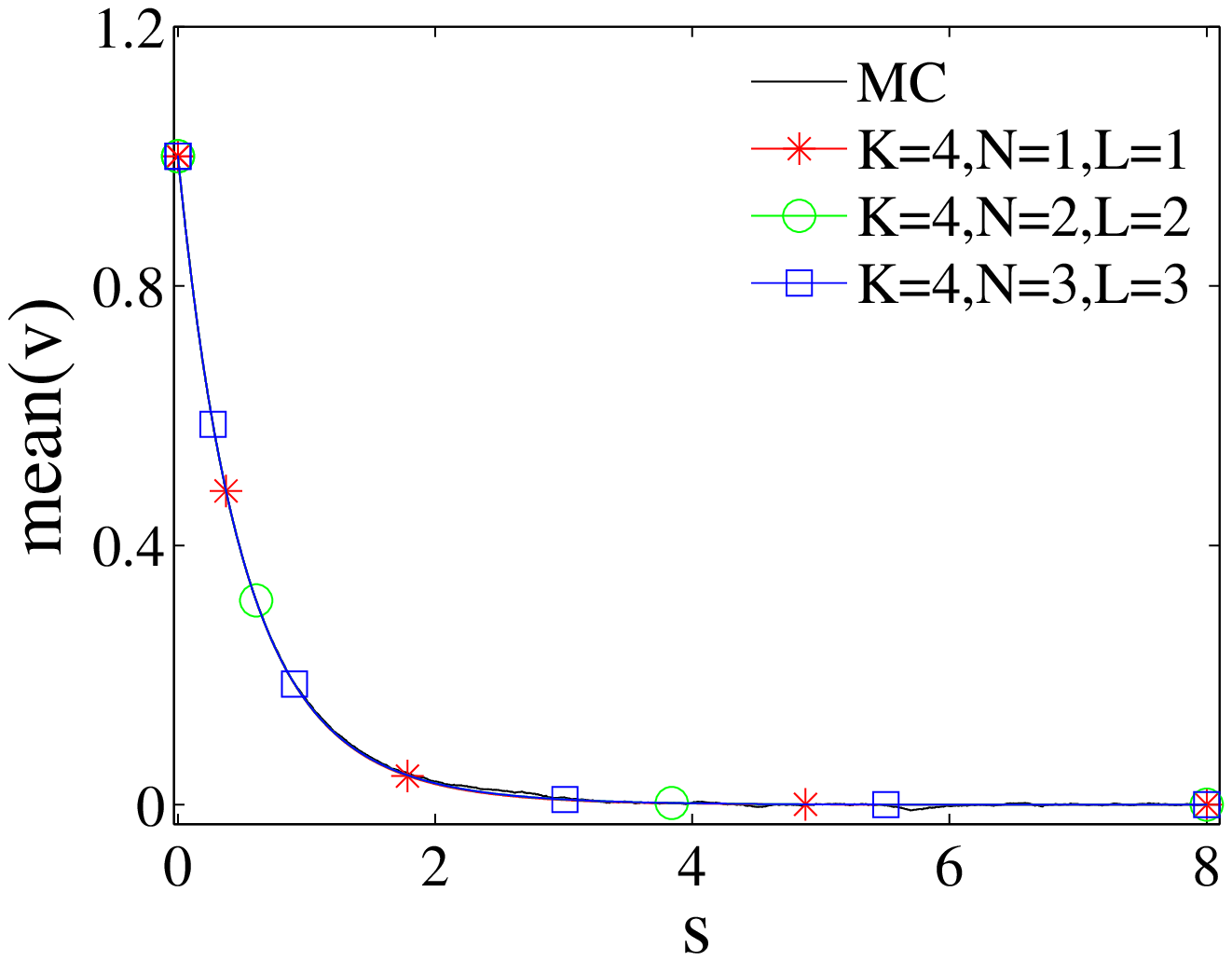}}
\subfigure[variance]{
\includegraphics[scale=0.24]{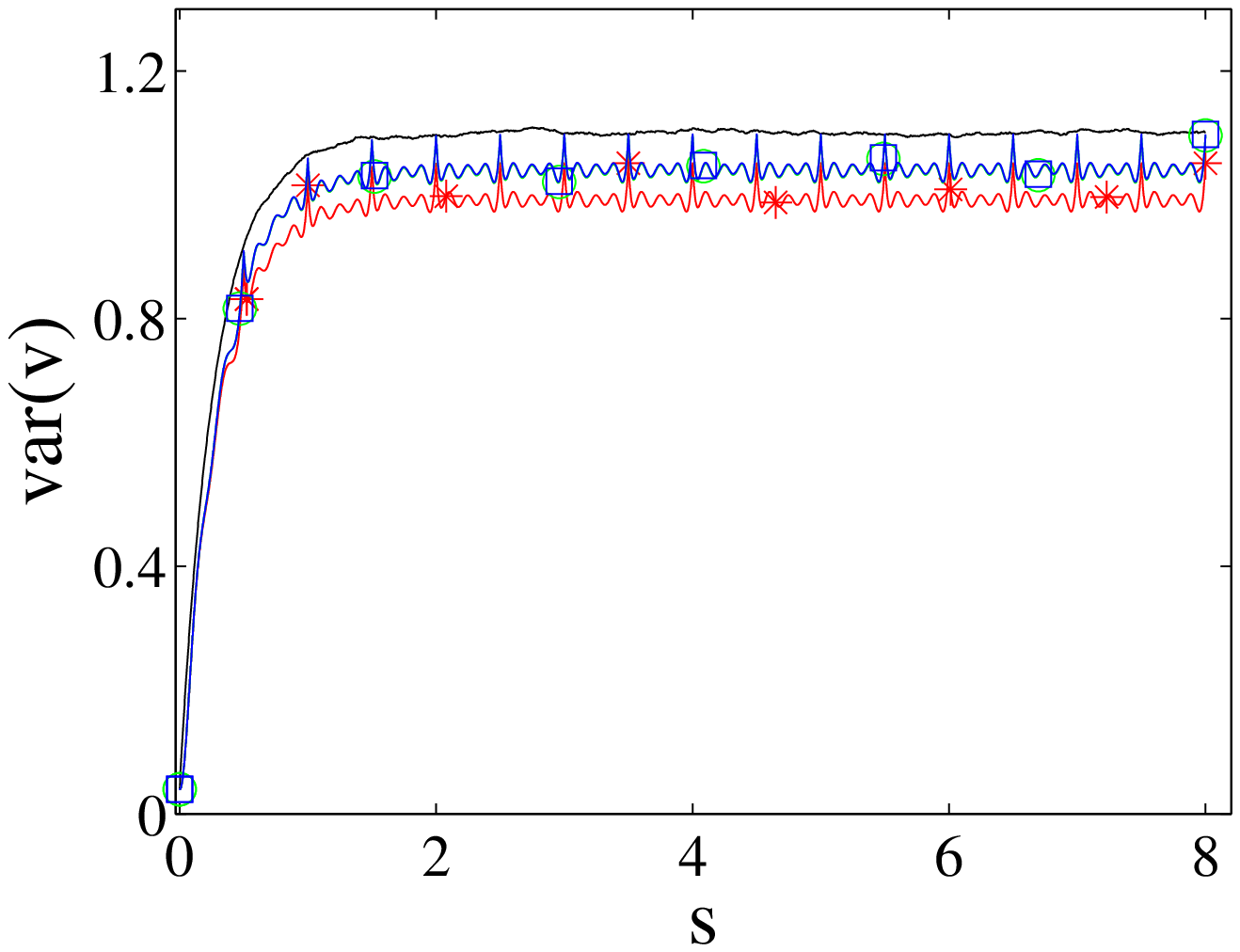}} 
\subfigure[ cumulants]{ \label{fig:Ex3_1c}
\includegraphics[scale=0.24]{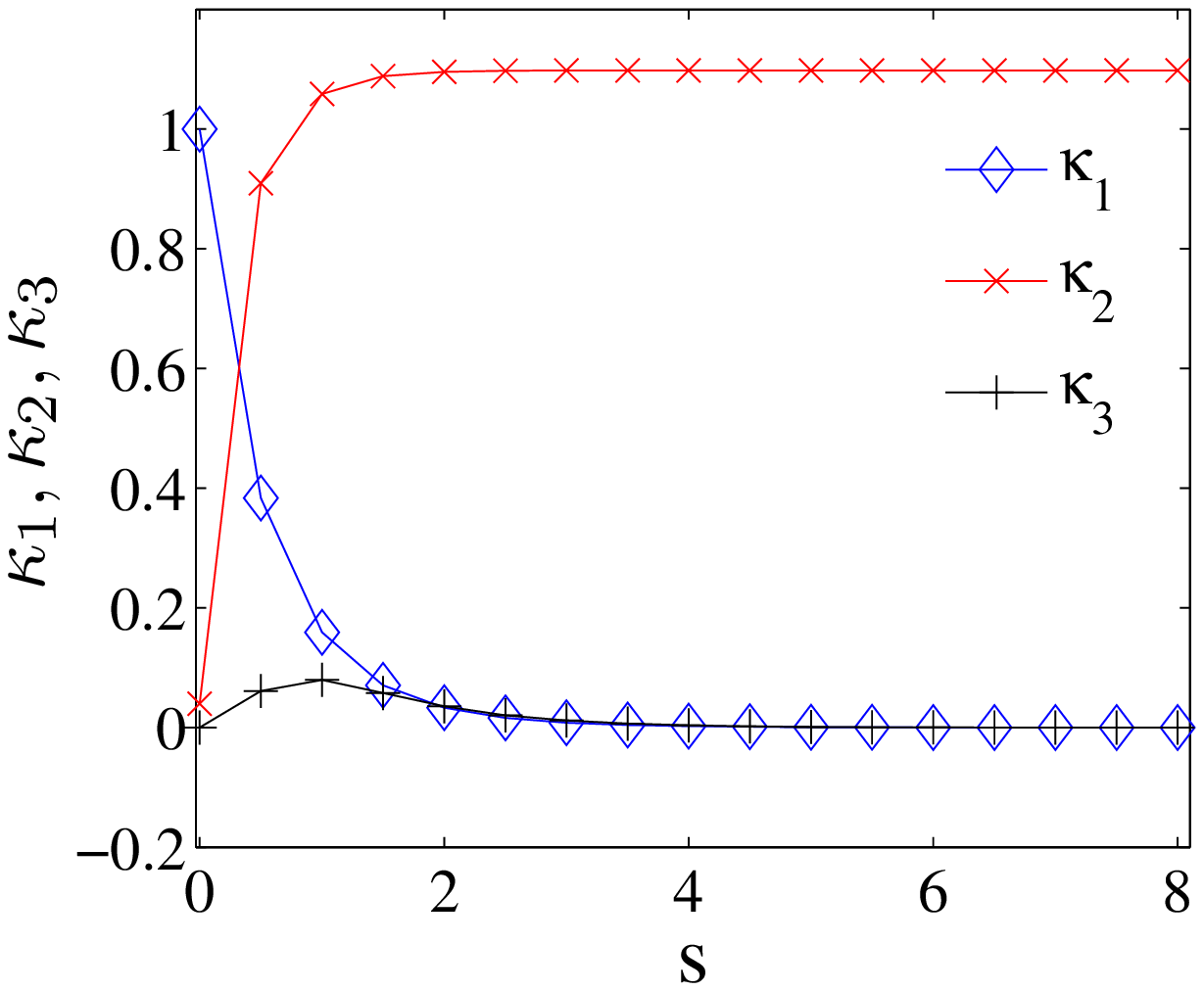}}
\subfigure[ cumulants]{ \label{fig:Ex3_1d}
\includegraphics[scale=0.24]{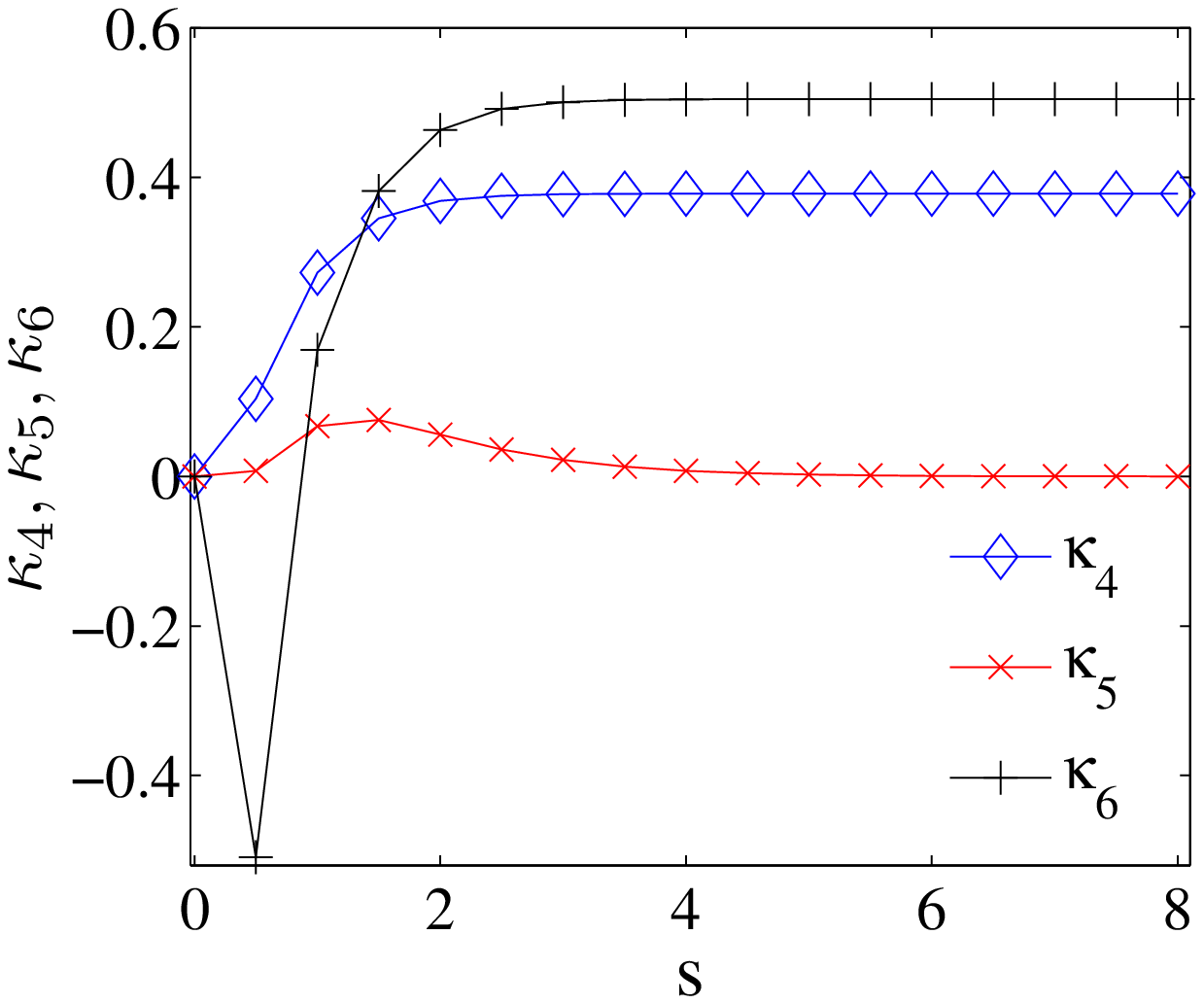}} 
\vspace{-0.2cm}
\caption{DgPC $\Delta t =0.2$.}
\label{fig:Ex3_1}
\end{figure} 

As expected, the mean decreases monotonically and is approximated accurately by the algorithm. The estimation of the variance becomes more accurate as $N$ increases. Furthermore, Figures \ref{fig:Ex3_1c} and \ref{fig:Ex3_1d} show that the cumulants become stationary for long times indicating that the numerical approximations converge to a measure which is non-Gaussian.  

Table \ref{table:Ex3} compares the first six cumulants obtained by our algorithm at time $t=8$ with $K=4$, $N=3$, $L=4$ to the Monte Carlo method. For further comparison, we also provide cumulants obtained by averaging   Fokker--Planck density with respect to the known, explicit, distribution of the damping.   It can be observed from the following table that both our algorithm and Monte Carlo capture cumulants reasonably well although the accuracy degrades at higher orders.

\begin{table}[!htb]
\begin{minipage}[b]{1\linewidth}
\centering
\renewcommand{\tabcolsep}{0.1cm}
\renewcommand{\arraystretch}{1.00}
 \begin{tabular} { |r | r | r |r|r|r|r|}
  \hline
   & $\kappa_1$&$\kappa_2$ &$\kappa_3$  & $\kappa_4$&$\kappa_5$ & $\kappa_6$\\ \hline
   DgPC & 1.68E-5 &   1.10 & 3.80E-5 & 3.78E-1 &  8.72E-5 &5.04E-1\\ 
   MC & -1.78E-4 &   1.10   &   -3.25E-3& 3.70E-1  &-6.34E-2 &  4.65E-1 \\
    FP &  0&     1.10 &   0 &3.79E-1&0  &5.29E-1\\ 
    \hline  
  \end{tabular}
  \caption{Cumulants.}
  \label{table:Ex3}
\end{minipage}
\end{table}

The above calculations provide an example of stochasticity with two distinct components: the variables $\xi$ that are Markov and can be projected out and the non-Markov variable $b_v$  that has long-time effects on the dynamics. The variables $(v,b_v)$ are strongly correlated for positive times. PCE thus need to involve orthogonal polynomials that reflect these correlations and cannot be written as tensorized products of orthogonal polynomials in $v$ and $b_v$.

\end{exmp}

 \begin{exmp} \rm \,
 We demonstrate that our method is not limited to equations forced by Brownian motion. We apply it to an equation where the forcing is a nonlinear function of Brownian motion:
\begin{align*}
 d v(s) = - b_v \, v(s)\, ds + \sigma_v \, d(W^2(s)-s),\quad v(0)=1,  
\end{align*}
with parameters $b_v=6$,  $\sigma_v=1$. To depict the effect of  different choices of number of restarts, we take  $\Delta t =0.5,0.3,0.1$ using $K=5$, $N=2$, $L=2$ in each expansion.

\begin{figure}[!htb]
\centering
\subfigure[mean]{
\includegraphics[scale=0.24]{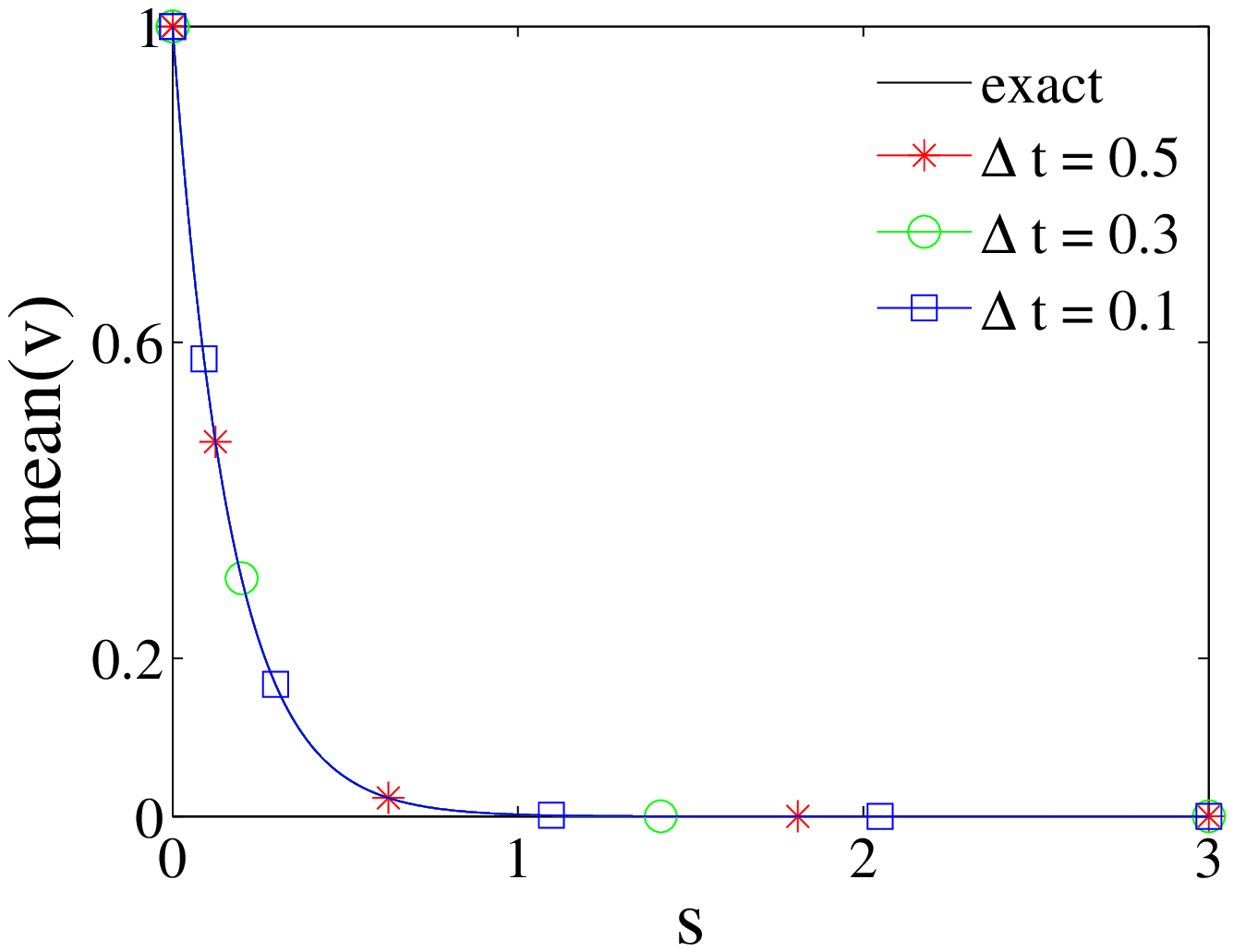}}
\subfigure[variance]{
\includegraphics[scale=0.24]{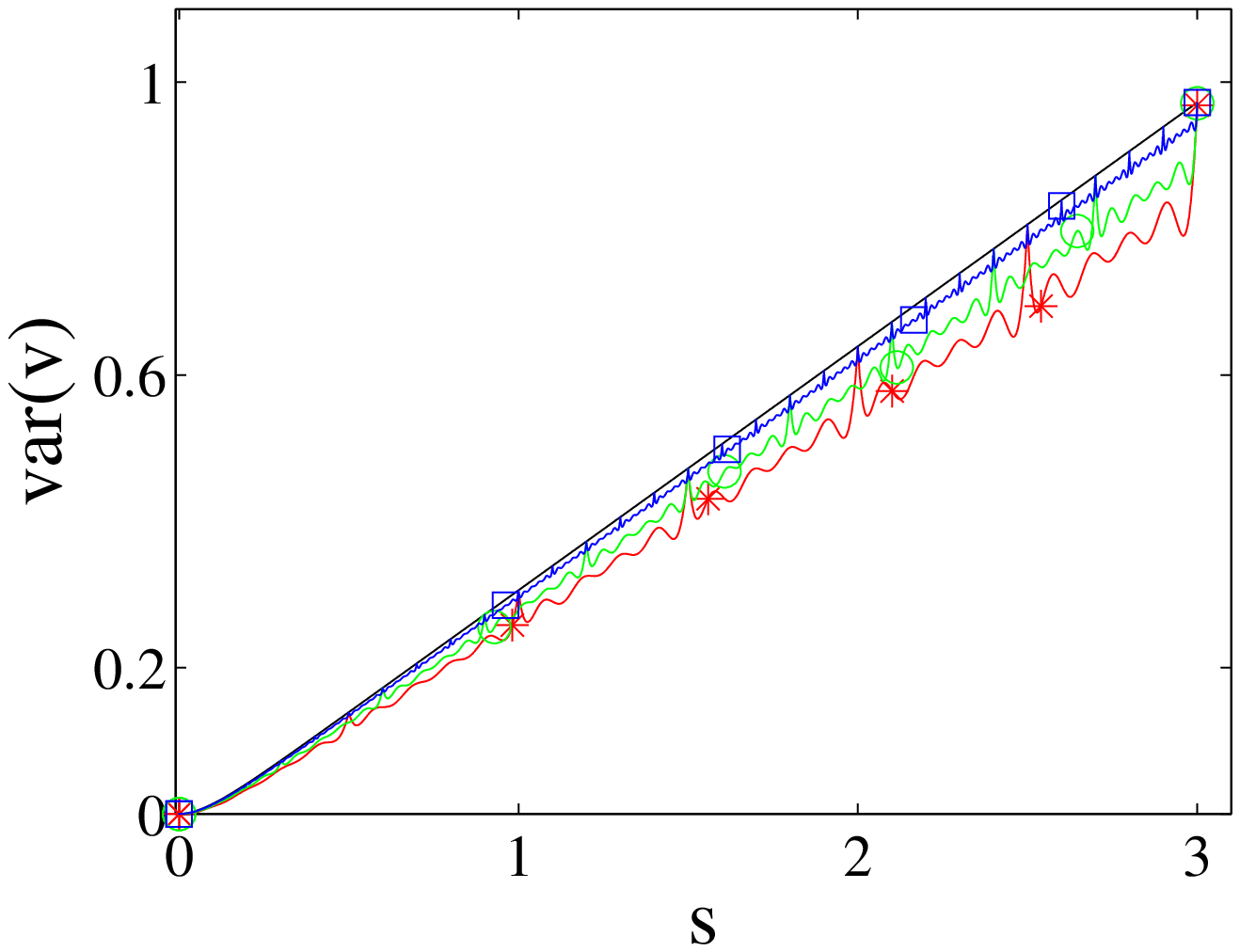}} 
\subfigure[$\epsilon_{\mbox{mean}}$]{
\includegraphics[scale=0.24]{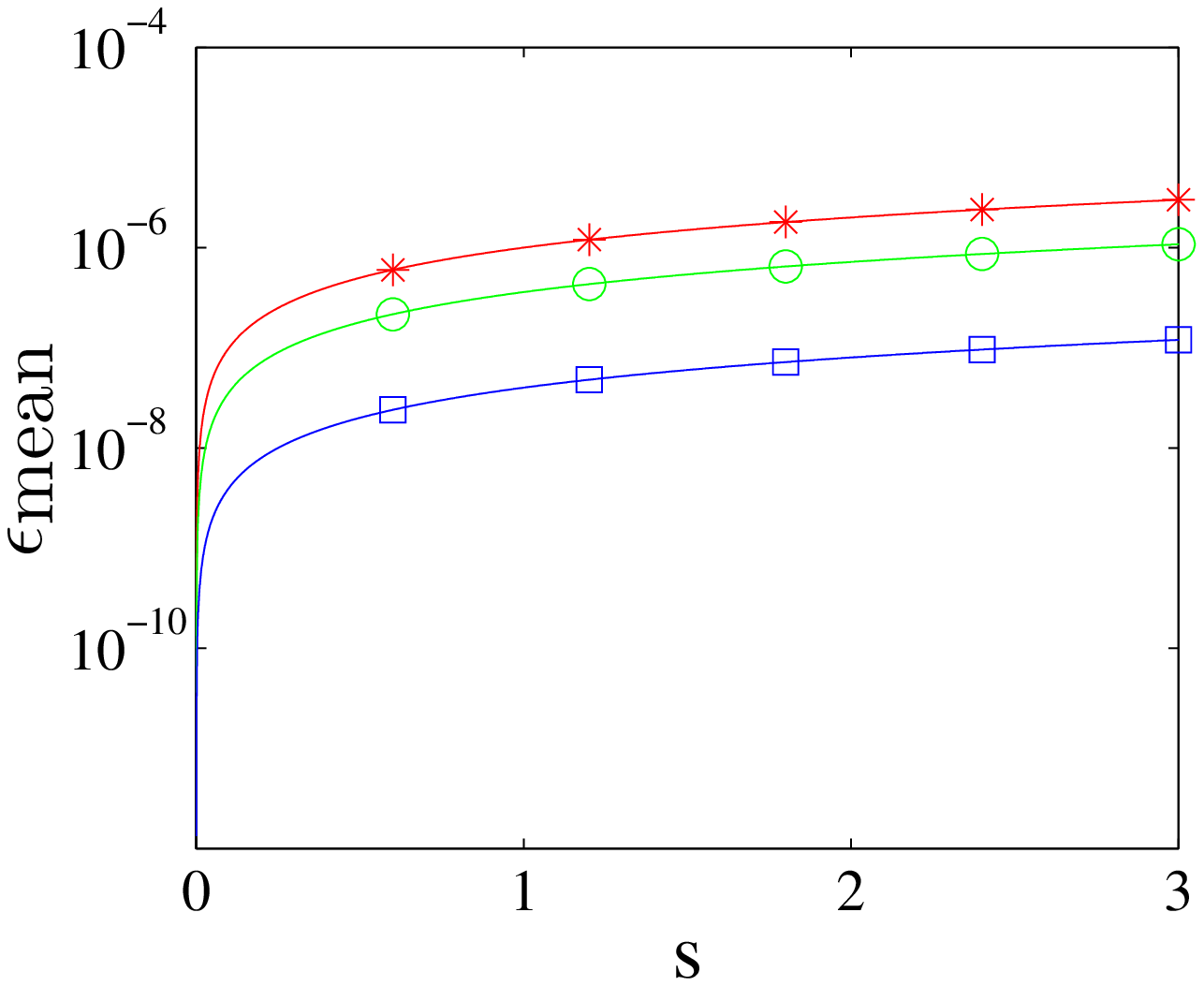}}
\subfigure[ $\epsilon_{\mbox{var}}$]{
\includegraphics[scale=0.24]{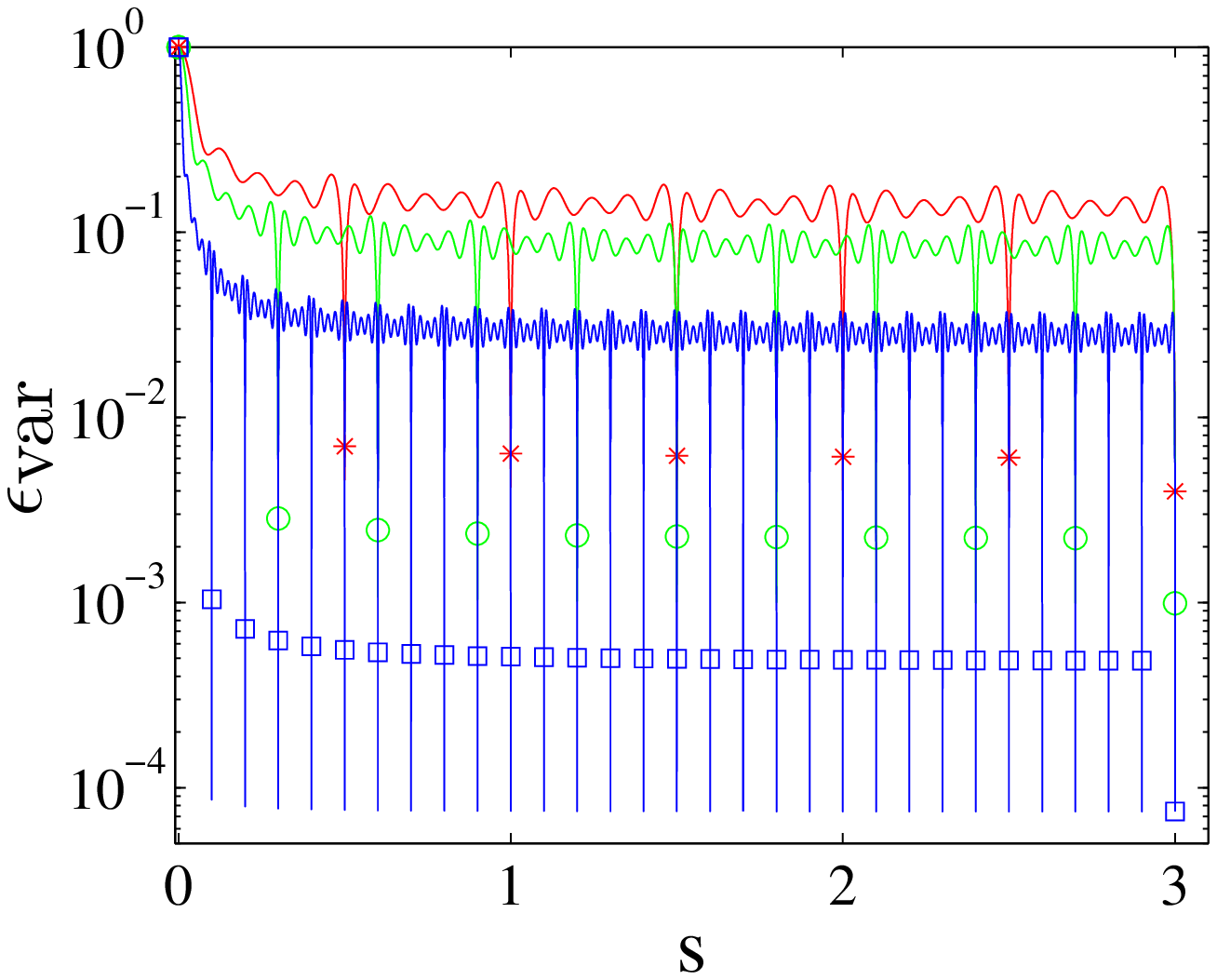}} 
\vspace{-0.2cm}
\caption{DgPC $\Delta t= 0.5,0.3,0.1$.}
\label{fig:Ex4}
\end{figure} 

From Figure \ref{fig:Ex4}, we observe that second order statistics are captured accurately and an increasing number of restarts provides better approximations as anticipated.  Although the system does not converge to a steady state (variance of $v(s)$ increases linearly), this example illustrates that DgPC  is able to capture behaviors of solutions where the forcing is a nonlinear function of Wiener process. 

\end{exmp}

 \begin{exmp} \rm \,
This example concerns the nonlinear system  \eqref{eq:system}, where the dynamics of $u$ exhibit intermittent non-Gaussian behavior. A time dependent deterministic periodic forcing is considered with the second equation being an OU process, i.e. $a_v=0$ in \eqref{eq:system}, and the parameters are taken as
$
b_u =1.4$,  $b_v=10$ , $\sigma_u =0.1$, $\sigma_v =10$ , $a_u=1$, 
 $f(t)= 1 + 1.1 \cos(2t+1)+0.5\cos(4t)$,
 with initial conditions $
u_0 = N(0, \sigma_u^2/8b_u) \indep v_0 = N(0, {\sigma_v^2/8b_v})
$; see also~\cite{BM13}.
The initial variables are independent of each other and of the stochastic forcing. Intermittency is introduced by using an OU process, which acts as a multiplicative noise fluctuating in the damping. The second order statistics for this case ($a_v=0$) can be derived analytically~\cite{GHM10} and we plot them in black in the following figures using quadrature methods with sufficiently high number of quadrature points.

We first depict the results for $u$ obtained by standard Hermite PCE in the first row of Figure \ref{fig:ex6_1}. It can be observed that truncated PC approximations are accurate for short times. However, they quickly lose their accuracy as time grows.  This is consistent with the simulations presented in~\cite{BM13}. 
The second row of the same figure shows that the accuracy is vastly improved in DgPC using $\Delta t =0.1$, and second order statistics are accurate up to three or four digits. Moreover, errors in the variance oscillate around $O(10^{-3})$  for $L=3$ throughout the time length suggesting that the approximation retains its accuracy in the long term, which is consistent with theoretical predictions. Our algorithm easily outperforms standard PCE in capturing the long-time behavior of the nonlinear coupled system \eqref{eq:system} in this scenario.
 
\begin{figure}[!htb]
\centering
\subfigure[mean(u)]{
\includegraphics[scale=0.24]{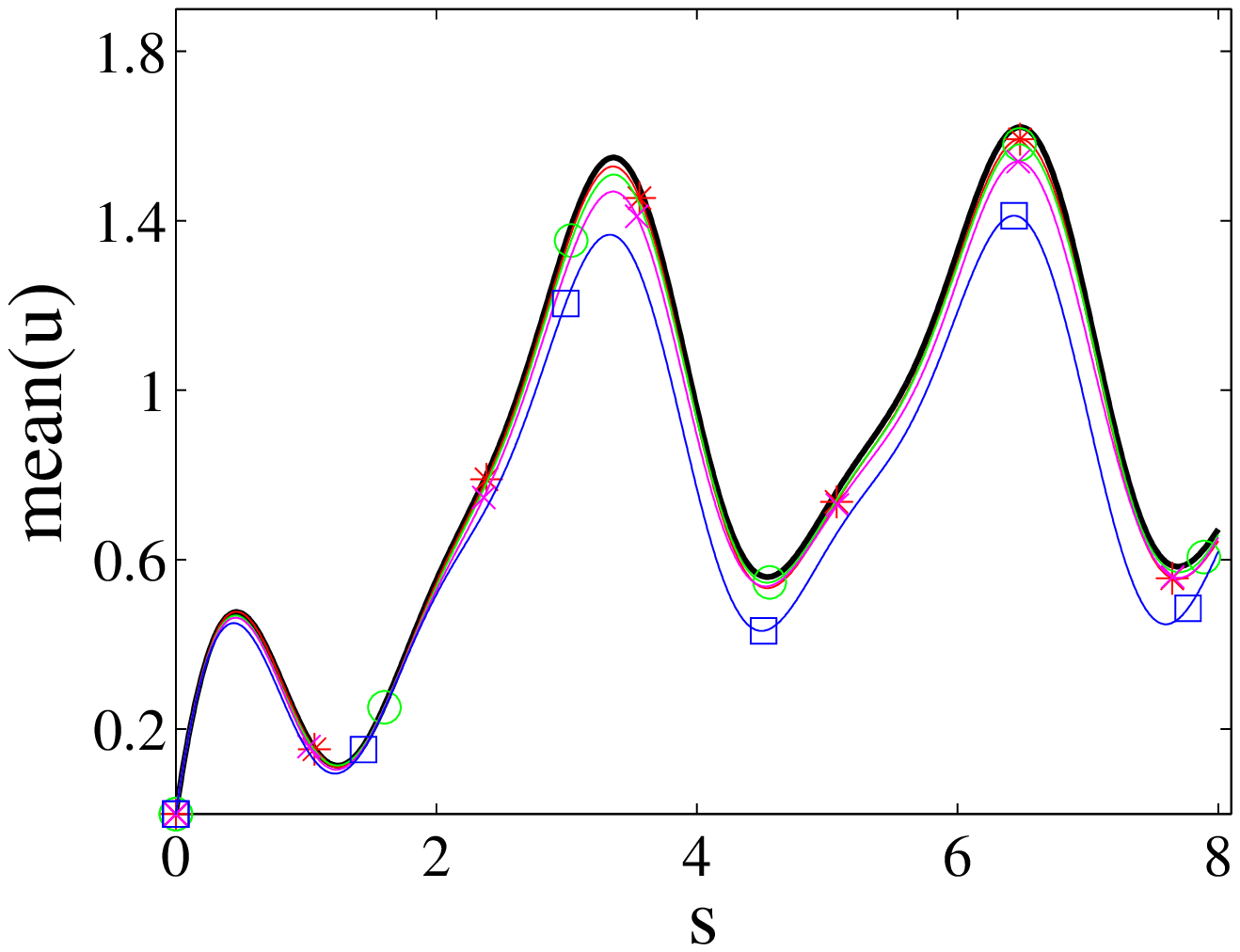}}
\subfigure[var(u)]{
\includegraphics[scale=0.24]{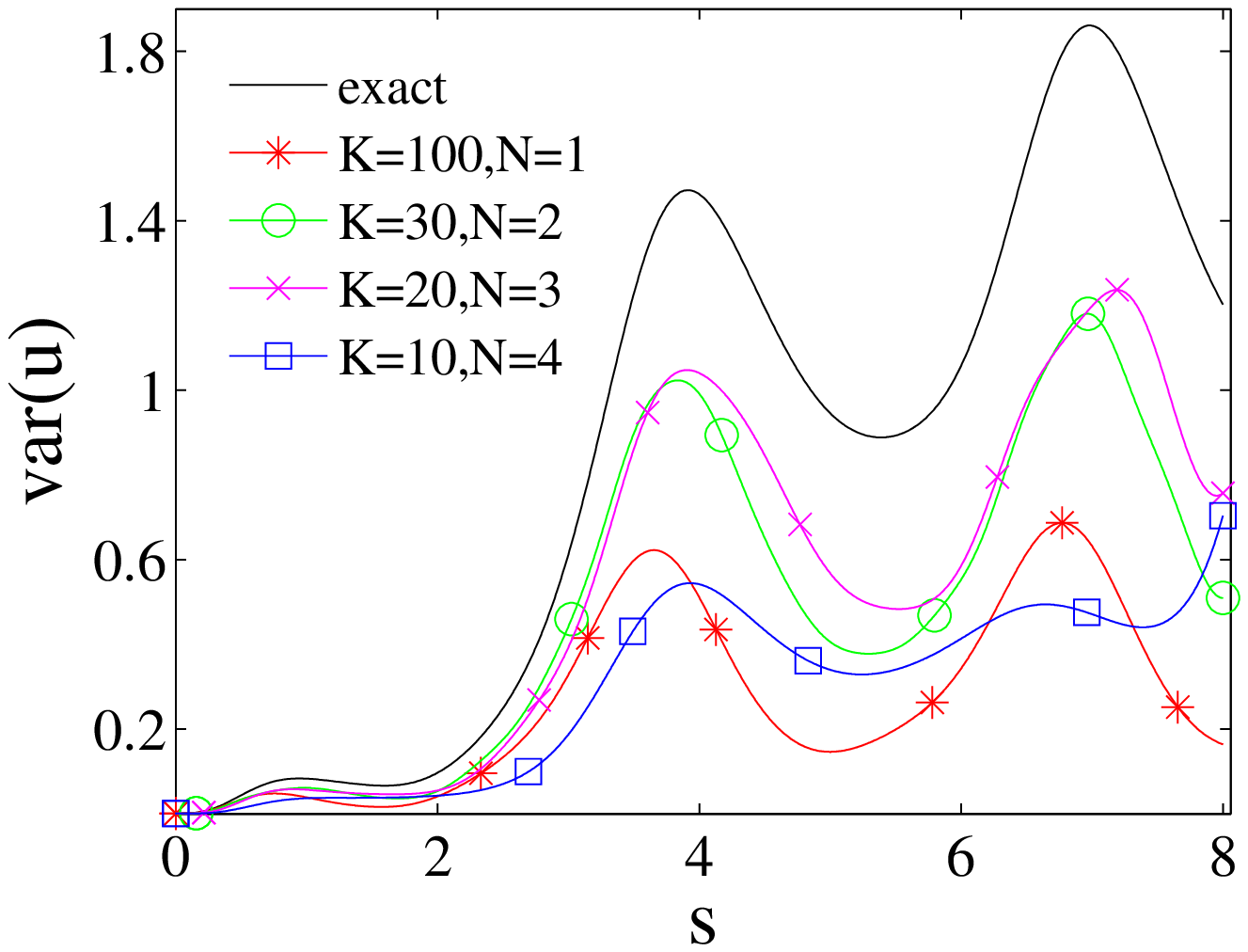}}
\subfigure[$\epsilon_{\mbox{mean}}$]{
\includegraphics[scale=0.24]{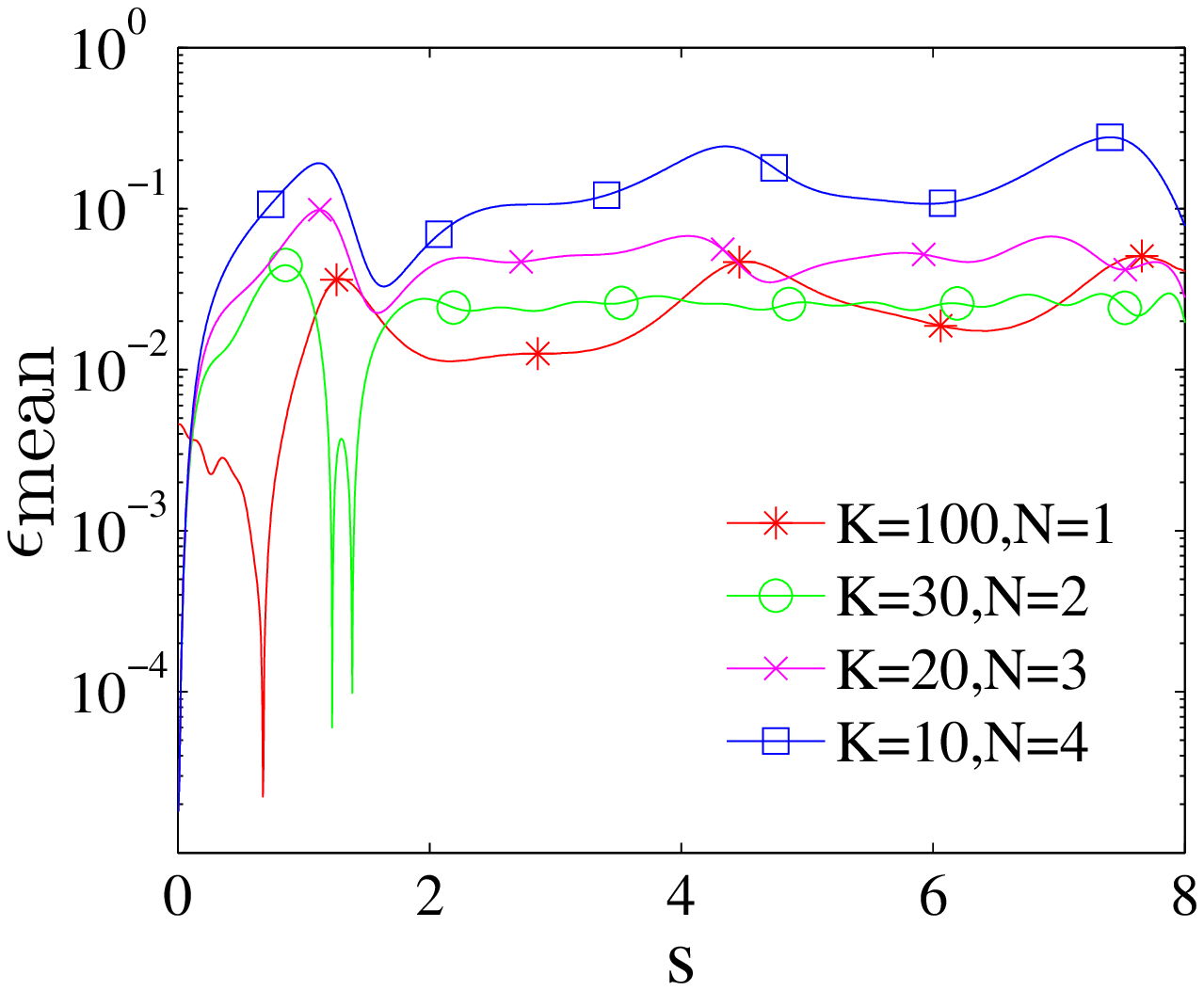}}
\subfigure[$\epsilon_{\mbox{var}}$]{
\includegraphics[scale=0.24]{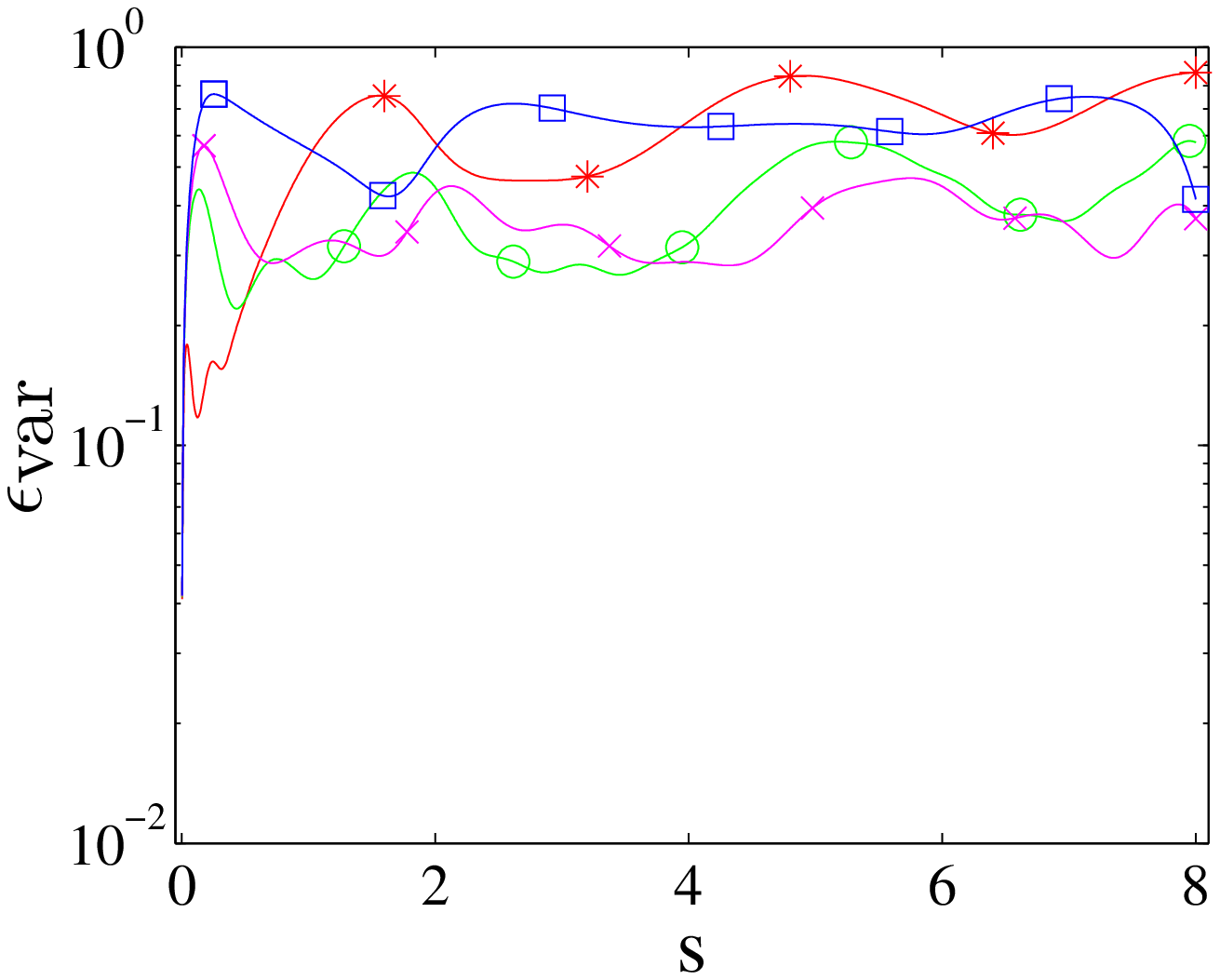}}  \\
\vspace{-0.2cm}
\subfigure[mean(u)]{
\includegraphics[scale=0.24]{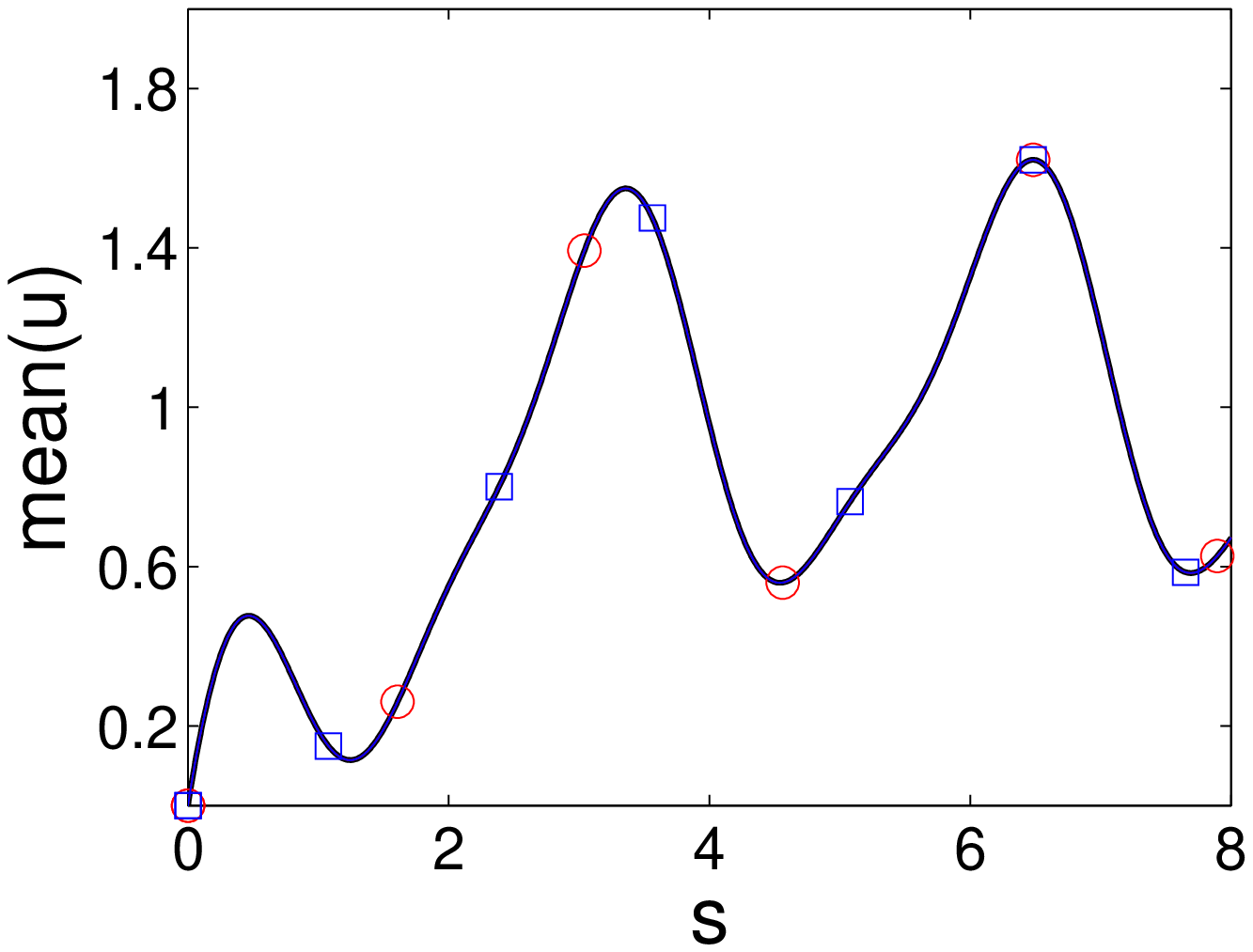}}
\subfigure[var(u)]{
\includegraphics[scale=0.24]{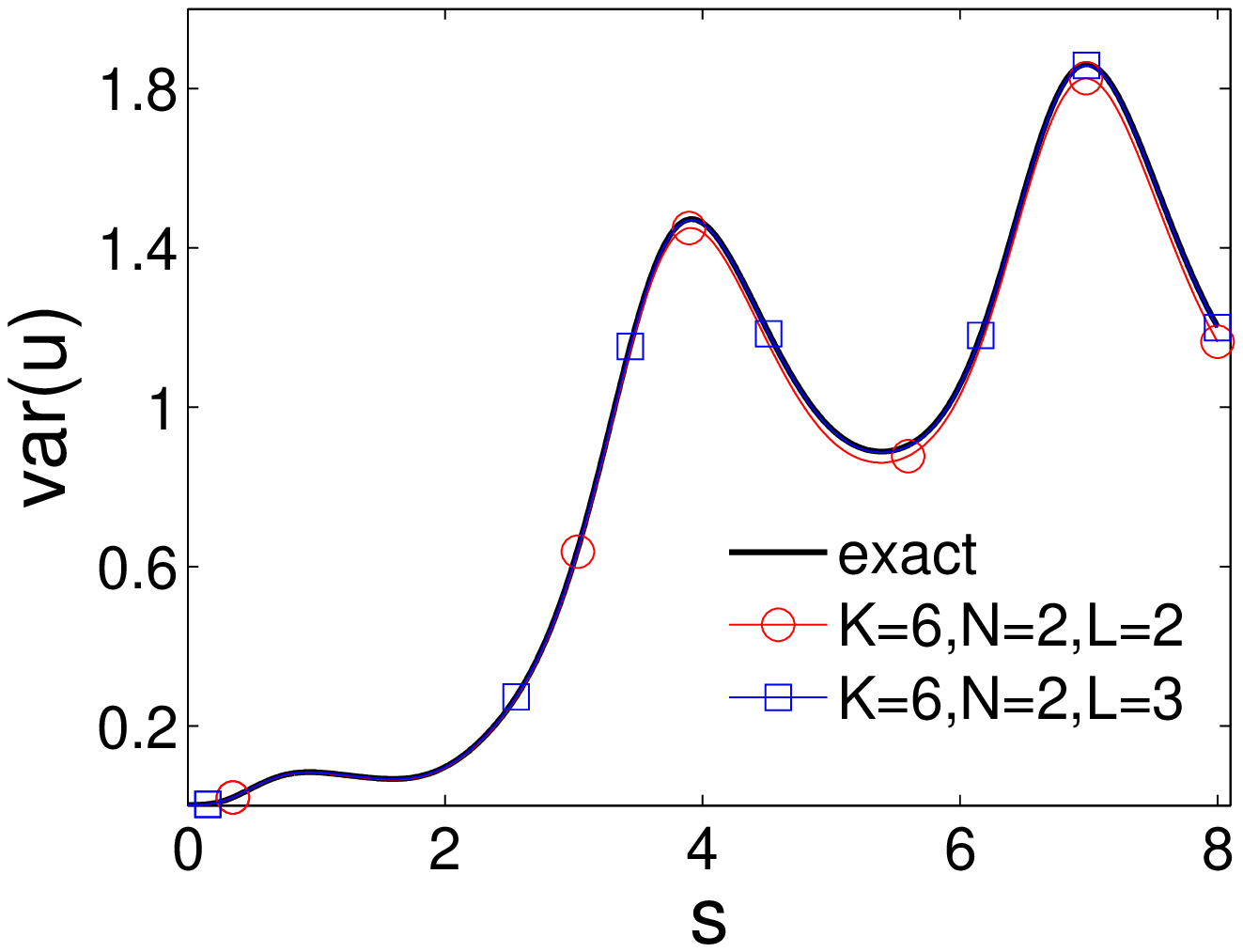}} 
\subfigure[$\epsilon_{\mbox{mean}}$]{
\includegraphics[scale=0.24]{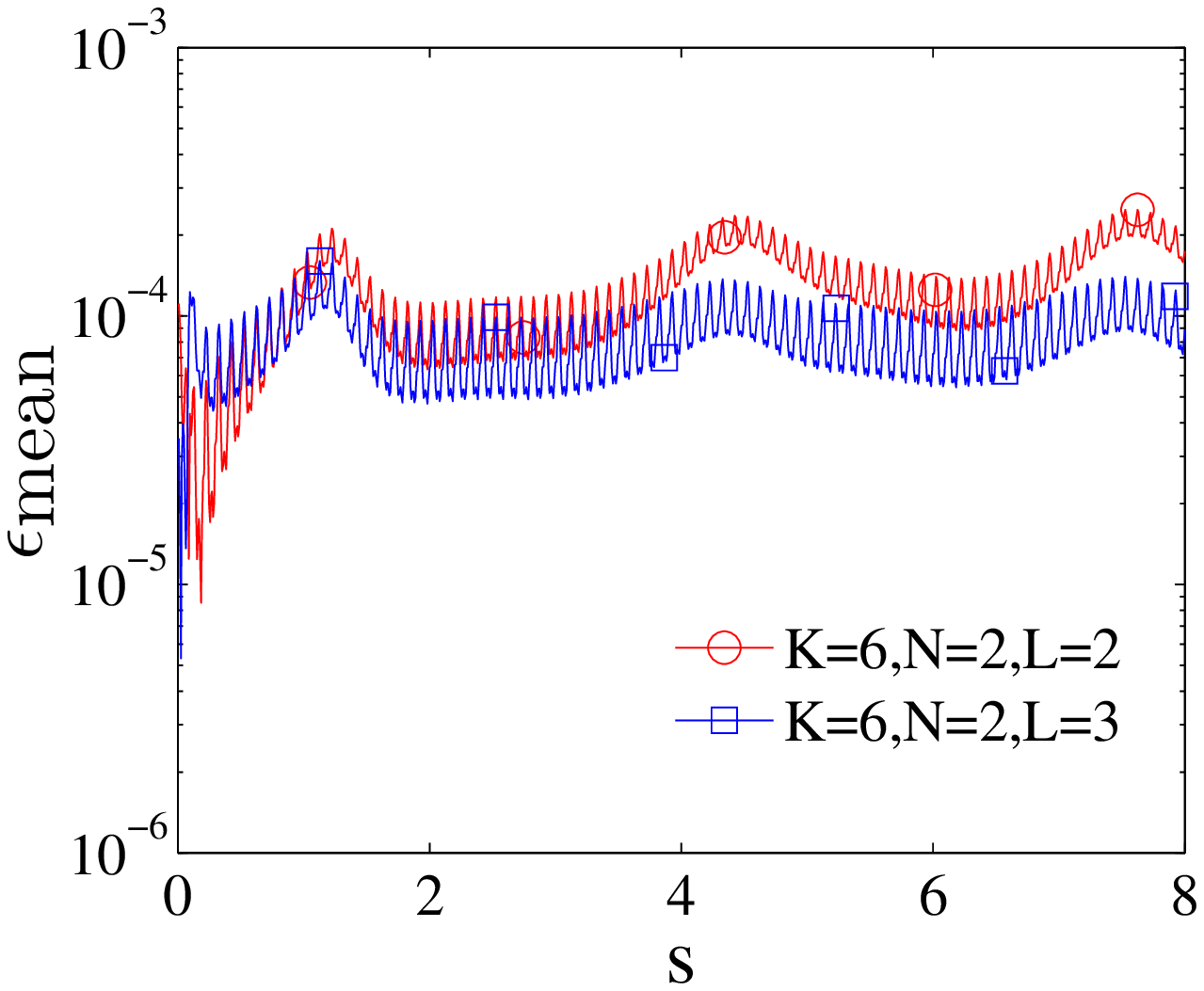}} 
\subfigure[$\epsilon_{\mbox{var}}$]{
\includegraphics[scale=0.24]{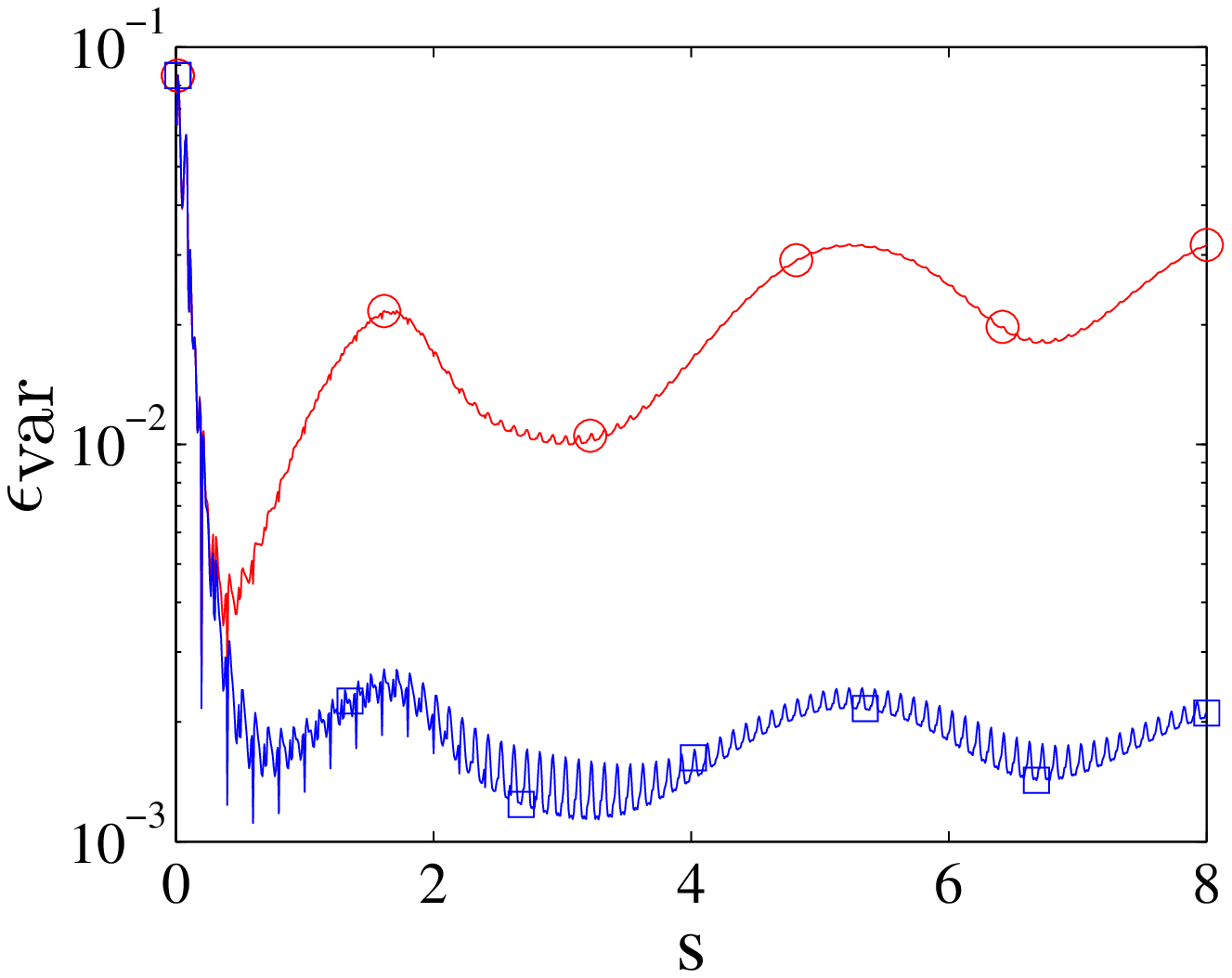}}
\vspace{-0.2cm}
\caption{Hermite PC \& DgPC $\Delta t = 0.1$.}
\label{fig:ex6_1}
\end{figure}

\end{exmp}

 \begin{exmp} \rm \,
 For this scenario, we consider the system parameters as $
a_u=1$ , $a_v=0$, $b_u =1.2$, $b_v=0.5$,$\sigma_u =0.5$,  $\sigma_v =0.5$ without the deterministic forcing $f=0$ and
with the initials $
u_0 = N(1,\sigma_u^2/8b_u) \, \indep \, v_0 = N(0,{\sigma_v^2/8b_v}).
$ In this regime, the dynamics of $u(s)$ are characterized by unstable bursts of large amplitude~\cite{BM13}.

Second order statistics obtained by Hermite PCE and our algorithm are presented in Figure \ref{fig:ex7_2}, where the first two parts (Figures \ref{fig:ex7_1a} and \ref{fig:ex7_1b}) correspond to Hermite PCE. We observe from Figure \ref{fig:ex7_1b} that the relative errors for standard Hermite expansions are unacceptably high. 
 On the other hand, as the degrees of freedom is increased in DgPC, there is a clear pattern of convergence of second order statistics; see Figures \ref{fig:ex7_2e} and \ref{fig:ex7_2f}. Note, however,  that short-time accuracy is better than accuracy for long times. This behavior may be explained by the onset of intermittency as nonlinear effects kick in after short times. Further, Figure \ref{fig:ex7_2g} shows the error behavior of variance at $t=10$ as $N$ and $L$ vary, and $K$ is fixed. The rate of convergence is consistent with exponential convergence in  $N$ and $L$ as the logarithmic plot is almost linear. For similar convergence behaviors, see \cite{Xiu_thesis,Luo,XK02}.     

\begin{figure}[!htb]
\centering
\subfigure[mean(u) (Hermite)]{ \label{fig:ex7_1a}
\includegraphics[scale=0.24]{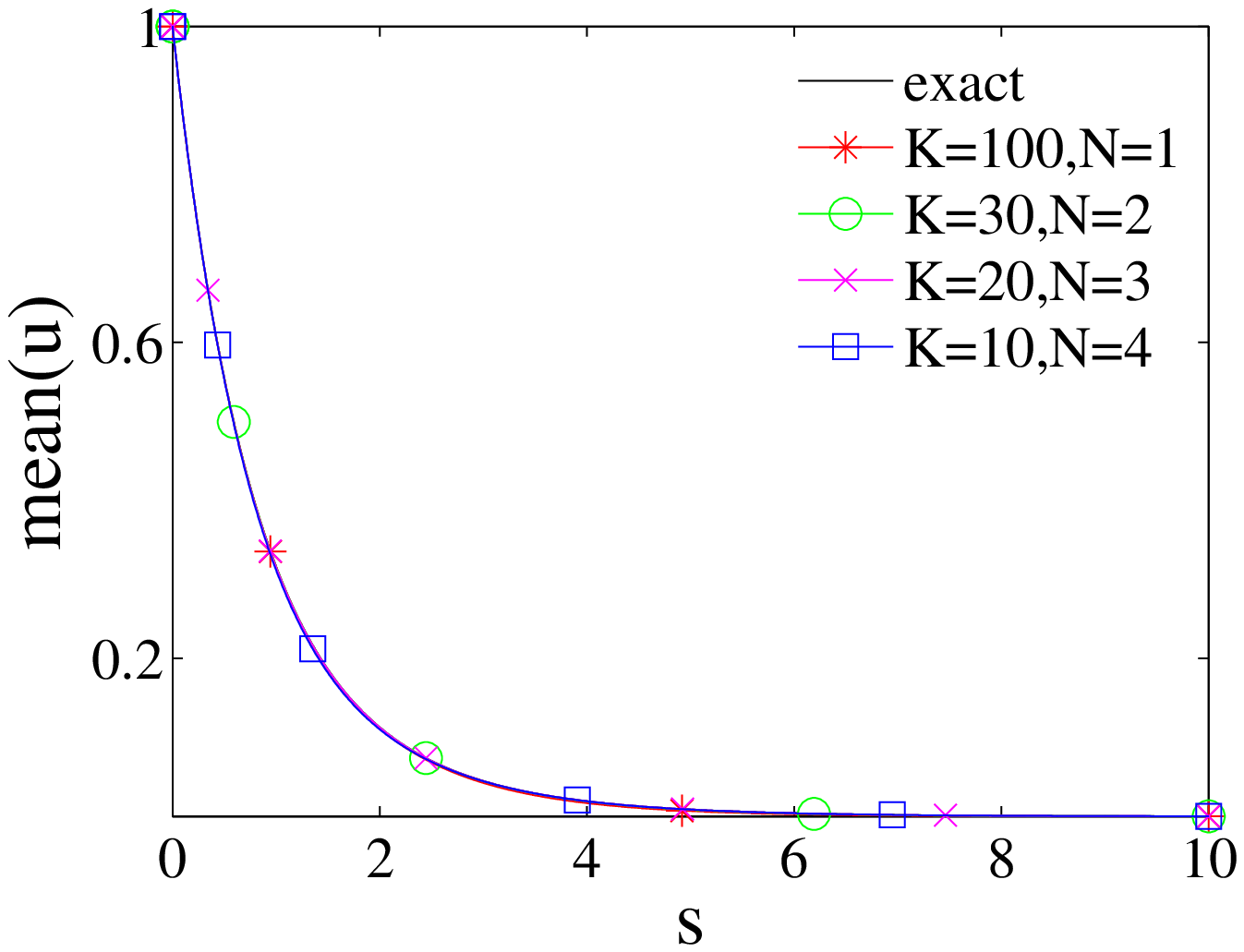}}
\subfigure[var(u) (Hermite)]{\label{fig:ex7_1b}
\includegraphics[scale=0.24]{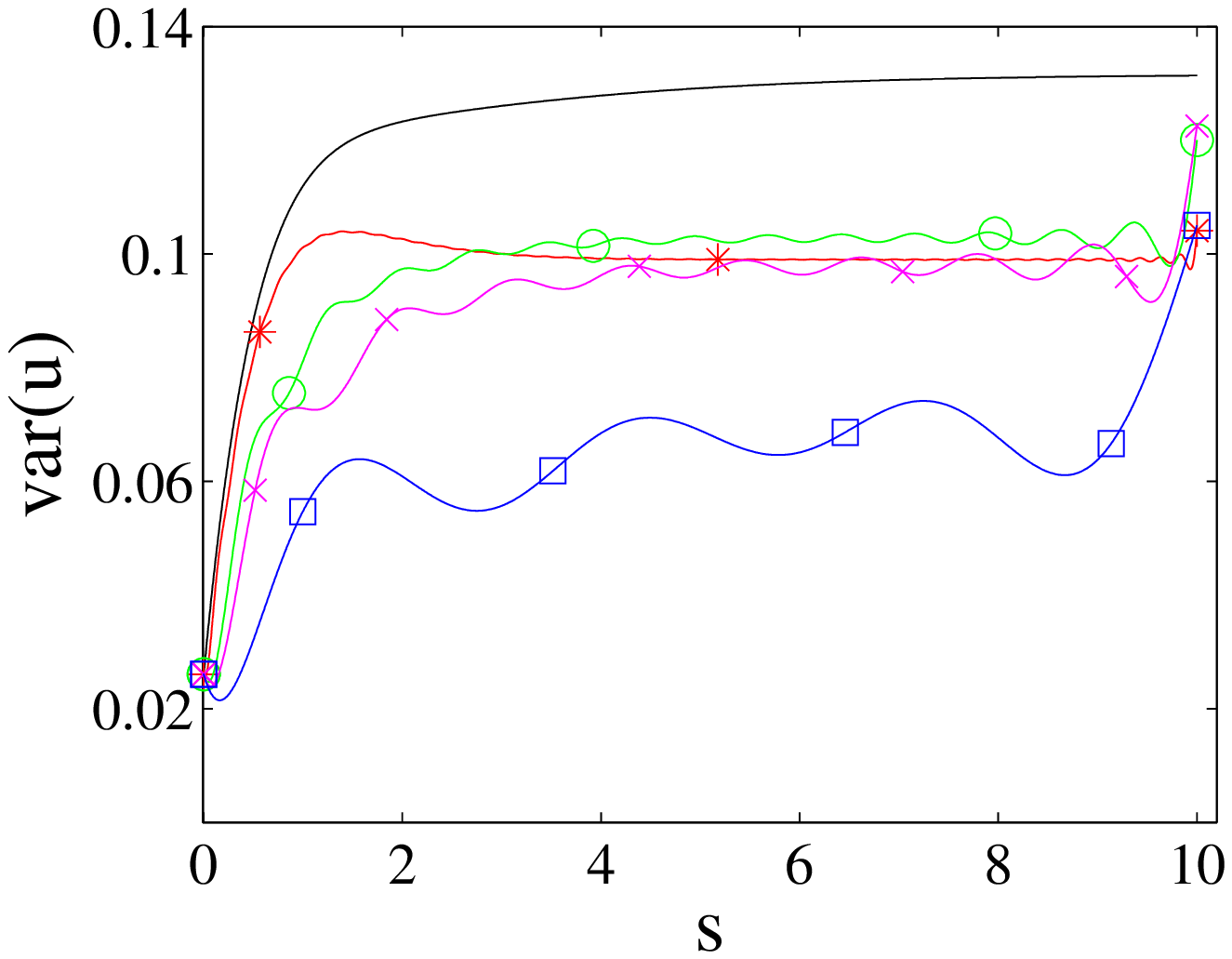} } 
\subfigure[mean(u)]{\label{fig:ex7_2c}
\includegraphics[scale=0.24]{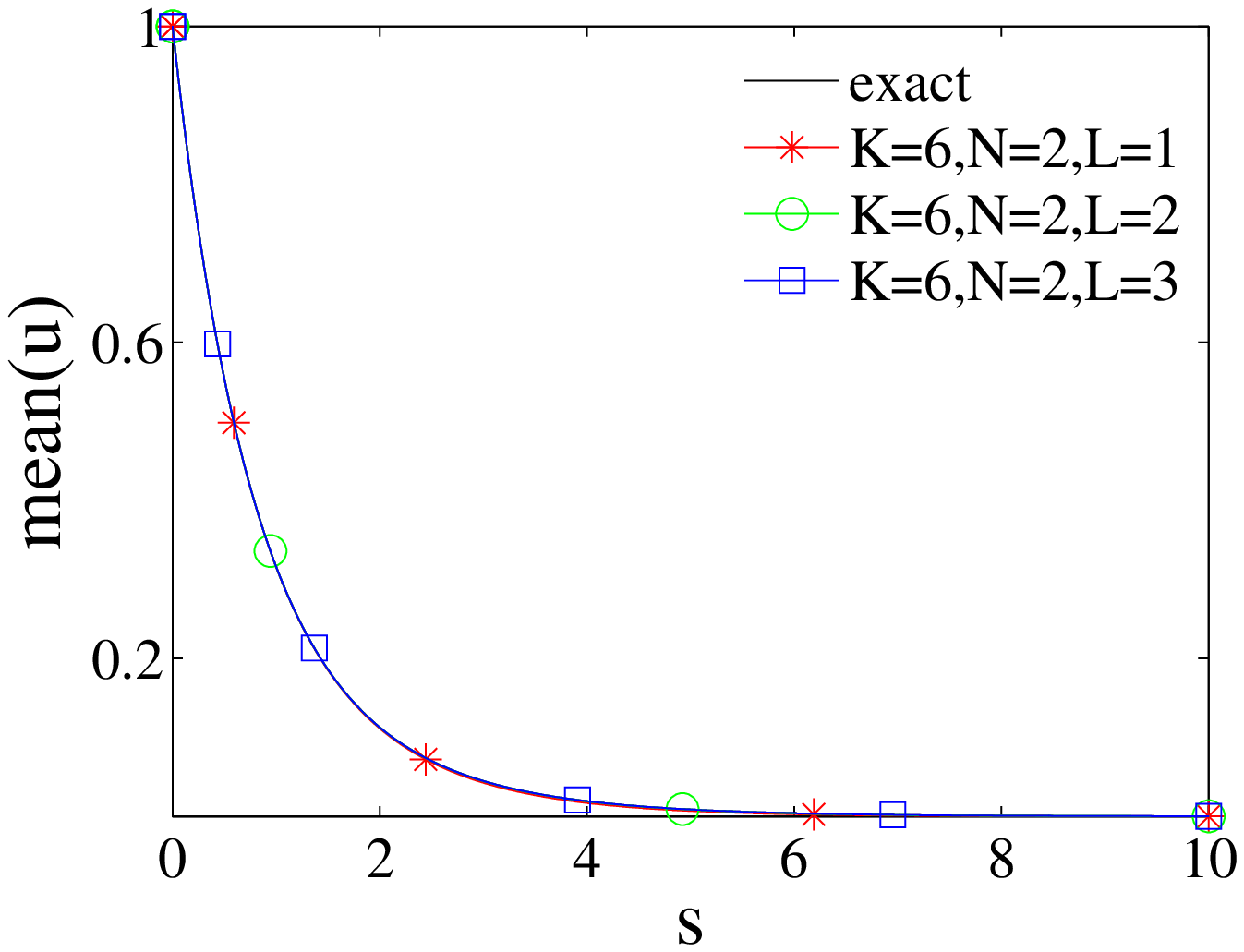}}
\subfigure[var(u)]{\label{fig:ex7_2d}
\includegraphics[scale=0.236]{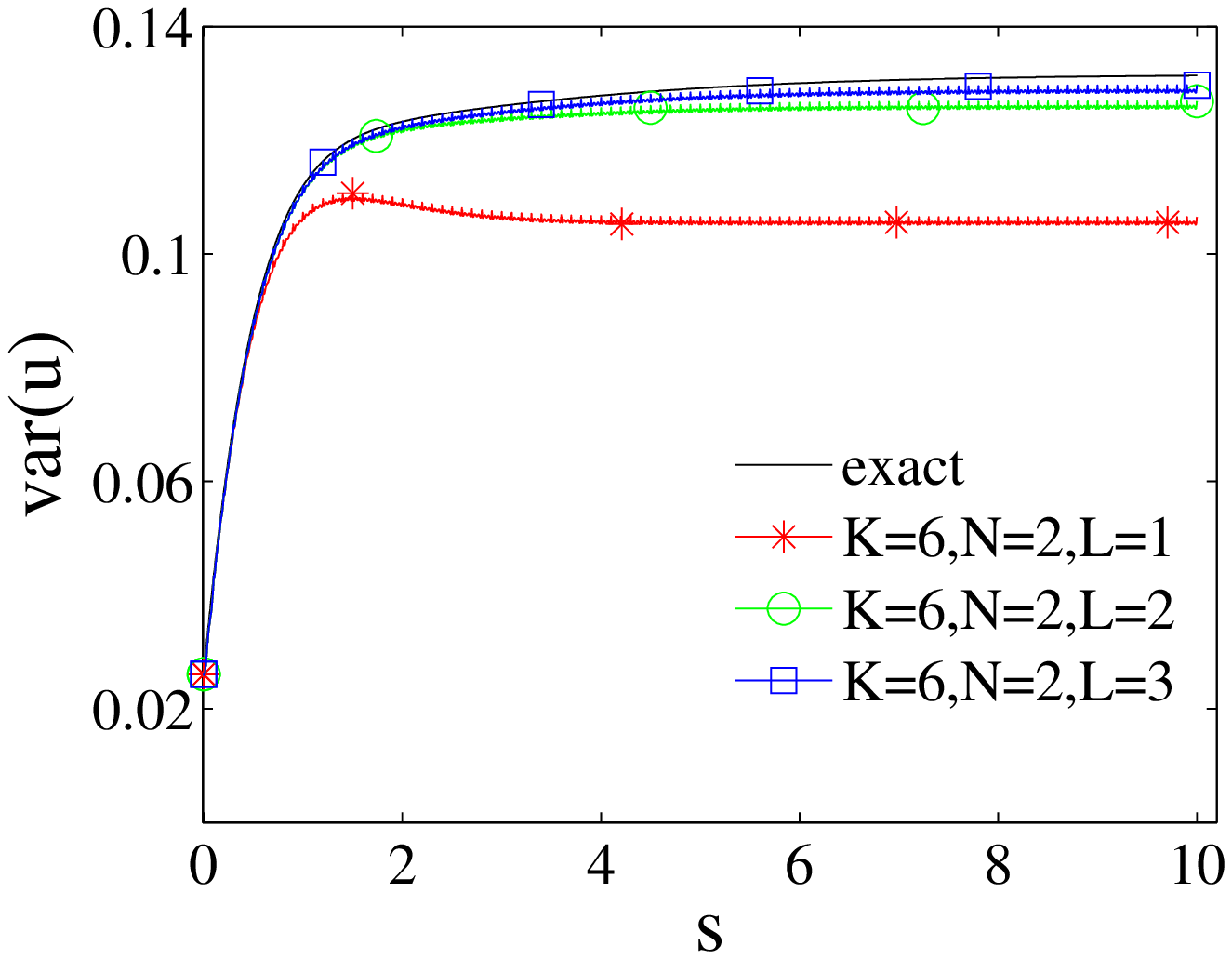}} \\
\vspace{-0.2cm}
\subfigure[$\epsilon_{\mbox{mean}}$]{\label{fig:ex7_2e}
\includegraphics[scale=0.24]{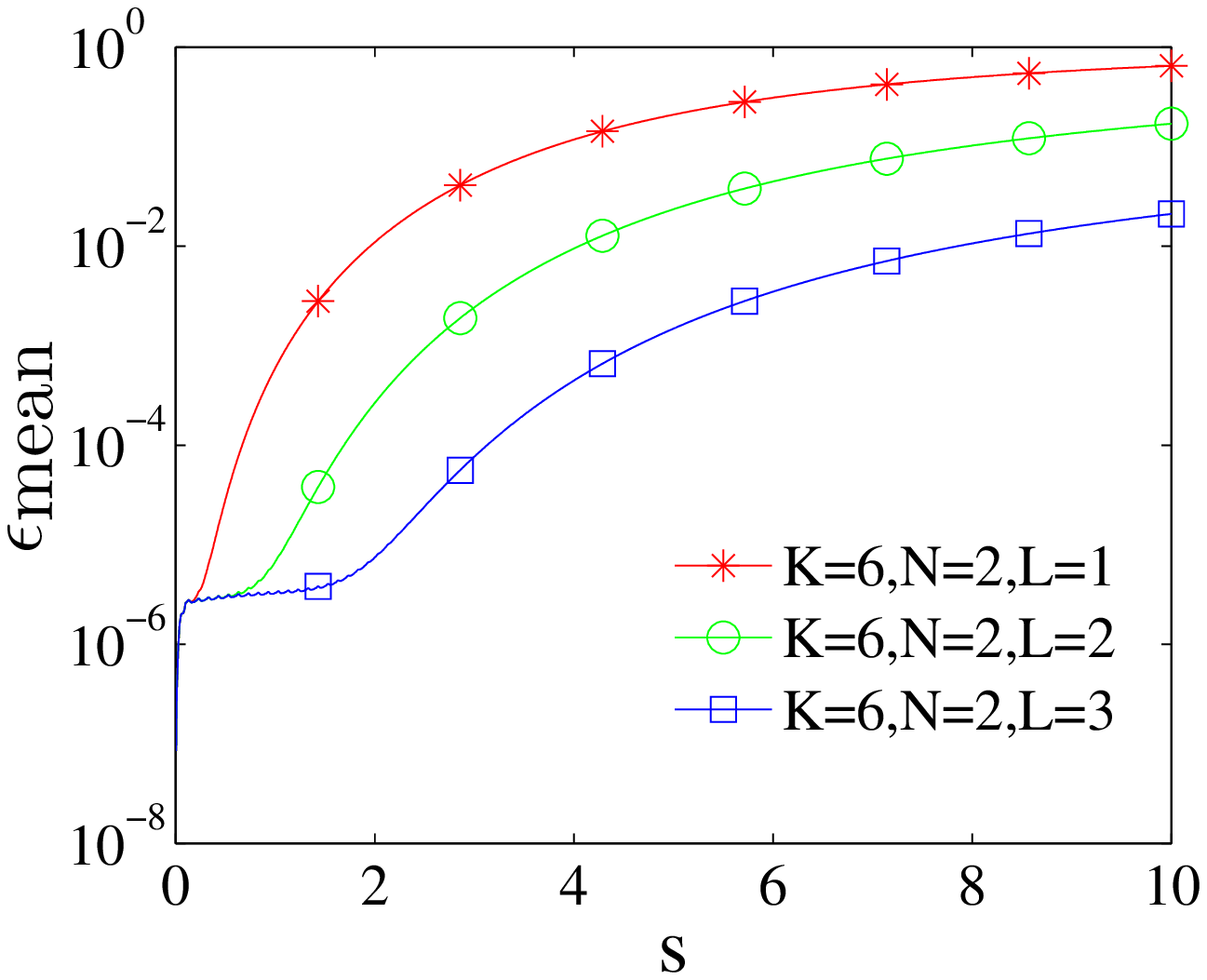}}
\subfigure[$\epsilon_{\mbox{var}}$]{\label{fig:ex7_2f}
\includegraphics[scale=0.24]{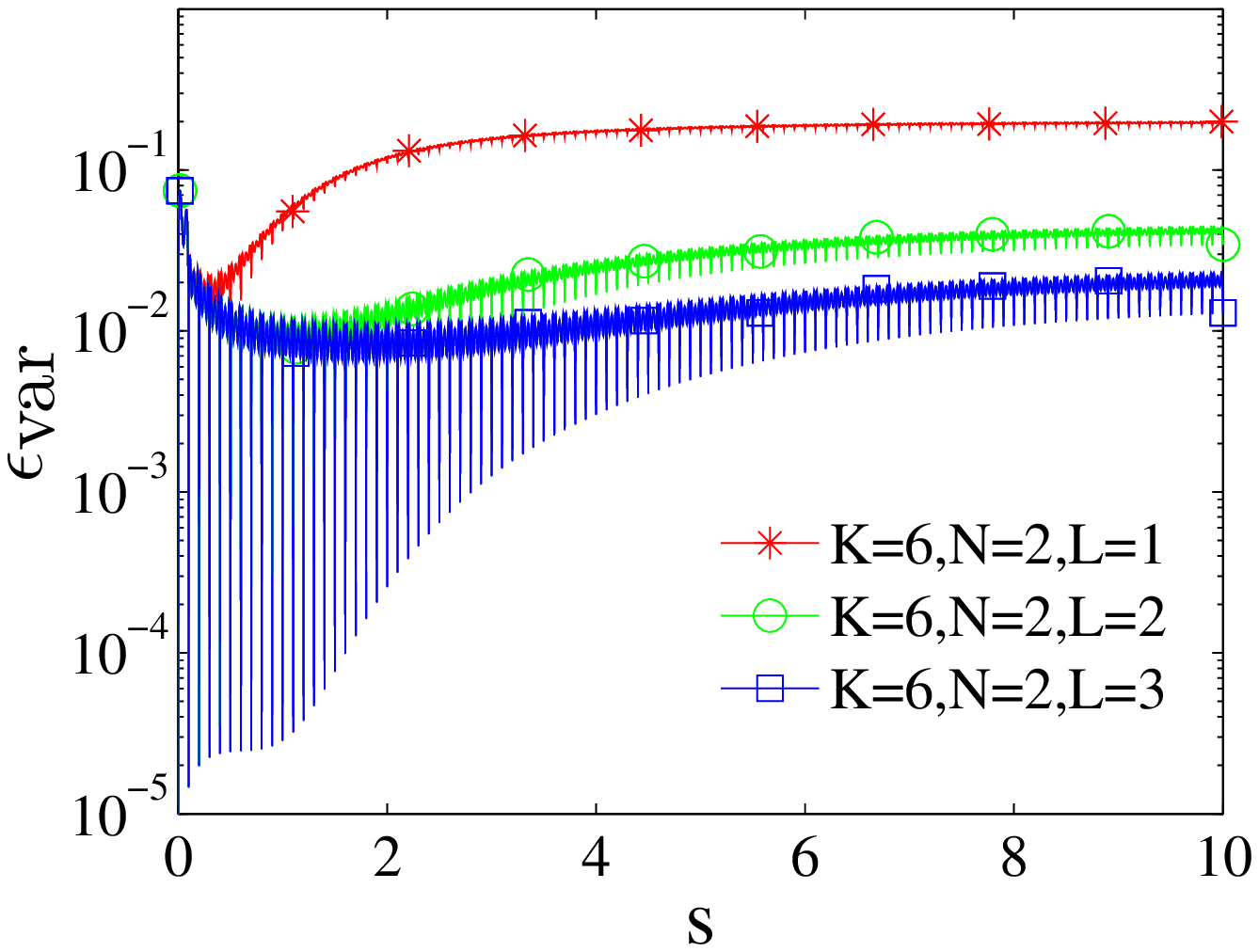}}
\subfigure[ $N,L$ conv. $t=10$]{ \label{fig:ex7_2g}
\includegraphics[scale=0.24]{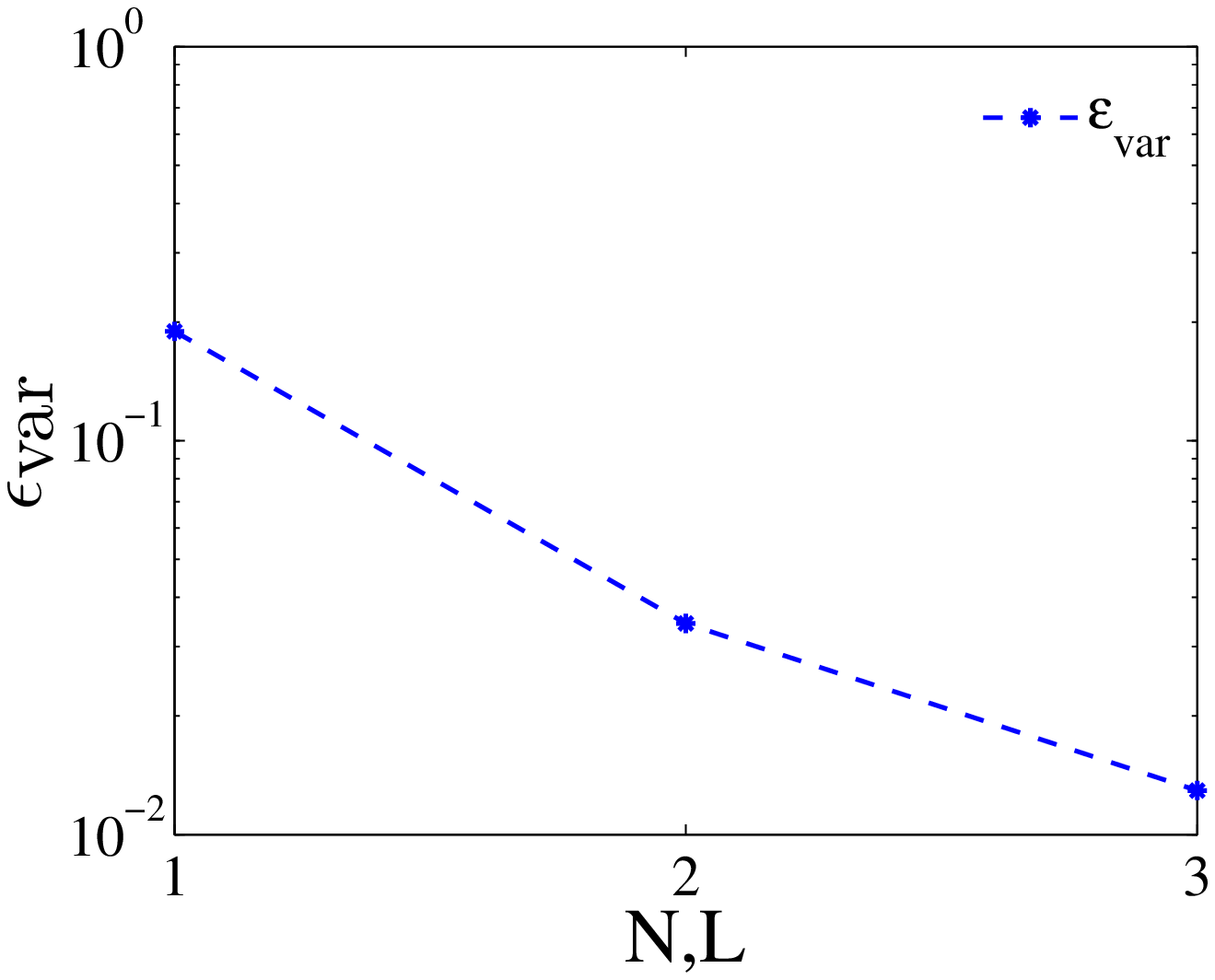}}
\vspace{-0.2cm}
\caption{Hermite PC \& DgPC $\Delta t = 0.1$.}
\label{fig:ex7_2}
\end{figure}

\begin{figure}[!htb]
\vspace{-0.2cm}
\centering
\subfigure[ $1^{st}\& 2^{nd}$ order]{
\includegraphics[scale=0.24]{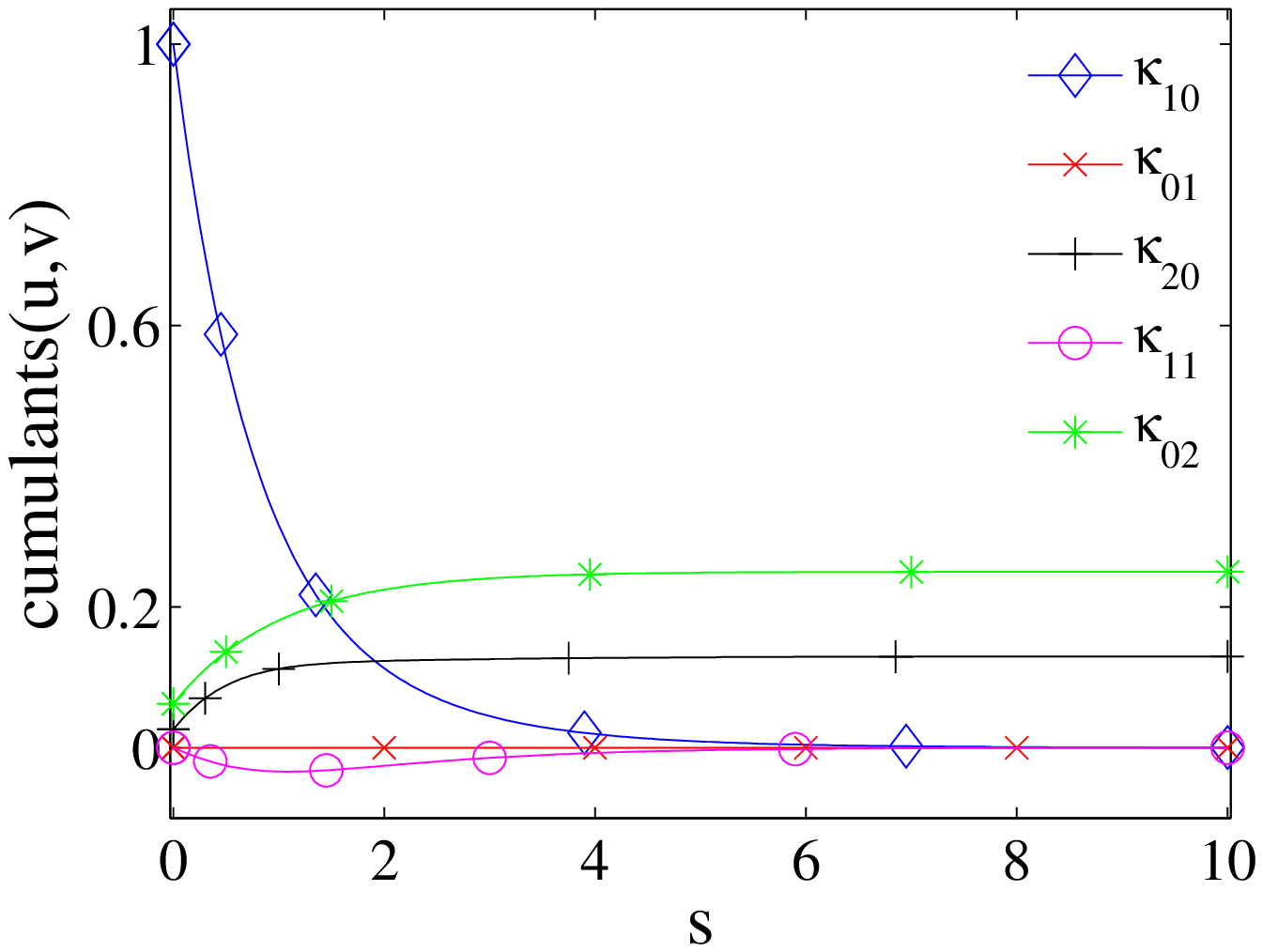}}
\subfigure[ $3^{rd}$ order ]{
\includegraphics[scale=0.24]{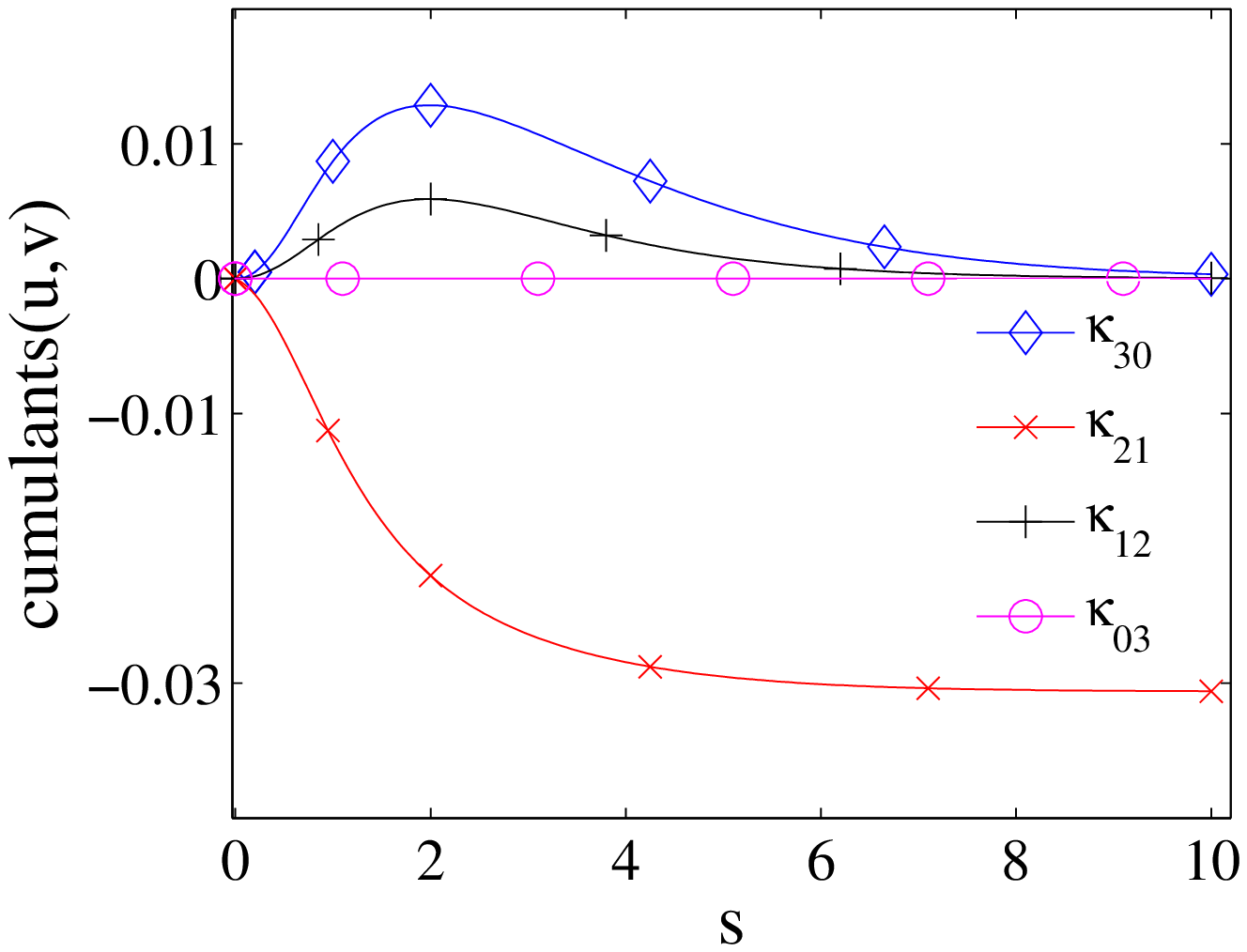}} 
\subfigure[ $4^{th}$ order ]{
\includegraphics[scale=0.24]{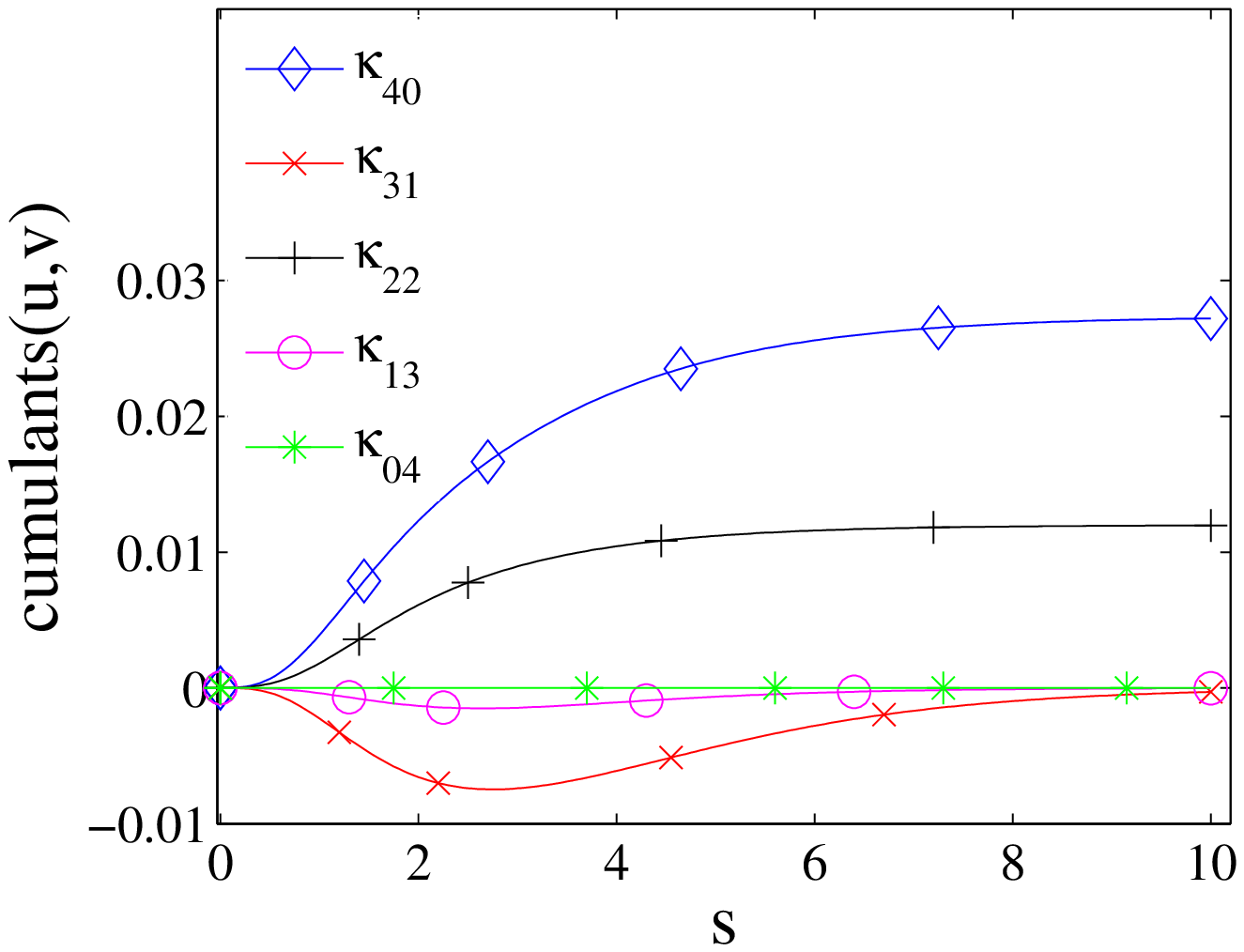}} \\
\vspace{-0.3cm}
\subfigure[ $5^{th}$ order]{
\includegraphics[scale=0.24]{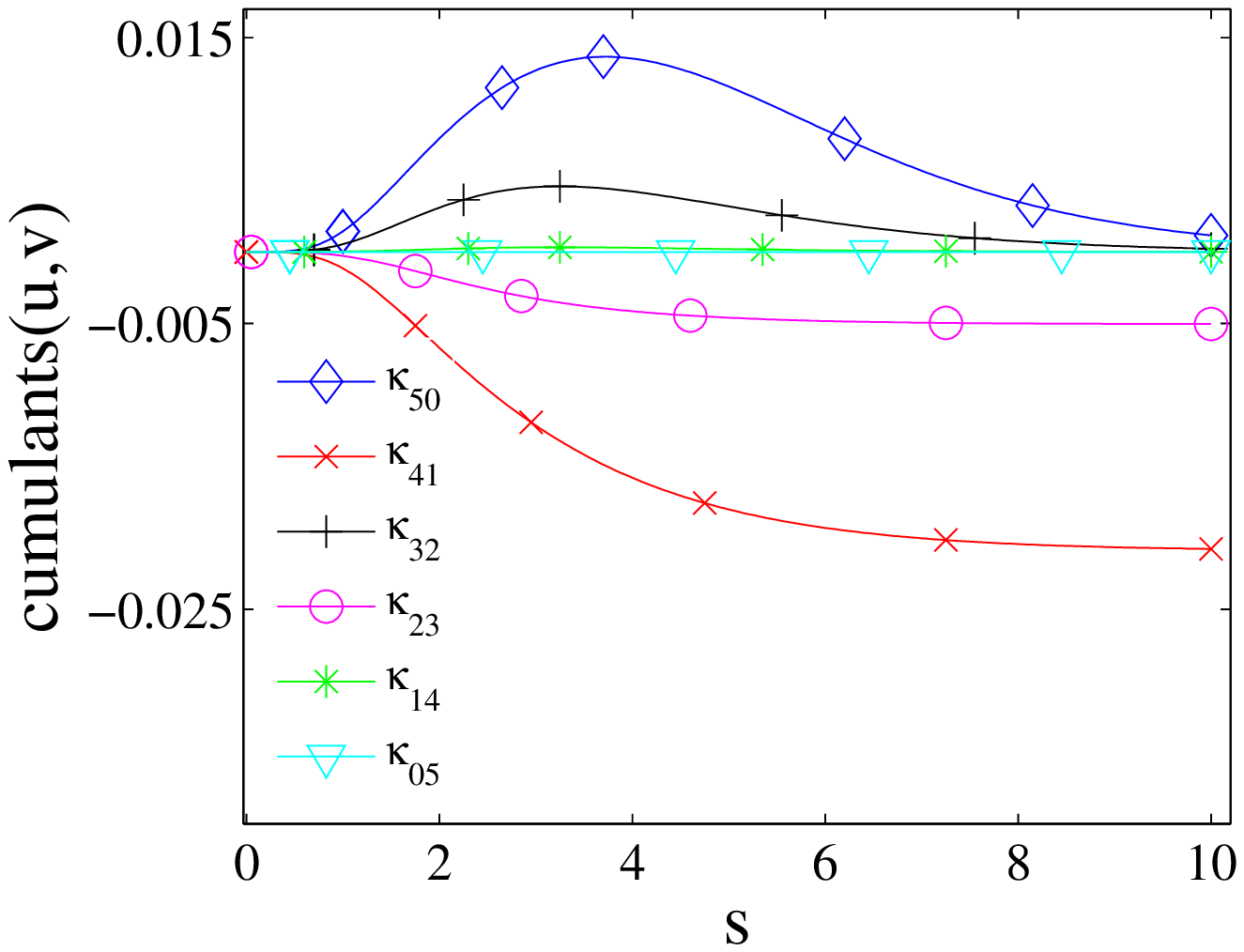}}
\subfigure[ $6^{th}$ order ]{
\includegraphics[scale=0.24]{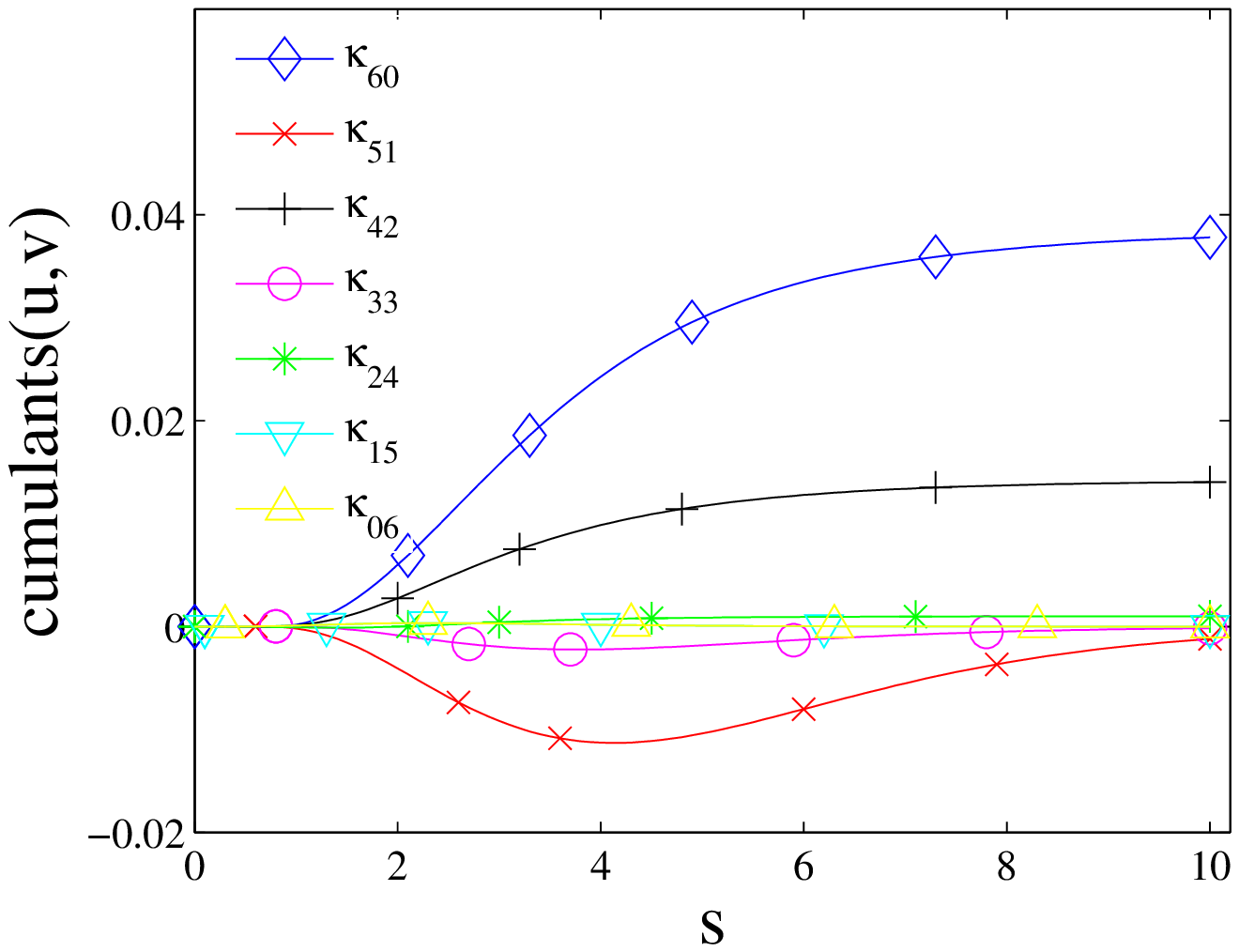}}
\subfigure[ kurtosis excess ]{ \label{fig:ex7_3f}
\includegraphics[scale=0.24]{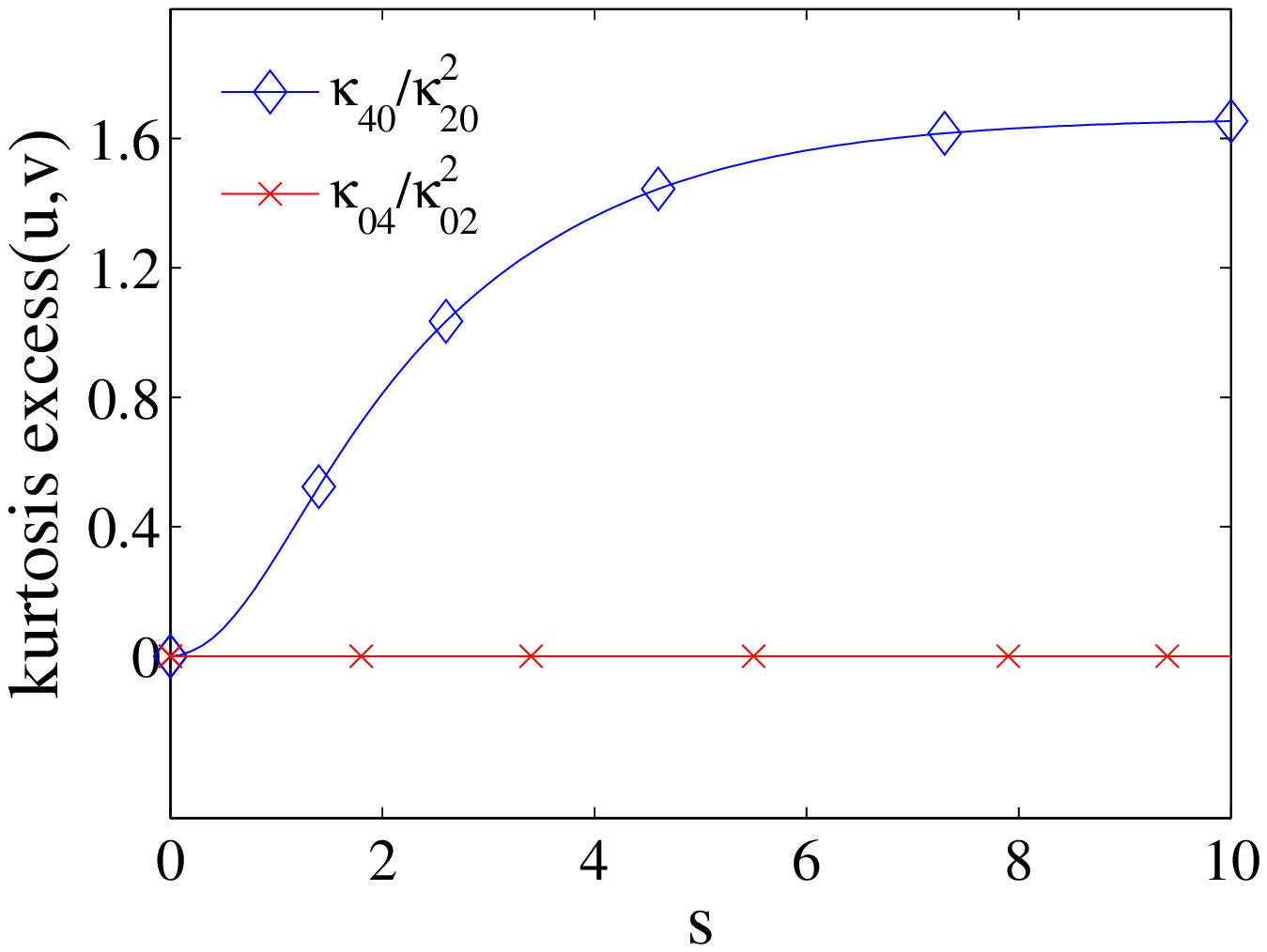}}
\vspace{-0.2cm}
\caption{Bivariate cumulants \rm ($a_v=0$).}
\label{fig:ex7_3}
\end{figure} 
Finally, in Figure \ref{fig:ex7_3}, we present the evolution of the first few cumulants obtained by DgPC, which shows that the system converges to a steady state distribution as time increases. Bivariate cumulants are ordered so that first and second subscripts correspond to $u$ and $v$, respectively. Figure \ref{fig:ex7_3f} shows  kurtosis excess for $u,v$ and clearly indicates that the dynamics of $v$ stay Gaussian, whereas those of $u$  converge to a non-Gaussian state.   
\end{exmp}

 \begin{exmp} \rm \,
  As a final example, using the same set of parameters of the previous example, we introduce a small perturbation to the second equation and set $a_v=0.03$ in \eqref{eq:system}. Another choice of perturbation as $a_v=0.01$ was considered before in~\cite{BM13}. Comparing Figurs \ref{fig:ex7_3} and \ref{fig:ex8}, we see that evolutions of higher order cumulants are perturbed in most cases as expected. In this case, it took a longer transient time to converge to an invariant measure as additional nonlinearity was introduced into the system.

\begin{figure}[!htb]
\centering
\subfigure[ $1^{st}\& 2^{nd}$ order]{
\includegraphics[scale=0.25]{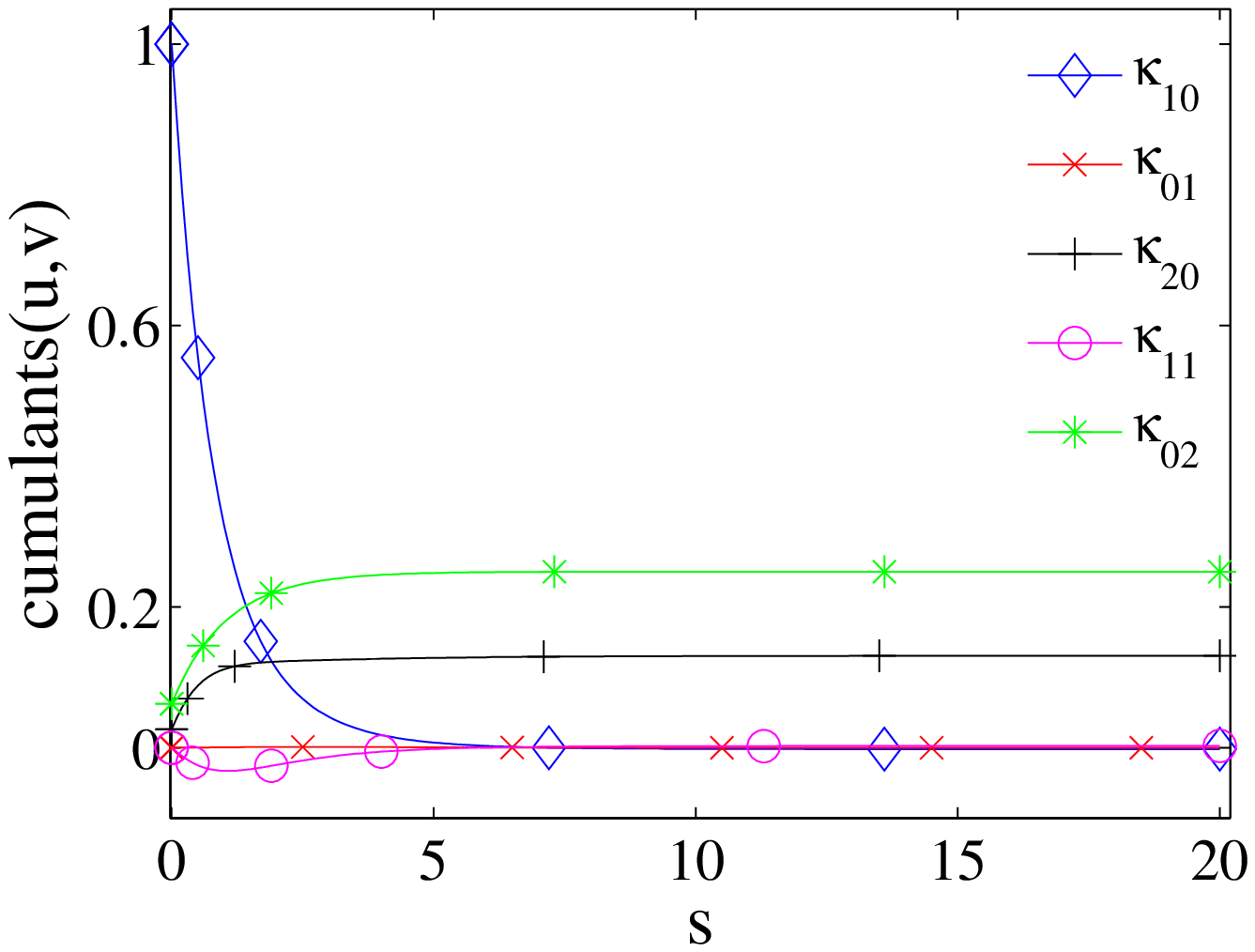}}
\subfigure[ $3^{rd}$ order ]{
\includegraphics[scale=0.25]{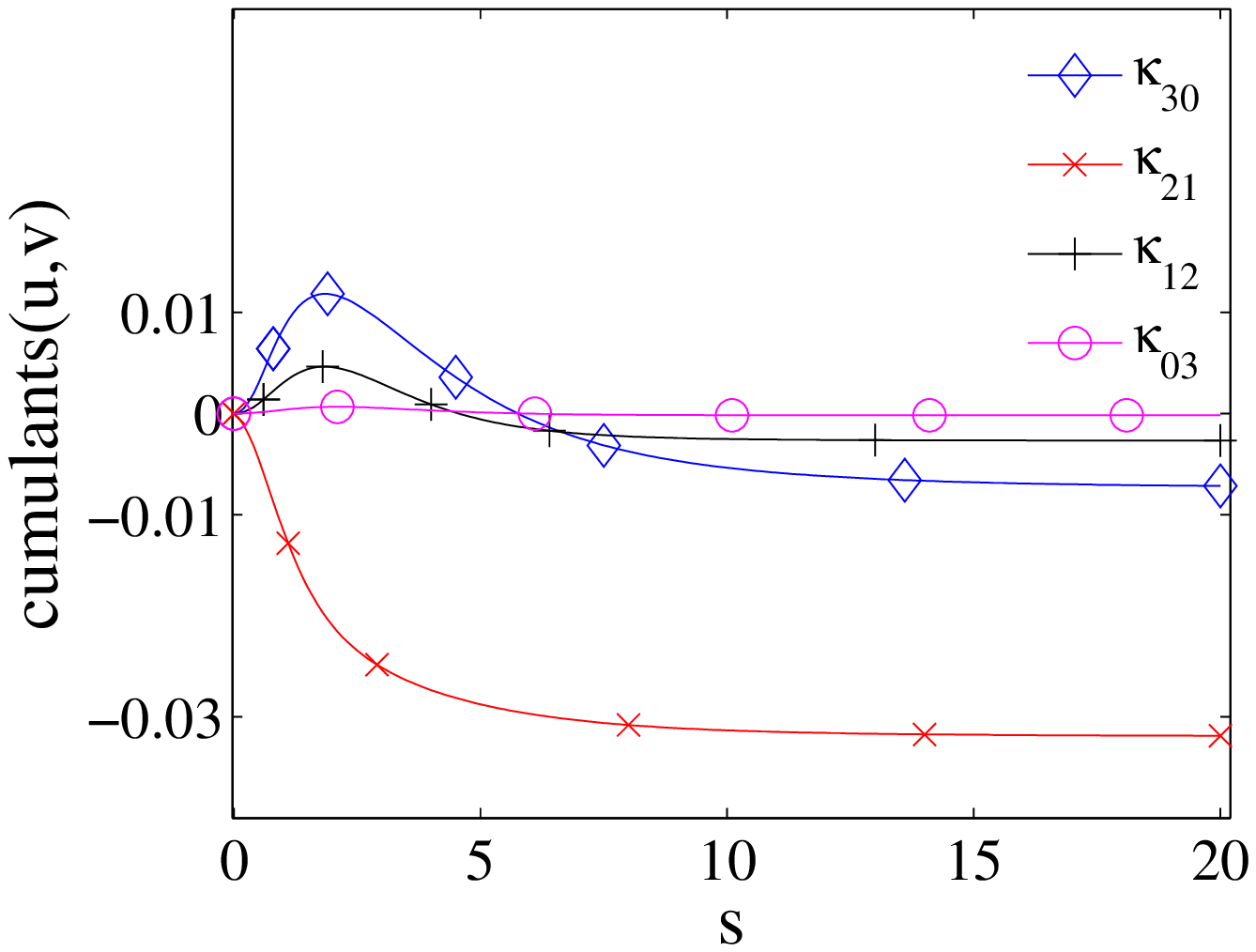}} 
\subfigure[ $4^{th}$ order]{
\includegraphics[scale=0.25]{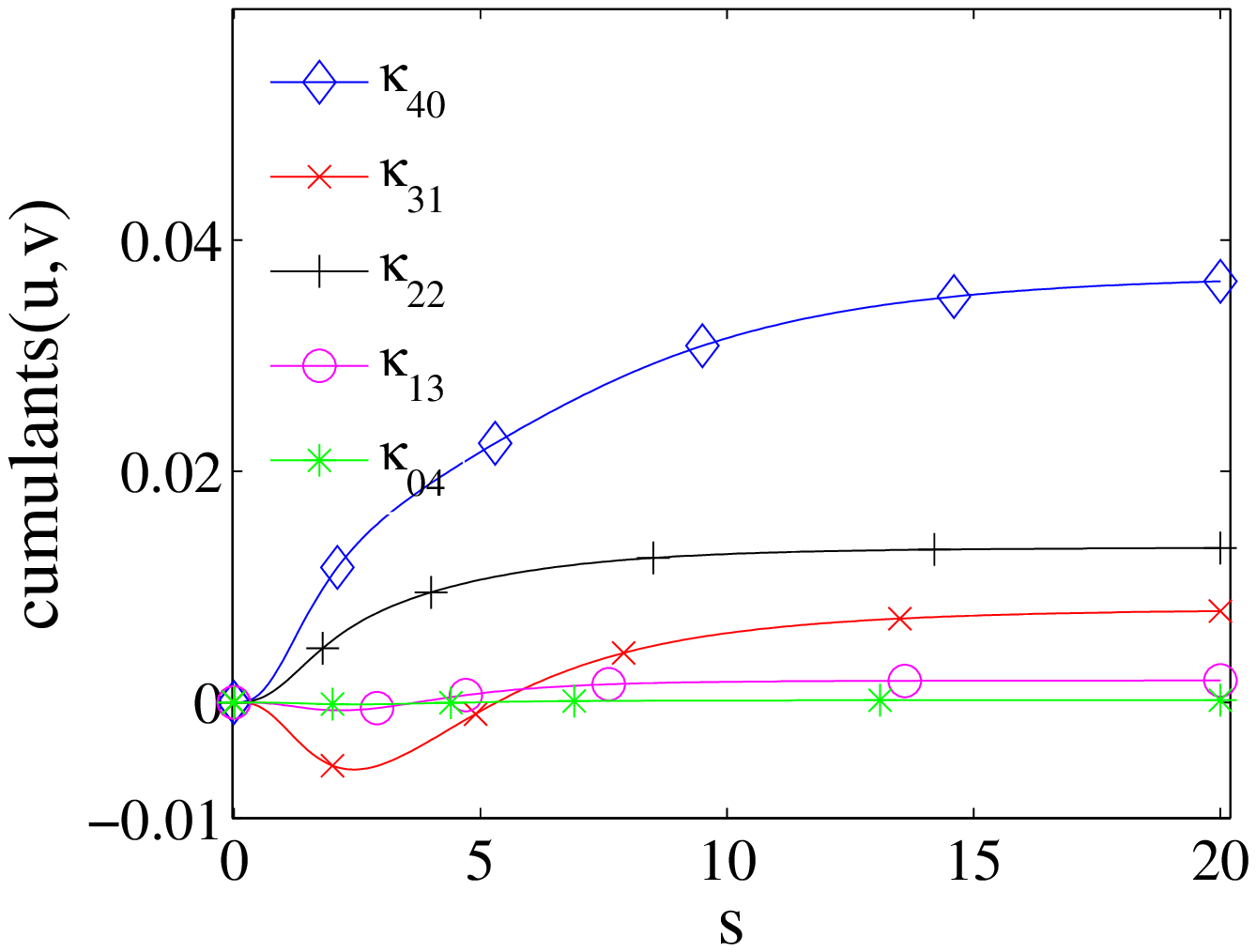}}\\
\vspace{-0.3cm}
\subfigure[ $5^{th}$ order]{
\includegraphics[scale=0.25]{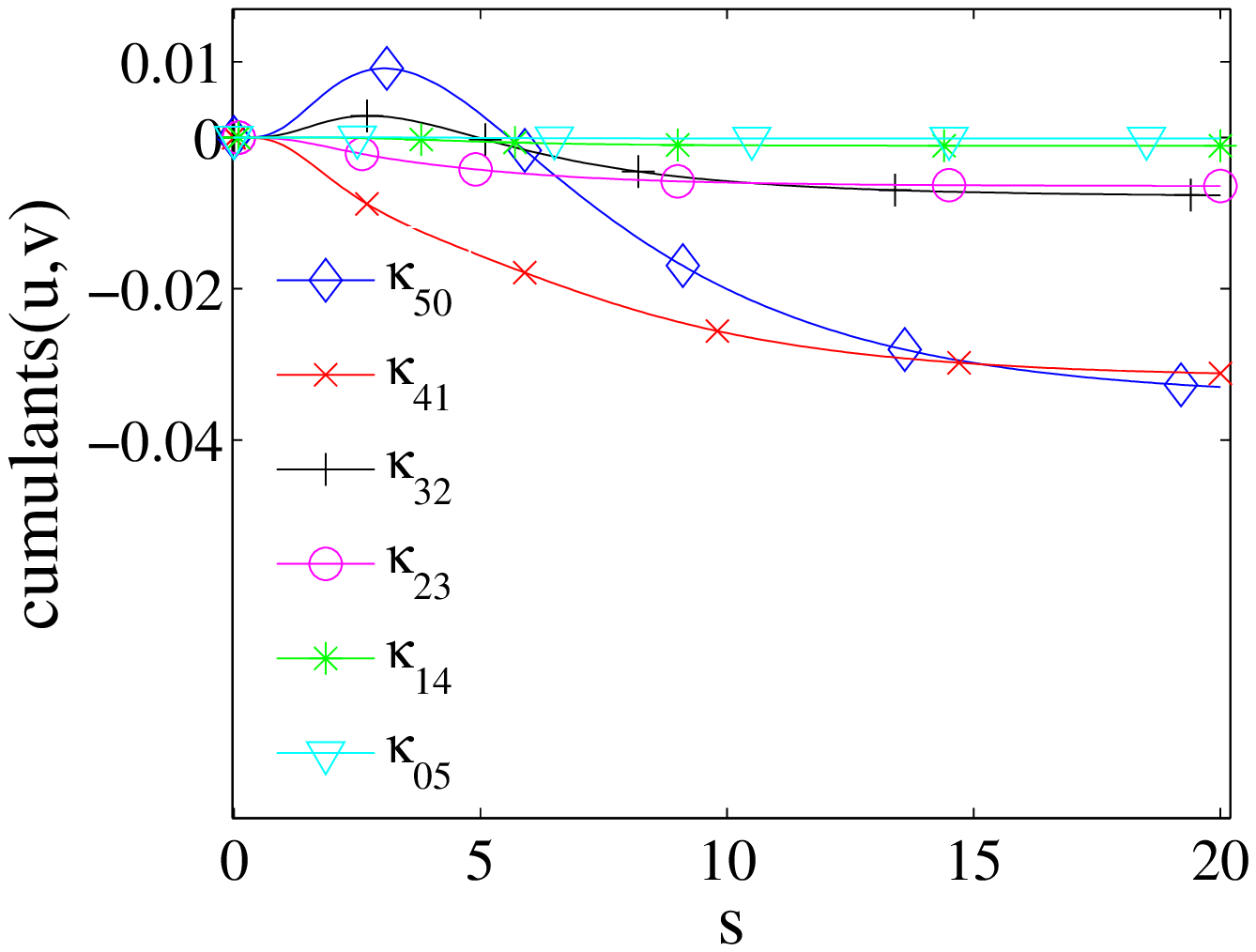}}
\subfigure[cumulants of $v$]{
\includegraphics[scale=0.25]{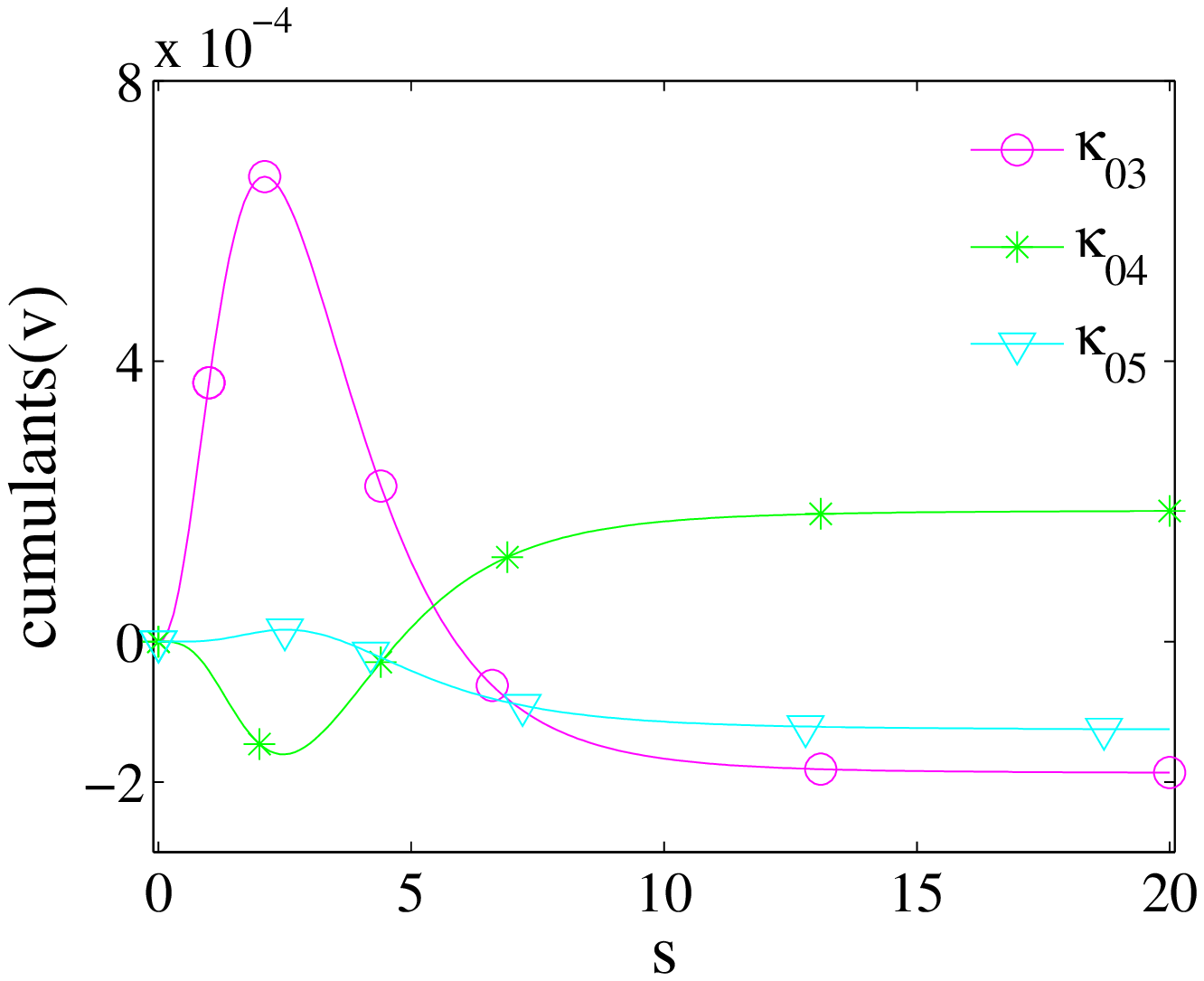}}
\subfigure[kurtosis excess ]{
\includegraphics[scale=0.25]{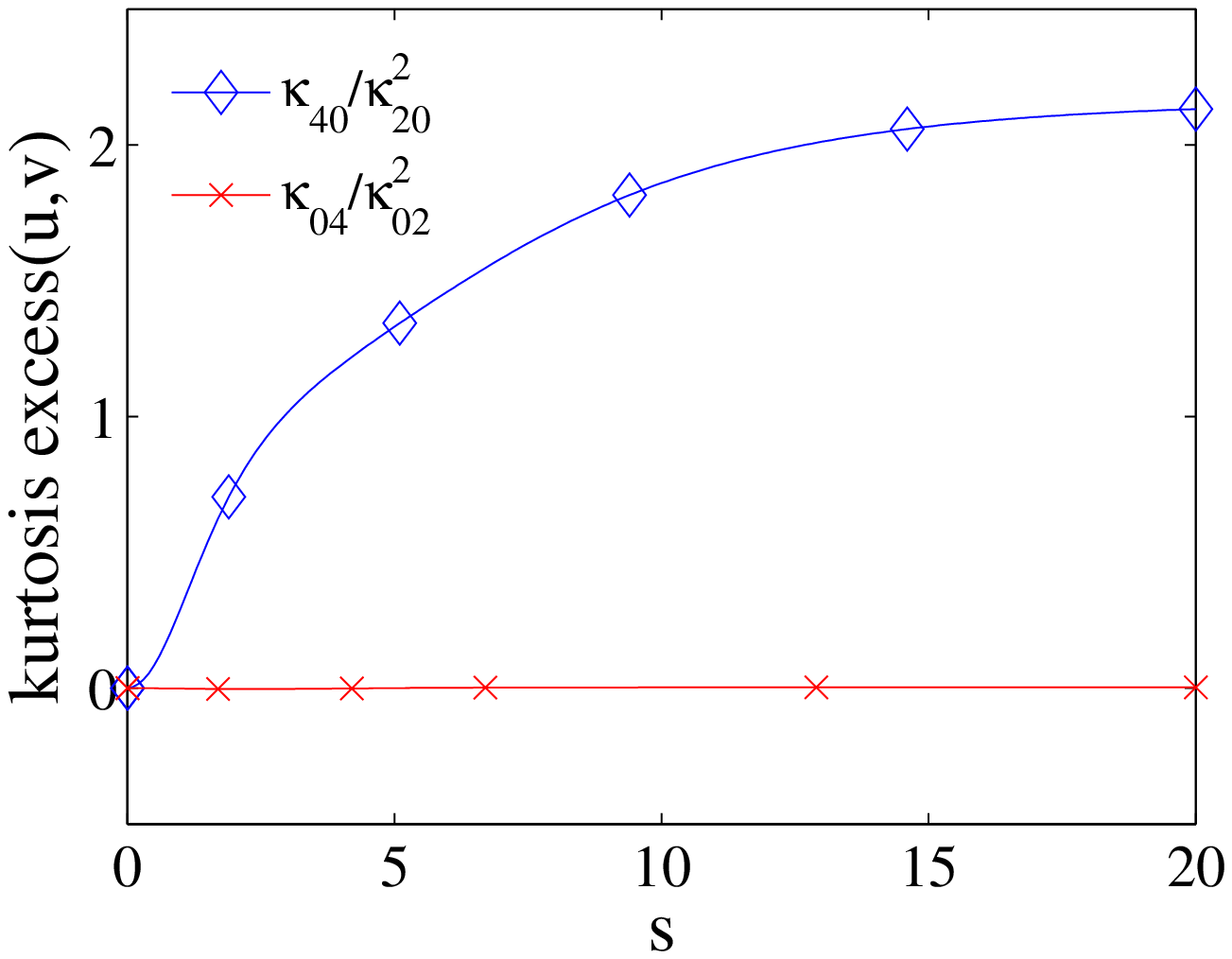}}
\vspace{-0.2cm}
\caption{Bivariate cumulants \rm ($a_v=0.03$).}
\label{fig:ex8}
\end{figure} 

\end{exmp}

\section{Conclusions}
\label{sec:conclu}

Polynomial chaos expansions (PCE) applied to the simulation of evolution equations suffer from two drawbacks: in the presence of stochastic forcing, the dimension of randomness is too large, and the long-time solution may not be sparsely represented in a fixed PCE basis \cite{BM13,GSVK10}. In the setting of Markovian random forcing, we have proposed a restart method that addresses the two aforementioned drawbacks. Such restarts of the PCE, which we called here Dynamical generalized Polynomial Chaos (DgPC), allow us to both to keep the number of random variables small and to obtain a solution that remains reasonably sparse in the evolving basis of orthogonal polynomials, as was done earlier in \cite{GSVK10}.

Following \cite{BM13}, we applied DgPC to the long-time numerical simulation of various stochastic differential equations (SDEs). To demonstrate the ability of the algorithm to reach long-time solutions, we computed invariant measures for SDEs that admit one and found a very good agreement between DgPC and other standard methods such as Monte Carlo-based simulations. We also presented a simple theoretical justification for such a convergence. 

The main computational difficulty in DgPC, as in most gPC-based methods, is the estimation of the orthogonal polynomials of evolving arbitrary measures. Our method, as in \cite{GSVK10}, is based on estimating moments of the multivariate distribution of interest and estimating orthogonal polynomials by a Gram--Schmidt procedure. The bottleneck in such estimations is the cost of computing moments. This is similar to the cost of solving differential equations by PCE methods, where moment estimations are also the most costly step. However, in DgPC, such estimations cannot be performed offline.

From a theoretical point of view, we need to ensure that the evolving measures remain {\em determinate} so that their orthogonal polynomials span square integrable functionals.  Since distributions with compact support are determinate, our theoretical results have been applied in the setting where SDE solutions are well approximated by compactly supported distributions. In this connection, we note that the calculation of orthogonal polynomials from knowledge of moments is an ill-posed problem, which is a serious impediment to gPC methods in general.

Extension of the method to larger systems than those considered here, including to systems obtained by solving stochastic partial differential equations (SPDEs) is challenging. The method works reasonably well for low dimensional SDEs and it needs modifications for larger systems. However, the general idea remains applicable. To keep the number of uncertain variables in check, the method has to be coupled with a compression technique such as Karhunen--Loeve projection at each restart time. Moreover, sparse truncation techniques further help to reduce costs. The analysis of extensions of DgPC to larger systems, faster methods to compute orthogonal polynomials, and adaptive restart procedures is the subject of ongoing research.

\section*{Acknowledgement}
This work was partially funded by AFOSR Grant NSSEFF- FA9550-10-1-0194 and NSF Grant DMS-1408867.

%

\bibliographystyle{plain}

\end{document}